\documentclass[a4paper,10pt]{article}%需要使用pdfLaTex进行编译

\usepackage{latexsym}
\usepackage{mathrsfs}
\usepackage{amsfonts,amsmath,amssymb}
\usepackage{indentfirst}
\usepackage{subeqnarray}
\usepackage{verbatim}
\usepackage{graphicx}
\usepackage{epstopdf}
\usepackage{stmaryrd}
\usepackage{multirow}
\usepackage{diagbox}
\usepackage{subcaption}
\usepackage{float}
\usepackage{booktabs}
\usepackage{xcolor}
\usepackage{extarrows}%新插的包
\usepackage{bm}

\newcommand{\bx}{\mbox{\boldmath{$x$}}}

\newcommand{\bz}{\mbox{\boldmath{$z$}}}

\newcommand{\bb}{\mbox{\boldmath{$b$}}}

\newcommand{\bn}{\mbox{\boldmath{$n$}}}

\newcommand{\bw}{\mbox{\boldmath{$w$}}}
\newcommand{\bW}{\mbox{\boldmath{$W$}}}

\newcommand{\by}{\mbox{\boldmath{$y$}}}

\newcommand{\R}{\mathbb{R}}

\newcommand{\cG}{\mathcal{G}}
\newcommand{\cL}{\mathcal{L}}
\newcommand{\cN}{\mathcal{N}}

\newcommand{\balpha}{\bm{\alpha}}

\newcommand{\bxi}{\bm{\xi}}
\newcommand{\dbx}{\mathrm{d}\bx}

\newcommand{\Real}{\mathbb{R}}
\newcommand{\set}[1]{\left\{#1\right\}}

\newtheorem{theorem}{Theorem}[section]
\newtheorem{lemma}[theorem]{Lemma}
\newtheorem{assumption}[theorem]{Assumption}
\newtheorem{corollary}[theorem]{Corollary}
\newtheorem{definition}[theorem]{Definition}
\newtheorem{example}[theorem]{Example}

\newtheorem{remark}[theorem]{Remark}
\numberwithin{equation}{section}

\newenvironment{proof}[1][Proof]{\textbf{#1.} }
{\ \rule{0.75em}{0.75em}\smallskip}

\usepackage[pdftex,unicode,colorlinks,linkcolor=blue]{hyperref}

\begin{document}
	
	\begin{center}
		\Large\bf Domain-Decomposed Randomized Neural Networks for \\ Partial Differential Equations in Unbounded Domains
	\end{center}
	
	\begin{center}
		Haixin Wang\footnote{School of Mathematics and Statistics, Xi'an Jiaotong University, Xi'an, Shaanxi 710049, P.R. China. Email: {\tt wanghaixin@stu.xjtu.edu.cn}.},
		\quad 
		Haoning Dang\footnote{School of Mathematics and Statistics, Xi'an Jiaotong University, Xi'an, Shaanxi 710049, P.R. China. Email: {\tt haoningdang.xjtu@stu.xjtu.edu.cn}.},
		\quad 
		Fei Wang\footnote{School of Mathematics and Statistics \& State Key Laboratory of Multiphase Flow in Power Engineering, Xi'an Jiaotong University, Xi'an, Shaanxi 710049, China. The work of this author was partially supported by the National Natural Science Foundation of China (Grant No.\ 92470115). Email: {\tt feiwang.xjtu@xjtu.edu.cn}}
		\quad {\rm and}
		\quad
		Shimin Guo\footnote{School of Mathematics and Statistics, Xi'an Jiaotong University, Xi'an, Shaanxi 710049, P.R. China. The work of this author was partially supported by the National Natural Science Foundation of China (Grant No.\ 12271427). Email: {\tt shiminguo@mail.xjtu.edu.cn}.}
	\end{center}

	\medskip

\begin{quote}
{\bf Abstract.}
Partial differential equations on unbounded domains are challenging because the exterior region must be represented without excessive truncation error. Truncation-based methods often require problem-dependent artificial boundary conditions, while global spectral bases may be inefficient for localized structures, irregular geometries, or solutions with different near-field and far-field behaviors. We propose a domain-decomposed randomized neural network framework for such problems. Different randomized subnetworks are assigned to different spatial regimes: a near-field subnetwork captures local and geometric features, whereas a far-field subnetwork represents exterior decay. The subnetworks are coupled by boundary and interface conditions, and only the output-layer coefficients are solved from linear least-squares systems arising from Petrov--Galerkin or collocation formulations. We develop a Petrov--Galerkin method for semi-unbounded elliptic problems and a collocation method for fully unbounded, perforated, and time-dependent problems. A conditional bounded-parameter approximation result is proved in a broken Sobolev norm, together with an error decomposition covering approximation, empirical-consistency/quadrature, and least-squares optimization errors. Numerical experiments for Poisson and time-dependent Schr\"odinger equations demonstrate the accuracy and flexibility of the proposed method.

\end{quote}


	\section{Introduction}

	Partial differential equations posed on unbounded domains arise naturally in wave propagation, quantum mechanics, diffusion, potential theory, and exterior-domain elliptic problems. Their numerical approximation is difficult because the infinite spatial extent cannot be represented by a finite mesh without additional modeling or approximation. As a result, traditional finite element, finite difference, and finite volume methods usually require domain truncation, coordinate mappings, absorbing layers, or other exterior-domain treatments before they can be implemented in practice.

The difficulty is not only geometric but also analytic. The numerical method must preserve the influence of the exterior region, represent the far-field behavior of the solution, and remain accurate near physical boundaries, obstacles, or localized structures. For example, Schr\"odinger equations on unbounded domains involve complex-valued solutions evolving in both space and time, while exterior elliptic problems may exhibit algebraic or exponential decay. These different behaviors motivate numerical methods that can treat near-field and far-field regimes in a flexible but computationally efficient way.

	Existing numerical methods for PDEs on unbounded domains can be broadly classified into two main categories. The first category consists of truncation-based approaches, including artificial boundary conditions, absorbing layers, and finite element, finite difference, or finite volume discretizations on bounded computational domains; representative works include \cite{Carolin2022Magnetostatic,Chen2022TwoGrid,Chen2023Analysis,Ding2023Finite,Silvia2022TwoFEMBEM,guo2023new,Guo2021Finite,kuhn2024explicit,li2019numerical,lin2025finite,Singh2023Rate,Tai2022Numerical,Xie2023fastBDF2}. These methods are flexible and compatible with standard discretizations, but their performance often depends on the quality of the exterior-domain treatment and on the choice of truncation or absorption parameters. The second category uses global or mapped basis functions on unbounded domains, with spectral methods based on Laguerre, Hermite, or related functions being representative examples \cite{Guo2022Dissipation,Ling2023Analysis,Dina2023Tanh,Shen2009spctral,tissaoui2025efficient,Zhang2021Spectral,zhu2024highly}. Such methods can achieve high accuracy for smooth solutions, but they may be less flexible for localized structures, complicated geometries, discontinuities, or solutions with different near-field and far-field behaviors.

These observations suggest that an effective method for unbounded-domain problems should combine far-field representation, geometric flexibility, and computational tractability. The central difficulty is not merely how to truncate the domain, but how to represent the exterior region in a way that remains coupled to the near-field solution. This motivates the use of separate trial spaces for different spatial regimes: a flexible near-field representation for local structures and a far-field representation for decay or asymptotic behavior. The proposed method follows this principle by assigning different randomized neural subnetworks to different subdomains and coupling them through interface equations.

	In recent years, neural-network-based methods have provided an alternative perspective for the numerical solution of partial differential equations. Instead of relying exclusively on mesh-based trial spaces, these methods use neural networks as flexible function approximators and determine the network parameters through collocation, residual minimization, variational principles, or physics-informed loss functions. Representative developments include early neural-network solvers for differential equations (\cite{Dissanayake1994NNPDE,Lagaris1998ANN}), physics-informed neural networks (\cite{Raissi2019PINN,Raissi2017PINNPartI,Raissi2017PINNPartII}), deep Ritz and deep Galerkin type methods (\cite{E2018deepRitz,Sirignano2018DGM}), and related approaches for complex geometries, nonlinear PDEs, variational formulations, and residual minimization (\cite{Jens2018unified,Tang2023Physics,JonathanW2023Greedy}). These works demonstrate the flexibility of neural-network trial spaces, while also motivating further study of accuracy, stability, conditioning, and efficient training for PDE problems on challenging domains.

Although neural-network-based methods are flexible, they usually require nonlinear optimization over a large number of trainable parameters, and the resulting training process may be sensitive to initialization, sampling strategies, and hyperparameter choices. Randomized neural networks (RaNNs) provide an efficient alternative by randomly assigning and then freezing the hidden-layer parameters, so that only the output-layer coefficients are determined from the numerical formulation. Consequently, the computation is often reduced to a linear least-squares problem rather than a fully nonlinear training problem. This idea has been shown to be effective for a variety of PDE problems; see, for example, \cite{dong2021LocalELM,shang2023randomized,sun2022local,chen2022bridging,shang2024randomized,dang2024local,li2025local,zhang2024transferable,shang2025overlapping,sun2026randomized,liu2025integral} and the references therein.
More recently, Dang, Wang, and Jiang introduced adaptive-growth RaNNs for PDE approximation (\cite{dang2024adaptive}), in which the randomized trial space is enriched progressively, while Yang and Wang proposed adaptive-distribution RaNNs (\cite{yang2026adaptive}), where low-dimensional sampling distributions for random parameters are learned to reduce the sensitivity to parameter selection.

 These developments demonstrate the efficiency of RaNN-based discretizations, but they do not by themselves resolve the additional difficulties caused by unbounded domains.
  A single randomized neural trial space with one parameter scale may not approximate both near-field structures and far-field decay accurately. These issues motivate a domain-decomposed RaNN framework in which different subnetworks are assigned to different spatial regimes and coupled by physically motivated interface conditions. In this way, the exterior region is represented by a far-field randomized trial space rather than being removed from the numerical formulation.

The main contributions of this paper are summarized as follows.
\begin{itemize}

    \item We propose a domain-decomposed RaNN framework for PDEs on unbounded domains. The method uses different randomized subnetworks for different parts of the domain, for example a near-field network and a far-field network, so that local structures and far-field decay can be represented by different random feature spaces.

    \item We develop two complementary discretizations. The first one is a domain-decomposed RaNN Petrov--Galerkin method (DD-RaNN-PG) for semi-unbounded elliptic problems, where the test space is constructed from piecewise polynomial functions and the unbounded direction is treated by Gauss-type quadrature. The second one is a domain-decomposed RaNN collocation method (DD-RaNN-CM) for fully unbounded or geometrically complicated domains, including perforated domains and time-dependent Schr\"odinger equations. Both formulations lead to linear least-squares systems for the output-layer coefficients.

    \item We provide a theoretical justification of the proposed framework. In particular, we establish a conditional bounded-parameter broken Sobolev approximation result for domain-decomposed randomized neural trial spaces on unbounded domains. We also decompose the total error into approximation, empirical-consistency/quadrature, and least-squares optimization components under suitable stability and consistency assumptions.
    
\end{itemize}

The rest of this paper is organized as follows.
Section~\ref{sec:framework} introduces the general domain-decomposed RaNN framework, including randomized neural trial spaces, near--far decomposition, and the generic least-squares formulation.
Section~\ref{sec:half} develops the domain-decomposed RaNN Petrov--Galerkin method for semi-unbounded elliptic problems.
Section~\ref{sec:whole} presents the domain-decomposed RaNN collocation method for fully unbounded and time-dependent problems, using a linear Schr\"odinger equation as a representative example.
Section~\ref{sec:analysis} provides the approximation and error analysis.
Section~\ref{sec:num} reports numerical experiments for semi-unbounded, fully unbounded, and time-dependent problems.
Finally, Section~\ref{sec:summary} concludes the paper.

%%%%%%%%%%%%%%%%%%%%%%%%%%%%%%%

\section{General Domain-Decomposed RaNN Framework}
\label{sec:framework}

In this section, we introduce the general domain-decomposed randomized neural network (DD-RaNN) framework used in this paper. The main idea is to represent different spatial regimes of an unbounded-domain solution by different randomized neural subnetworks and to determine only their output-layer coefficients from linear least-squares systems. The concrete Petrov--Galerkin formulation for semi-unbounded elliptic problems and the domain-decomposed collocation formulation for fully unbounded and time-dependent problems will be developed in Sections~\ref{sec:half} and~\ref{sec:whole}, respectively.

\subsection{Randomized Neural Trial Spaces}
\label{subsec:rann-trial-spaces}

Let $\Omega\subset\mathbb{R}^d$ be an unbounded domain, and let the input variable be $\bx\in\Omega$. A fully connected tensor neural network with $l$ hidden layers is written as
	\begin{align*}
		& \Phi^{(1)}(\bx) = \rho(\bW^{(1)}\bx,\bb^{(1)}),\\
		&\Phi^{(j)}(\bx) = \rho(\bW^{(j)}\Phi^{(j-1)}(\bx),\bb^{(j)}),
		\qquad j=2,\ldots,l,\\
		&U(\bx) = \bW^{(l+1)}\Phi^{(l)}(\bx),\quad \rho(\bx,\bb)=\prod_{j=1}^{d}\rho_j(x_j+b_{j}),
	\end{align*}
where $\Phi^{(j)}$ denotes the output vector of the $j$-th hidden layer,  $\bW^{(j)}$ and $\bb^{(j)}$ are the corresponding weight matrix and bias vector, and the tensor-product activation $\rho$ is applied separately to each neuron. The output layer is taken as a bias-free linear map.

In an RaNN, the hidden-layer parameters $\{\bW^{(j)},\bb^{(j)}\}_{j=1}^{l}$ are randomly generated and then fixed. Once these parameters are determined, the components of the last hidden layer define a finite-dimensional randomized trial space $\mathcal{U}(\bW,\bb,\rho):=\mathrm{span}\{\Phi^{(l)}_1,\ldots,\Phi^{(l)}_m\}$, where $m$ is the width of the last hidden layer. An RaNN approximation has the form
\[
u_h(\bx)
=
\sum_{k=1}^{m}\hat u_k\Phi_k^{(l)}(\bx).
\]
Thus the unknowns are only the output-layer coefficients $\{\hat u_k\}_{k=1}^{m}$.

In this paper, two randomization strategies are used. The first one samples both weights and biases from a uniform distribution,
\[
\bW,\bb\sim U(-r,r),
\]
where $r>0$ is a prescribed scale parameter. The second one samples $\bW\sim U(-r,r)$ and constructs the bias vector by
\[
\bb=(-\bW\odot \mathbf R)\mathbf 1,
\qquad
\mathbf R\sim U(-1,1),
\]
where $\odot$ denotes componentwise multiplication and $\mathbf 1$ is the all-ones vector of compatible size. This second choice controls the location of feature centers and is useful for reducing far-field quadrature errors on unbounded domains.

The essential computational feature of RaNNs is that the output-layer coefficients are not obtained by nonlinear training of all network parameters. Instead, they are determined by a global numerical formulation, such as a Petrov--Galerkin or collocation least-squares system. This leads to a linear algebraic problem for the output-layer coefficients.

\subsection{Near--Far Domain Decomposition}
\label{subsec:near-far-decomposition}

For PDEs on unbounded domains, a single randomized feature space with one parameter scale may be inefficient. The solution may have localized structures near physical boundaries or obstacles, while exhibiting decay or different asymptotic behavior in the far field. To represent these different regimes, we decompose the domain into non-overlapping subregions in the sense that
\[
\Omega=\bigcup_{\alpha=1}^{N_{\rm net}}\overline{\Omega_\alpha},
\qquad
\Omega_\alpha\cap\Omega_\beta=\emptyset\quad(\alpha\ne\beta)
\]
up to their common interfaces, and assign an independent randomized subnetwork to each subregion.
The numerical solution is written as
\[
u_h(\bx)
=
\sum_{\alpha=1}^{N_{\rm net}}
\chi_\alpha(\bx)u_h^{(\alpha)}(\bx),
\qquad
u_h^{(\alpha)}(\bx)
=
\sum_{k=1}^{m_\alpha}
\hat u_k^{(\alpha)}\phi_k^{(\alpha)}(\bx),
\]
where $\chi_\alpha$ is the indicator function of $\Omega_\alpha$, and
$\{\phi_k^{(\alpha)}\}_{k=1}^{m_\alpha}$ are the randomized features generated by the $\alpha$-th subnetwork.

For a two-network near--far decomposition, one typically writes
\[
\Omega=\Omega_1\cup\Omega_2,
\]
where $\Omega_1$ is a bounded near-field region and $\Omega_2$ is an exterior unbounded region. The approximation becomes
\[
u_h(\bx)
=
\begin{cases}
\displaystyle
\sum_{k=1}^{m_1}\hat u_k^{(1)}\phi_k^{(1)}(\bx),
& \bx\in\Omega_1,\\[1ex]
\displaystyle
\sum_{k=1}^{m_2}\hat u_k^{(2)}\phi_k^{(2)}(\bx),
& \bx\in\Omega_2.
\end{cases}
\]
For one-dimensional whole-space problems, one may instead use a three-network decomposition consisting of the left far-field region, the central near-field region, and the right far-field region.

The subregions are coupled by physical boundary equations and interface equations. For an interface
\[
\Gamma_{12}:=\partial\Omega_1\cap\partial\Omega_2,
\]
typical coupling conditions are
\[
u_h^{(1)}=u_h^{(2)},
\qquad
\beta_1\nabla u_h^{(1)}\cdot\bn_{12}
+
\beta_2\nabla u_h^{(2)}\cdot\bn_{21}
=0
\quad\text{on }\Gamma_{12},
\]
where $\beta_1$ and $\beta_2$ denote the traces of $\beta$ from the two sides and $\bn_{12}$ is the unit normal pointing from $\Omega_1$ to $\Omega_2$. When $\beta$ is continuous and a single normal direction is fixed, this condition reduces to equality of the normal fluxes.
These equations should be interpreted as coupling equations between two trial spaces. The artificial interface is therefore a domain-decomposition interface, not a truncation boundary where the exterior domain is discarded.

\subsection{Generic Least-Squares Form}
\label{subsec:generic-ls-formulations}

For convenience, consider the model elliptic equation
\begin{equation}
-\nabla\cdot(\beta\nabla u)+\gamma u=f,
\qquad
\beta,\gamma:\Omega\rightarrow\mathbb{R},
\qquad
\bx\in\Omega\subset\mathbb{R}^d.
\label{eq:GF}
\end{equation}
When $\Omega$ is unbounded, the solution is assumed to satisfy a prescribed far-field condition, for example
\[
u(\bx)\rightarrow0
\qquad \text{as } \|\bx\|_2\rightarrow\infty.
\]

If the domain has physical boundaries, boundary conditions may also be imposed. For example, one may prescribe
\begin{equation}
\begin{cases}
u(\bx)=g_D(\bx), & \bx\in\Gamma_D,\\
\beta\nabla u(\bx)\cdot\bn(\bx)=g_N(\bx), & \bx\in\Gamma_N,
\end{cases}
\label{bd:1}
\end{equation}
where $\partial\Omega=\Gamma_D\cup\Gamma_N$, $\Gamma_D\cap\Gamma_N=\emptyset$, and $\bn$ is the unit outward normal vector.

Depending on the discretization, the PDE equations may be obtained either from a weak formulation or from strong-form collocation. In a Petrov--Galerkin formulation, the rows of $\boldsymbol A$ are generated by testing the PDE against prescribed test functions. In a collocation formulation, the rows of $\boldsymbol A$ are generated by enforcing the strong residual at sampling points. In both cases, once the hidden-layer parameters are fixed, the PDE equations are linear with respect to the output-layer coefficients.

After the PDE equations, physical boundary or initial conditions, and interface coupling conditions are discretized, they are assembled into the generic weighted least-squares system
\begin{equation}
\begin{bmatrix}
\boldsymbol A\\
\xi\boldsymbol B\\
\eta\boldsymbol C
\end{bmatrix}
\boldsymbol U
=
\begin{bmatrix}
\boldsymbol L\\
\xi\boldsymbol F\\
\eta\boldsymbol Z
\end{bmatrix}.
\label{eq:generic-stacked-ls}
\end{equation}
Here $\boldsymbol A$ represents the PDE equations, $\boldsymbol B$ represents physical boundary or initial conditions, $\boldsymbol C$ represents interface coupling conditions, and $\xi,\eta>0$ are weighting parameters. Thus the proposed framework reduces the determination of all output-layer coefficients to a linear least-squares problem. The concrete construction of $\boldsymbol A$, $\boldsymbol B$, and $\boldsymbol C$ will be given in Sections~\ref{sec:half} and~\ref{sec:whole}.

\section{Domain-Decomposed RaNN Petrov--Galerkin Method for\\ Semi-Unbounded Elliptic Problems}
\label{sec:half}

In this section, we specialize the general domain-decomposed RaNN framework to semi-unbounded elliptic problems and derive a broken Petrov--Galerkin least-squares formulation, which we call the domain-decomposed RaNN Petrov--Galerkin method (DD-RaNN-PG). We consider
\begin{equation}
\begin{cases}
-\nabla\cdot(\beta(\bx)\nabla u(\bx))+\gamma(\bx)u(\bx)=f(\bx),
& \bx\in\Omega=\mathbb{R}_+^d,\\
\displaystyle \lim_{\|\bx\|_2\rightarrow+\infty}u(\bx)=0.
\end{cases}
\label{eq:half}
\end{equation}
The physical boundary conditions are
\begin{equation}
u(\bx)=g_D(\bx)\quad\text{on }\Gamma_D,
\qquad
\beta(\bx)\nabla u(\bx)\cdot\bn(\bx)=g_N(\bx)
\quad\text{on }\Gamma_N,
\qquad
\partial\Omega=\Gamma_D\cup\Gamma_N,
\quad
\Gamma_D\cap\Gamma_N=\emptyset .
\label{bd2}
\end{equation}

A key difficulty in semi-unbounded problems is that the solution may exhibit different behavior near the physical boundary and in the far field. A single randomized feature distribution is often inefficient for both regimes. Therefore, we decompose the semi-infinite domain into a bounded near-field region and an exterior unbounded region, and use two randomized subnetworks with different parameter ranges.

\subsection{Near--Far Decomposition and Quadrature Grid}
\label{subsec:pg-decomposition-testspace}

Let $x_{\mathrm{cut}}>0$ be a prescribed near--far separation parameter. In this Petrov--Galerkin construction, we use a tensor-product box-type decomposition of the half-space,
\[
\Omega_1=[0,x_{\mathrm{cut}}]^d,
\qquad
\Omega_2=\Omega\setminus\Omega_1,
\qquad
\Omega=\mathbb R_+^d .
\]
For $d=1$, this reduces to the standard decomposition of the half-line into a bounded interval and an exterior interval. For $d\ge2$, the interface is the piecewise smooth box boundary
\[
\Gamma_I:=\partial\Omega_1\cap\partial \Omega_2.
\]
This tensor-product construction is convenient for the Petrov--Galerkin formulation and the associated quadrature. More general interface geometries, such as circular or spherical near--far interfaces, are treated below by the collocation formulation.

We construct the quadrature and test-function grid dimension by dimension. For the near-field interval $[0,x_{\mathrm{cut}}]$, we use a uniform partition
\[
0=x_0<x_1<\cdots<x_{N_1}=x_{\mathrm{cut}},
\qquad
h=\frac{x_{\mathrm{cut}}}{N_1}.
\]
For the exterior interval $(x_{\mathrm{cut}},+\infty)$, we use shifted positive Hermite--Gauss nodes. 
Let
\[
-y_{N_2}<\cdots<-y_1<y_1<\cdots<y_{N_2}
\]
be the $2N_2$ Hermite--Gauss nodes. We retain the positive nodes and shift them to the exterior region:
\[
x_{N_1+k}=x_{\mathrm{cut}}+S\,y_k,
\qquad
k=1,\ldots,N_2,
\]
where $S>0$ is a scaling parameter. Thus the one-dimensional node sequence is
\[
0=x_0< x_1<\cdots<x_{N_1}=x_{\mathrm{cut}}
<
x_{N_1+1}
<
\cdots
<
x_N,
\qquad
N=N_1+N_2.
\]
The tensor product of this node sequence gives the quadrature structure in $\mathbb{R}_+^d$.

The test space is chosen as the standard continuous piecewise linear finite element space associated with this tensor-product grid:
\[
V_h=\mathrm{span}\{v_j\}_{j=1}^{M}.
\]
Since the physical domain is unbounded, the exterior integrals are evaluated numerically by Gauss-type quadrature.

For the two subnetworks, define
\[
U_h^{(1)}=\mathrm{span}\{\phi_k^{(1)}\}_{k=1}^{m_1},
\qquad
U_h^{(2)}=\mathrm{span}\{\phi_k^{(2)}\}_{k=1}^{m_2}.
\]
The functions $\{\phi_k^{(1)}\}$ are used on $\Omega_1$, and the functions $\{\phi_k^{(2)}\}$ are used on $\Omega_2$. The broken trial function is represented as
\begin{equation}
u_h(\bx)
=
\begin{cases}
\displaystyle
\sum_{k=1}^{m_1}\hat u_k^{(1)}\phi_k^{(1)}(\bx),
& \bx\in\Omega_1,\\[1ex]
\displaystyle
\sum_{k=1}^{m_2}\hat u_k^{(2)}\phi_k^{(2)}(\bx),
& \bx\in\Omega_2.
\end{cases}
\label{eq:pg-two-net-ansatz}
\end{equation}
Before the interface equations are imposed, $u_h$ should be regarded as a broken trial function.

\subsection{Broken Petrov--Galerkin Discretization}
\label{subsec:pg-discretization}

For a sufficiently smooth exact solution, integration by parts gives the weak formulation
\[
a(u,v)=l(v),
\qquad
\forall v\in V,
\]
where, for test functions satisfying $v|_{\Gamma_D}=0$,
\begin{align}
a(u,v)
&=
\int_{\Omega}\beta\nabla u\cdot\nabla v\,\mathrm d\bx
+
\int_{\Omega}\gamma uv\,\mathrm d\bx,
\label{eq:pg-bilinear-form}\\
l(v)
&=
\int_{\Omega}fv\,\mathrm d\bx
+
\int_{\Gamma_N}g_Nv\,\mathrm ds.
\label{eq:pg-linear-form}
\end{align}

Since the numerical trial function \eqref{eq:pg-two-net-ansatz} is constructed piecewise, it is not globally conforming before the interface equations are imposed. We therefore assemble the PDE part by using the broken bilinear form
\[
a_h(u_h,v_h)
=
\sum_{\alpha=1}^{2}
\left[
\int_{\Omega_\alpha}
\beta\nabla u_h^{(\alpha)}\cdot\nabla v_h\,\mathrm d\bx
+
\int_{\Omega_\alpha}
\gamma u_h^{(\alpha)}v_h\,\mathrm d\bx
\right].
\]
For a smooth exact solution satisfying continuity of the solution and the normal flux across $\Gamma_I$, this broken formulation is consistent with the original weak formulation. For the numerical solution, these transmission conditions are imposed separately through the interface equations in Section~\ref{subsec:pg-boundary-interface}. Thus the PDE equations and the interface equations should be viewed as two coupled components of one least-squares system.

The Petrov--Galerkin equations for the PDE part are
\[
a_h(u_h,v_j)=l(v_j),
\qquad
j=1,\ldots,M.
\]
Substituting \eqref{eq:pg-two-net-ansatz} gives
\begin{equation}
\begin{aligned}
\sum_{k=1}^{m_1}\hat u_k^{(1)}
\left[
\int_{\Omega_1}\beta\nabla\phi_k^{(1)}\cdot\nabla v_j\,\mathrm d\bx
+
\int_{\Omega_1}\gamma\phi_k^{(1)}v_j\,\mathrm d\bx
\right]
+
\sum_{k=1}^{m_2}\hat u_k^{(2)}
\left[
\int_{\Omega_2}\beta\nabla\phi_k^{(2)}\cdot\nabla v_j\,\mathrm d\bx
+
\int_{\Omega_2}\gamma\phi_k^{(2)}v_j\,\mathrm d\bx
\right]
=
l(v_j).
\end{aligned}
\label{a1a2}
\end{equation}
Thus the Petrov--Galerkin equations can be written as
\begin{equation}
\boldsymbol A\boldsymbol U=\boldsymbol L,
\qquad
\boldsymbol A=(\boldsymbol A^{(1)},\boldsymbol A^{(2)}),
\qquad
\boldsymbol U^\top=((\boldsymbol U^{(1)})^\top,(\boldsymbol U^{(2)})^\top),
\label{au3}
\end{equation}
where
\[
\boldsymbol A_{jk}^{(\alpha)}=a_h(\phi_k^{(\alpha)},v_j),
\qquad
\boldsymbol L_j=l(v_j),
\qquad
\boldsymbol U^{(\alpha)}
=
(\hat u_1^{(\alpha)},\ldots,\hat u_{m_\alpha}^{(\alpha)})^\top,
\qquad
\alpha=1,2.
\]

\subsection{Boundary and Interface Conditions}
\label{subsec:pg-boundary-interface}

The Dirichlet boundary condition is enforced at boundary sampling points
$\{\by_j\}_{j=1}^{N_b}\subset\Gamma_D$. If $\by_j\in\Gamma_D\cap\overline{\Omega_\alpha}$, we impose
\[
\sum_{k=1}^{m_\alpha}
\hat u_k^{(\alpha)}\phi_k^{(\alpha)}(\by_j)
=
g_D(\by_j).
\]
Collecting these equations gives
\begin{equation}
\boldsymbol B\boldsymbol U=\boldsymbol F.
\label{bd3}
\end{equation}

Next, we impose interface coupling between the near-field and far-field subnetworks. Let
$\{\bz_j\}_{j=1}^{N_c}\subset\Gamma_I$ be interface sampling points. We enforce
\[
u_h^{(1)}(\bz_j)=u_h^{(2)}(\bz_j),
\qquad
\beta_1(\bz_j)\nabla u_h^{(1)}(\bz_j)\cdot\bn
=
\beta_2(\bz_j)\nabla u_h^{(2)}(\bz_j)\cdot\bn,
\qquad
j=1,\ldots,N_c,
\]
where \(\bn\) is a fixed unit normal vector on \(\Gamma_I\), taken to point from \(\Omega_1\) to \(\Omega_2\). Equivalently,
\begin{equation}
\begin{cases}
\displaystyle
\sum_{k=1}^{m_1}\hat u_k^{(1)}\phi_k^{(1)}(\bz_j)
-
\sum_{k=1}^{m_2}\hat u_k^{(2)}\phi_k^{(2)}(\bz_j)
=0,\\[2ex]
\displaystyle
\sum_{k=1}^{m_1}\hat u_k^{(1)}
\beta_1(\bz_j)\nabla\phi_k^{(1)}(\bz_j)\cdot\bn
-
\sum_{k=1}^{m_2}\hat u_k^{(2)}
\beta_2(\bz_j)\nabla\phi_k^{(2)}(\bz_j)\cdot\bn
=0.
\end{cases}
\end{equation}
The corresponding interface matrix equation is
\begin{equation}
\boldsymbol C\boldsymbol U=\boldsymbol Z,
\qquad
\boldsymbol C\in\mathbb{R}^{2N_c\times(m_1+m_2)},
\qquad
\boldsymbol Z=\mathbf 0.
\label{nbd3}
\end{equation}

Combining \eqref{au3}, \eqref{bd3}, and \eqref{nbd3}, we obtain the final least-squares system
\begin{equation}
\begin{bmatrix}
\boldsymbol A\\
\xi\boldsymbol B\\
\eta\boldsymbol C
\end{bmatrix}
\boldsymbol U
=
\begin{bmatrix}
\boldsymbol L\\
\xi\boldsymbol F\\
\eta\boldsymbol Z
\end{bmatrix},
\label{ab3}
\end{equation}
where $\xi$ and $\eta$ are penalty parameters for the boundary and interface equations, respectively. Solving \eqref{ab3} by least--squares yields the output-layer coefficients of the two subnetworks.

\begin{remark}[Difference from artificial boundary conditions]
The interface equations in \eqref{nbd3} should not be interpreted as conventional artificial boundary conditions used to close a truncated problem. In classical truncation methods, the exterior domain is removed and its effect must be approximated by an absorbing, transparent, or asymptotic boundary condition imposed on the artificial boundary. In the present method, the exterior region is still represented by a separate randomized neural subnetwork. The interface conditions only couple the near-field and far-field trial spaces. Thus, the role of the interface is to connect two representations, rather than to replace the exterior problem by a prescribed boundary operator.
\end{remark}

\section{Domain-Decomposed RaNN Collocation Method for Fully\\ Unbounded Time-Dependent Problems}
\label{sec:whole}

In this section, we develop a domain-decomposed RaNN collocation method (DD-RaNN-CM) for fully unbounded time-dependent problems. The method is particularly useful when the domain has more than one far-field direction, as in whole-space problems. We present the formulation using a one-dimensional time-dependent Schr\"odinger equation as a representative example:
\begin{equation}
\begin{cases}
i u_t+u_{xx}=f,
& x\in\Omega=\mathbb R,\quad t\in[0,T],\\
u(x,0)=g(x),
& x\in\mathbb R,\\
u(x,t)\rightarrow0,
& |x|\rightarrow\infty.
\end{cases}
\label{eq:Whole}
\end{equation}

\subsection{Fully Unbounded Spatio-Temporal Decomposition}
\label{subsec:whole-decomposition}

For fully unbounded problems, a single randomized neural network with one parameter scale may have difficulty resolving both the localized core of the solution and the two far-field tails. We therefore adopt a three-network decomposition. Let
\[
x_{\mathrm{cut}}^{(1)}<x_{\mathrm{cut}}^{(2)}
\]
be two interface locations. The whole line is decomposed into
\[
\Omega_1=(-\infty,x_{\mathrm{cut}}^{(1)}),
\qquad
\Omega_2=[x_{\mathrm{cut}}^{(1)},x_{\mathrm{cut}}^{(2)}],
\qquad
\Omega_3=(x_{\mathrm{cut}}^{(2)},+\infty).
\]
The left and right subnetworks represent the far-field tails, while the middle subnetwork represents the near-field region.

The spatio-temporal numerical solution is written as
\begin{equation}
u_h(x,t)=
\begin{cases}
\displaystyle
\sum_{k=1}^{m_1}
\hat u_k^{(1)}\phi_k^{(1)}(x,t),
& x\in(-\infty,x_{\mathrm{cut}}^{(1)}),\quad t\in[0,T],\\[2ex]
\displaystyle
\sum_{k=1}^{m_2}
\hat u_k^{(2)}\phi_k^{(2)}(x,t),
& x\in[x_{\mathrm{cut}}^{(1)},x_{\mathrm{cut}}^{(2)}],\quad t\in[0,T],\\[2ex]
\displaystyle
\sum_{k=1}^{m_3}
\hat u_k^{(3)}\phi_k^{(3)}(x,t),
& x\in(x_{\mathrm{cut}}^{(2)},+\infty),\quad t\in[0,T].
\end{cases}
\label{solu:Whole}
\end{equation}
Here $\{\phi_k^{(\alpha)}\}_{k=1}^{m_\alpha}$ are randomized spatio-temporal features generated by the $\alpha$-th subnetwork.

\subsection{Collocation Discretization}
\label{subsec:whole-collocation}

For each subregion, let
\[
\{(x_j^{(\alpha)},t_j^{(\alpha)})\}_{j=1}^{N_\alpha},
\qquad
\alpha=1,2,3,
\]
be the collocation points. In practice, the spatial coordinates are sampled from distributions adapted to the corresponding subregion, while the temporal coordinates are sampled in $[0,T]$.

At the collocation points, the Schr\"odinger equation \eqref{eq:Whole} gives
\begin{equation}
i\partial_t u_h(x_j^{(\alpha)},t_j^{(\alpha)})
+
\partial_{xx}u_h(x_j^{(\alpha)},t_j^{(\alpha)})
=
f(x_j^{(\alpha)},t_j^{(\alpha)}),
\qquad
j=1,\ldots,N_\alpha,
\quad
\alpha=1,2,3.
\label{strong:whole}
\end{equation}
Substituting \eqref{solu:Whole} yields
\begin{equation}
\sum_{k=1}^{m_\alpha}
\hat u_k^{(\alpha)}
\left[
i\partial_t\phi_k^{(\alpha)}(x_j^{(\alpha)},t_j^{(\alpha)})
+
\partial_{xx}\phi_k^{(\alpha)}(x_j^{(\alpha)},t_j^{(\alpha)})
\right]
=
f(x_j^{(\alpha)},t_j^{(\alpha)}),
\end{equation}
for $j=1,\ldots,N_\alpha$ and $\alpha=1,2,3$.

Thus each subnetwork contributes a linear block
\[
\boldsymbol A^{(\alpha)}\boldsymbol U^{(\alpha)}
=
\boldsymbol L^{(\alpha)},
\qquad
\boldsymbol U^{(\alpha)}
=
(\hat u_1^{(\alpha)},\ldots,\hat u_{m_\alpha}^{(\alpha)})^\top,
\]
where
\[
\boldsymbol A_{jk}^{(\alpha)}
=
i\partial_t\phi_k^{(\alpha)}(x_j^{(\alpha)},t_j^{(\alpha)})
+
\partial_{xx}\phi_k^{(\alpha)}(x_j^{(\alpha)},t_j^{(\alpha)}),
\qquad
\boldsymbol L_j^{(\alpha)}
=
f(x_j^{(\alpha)},t_j^{(\alpha)}).
\]
Stacking the three blocks gives
\begin{equation}
\boldsymbol A\boldsymbol U=\boldsymbol L,
\qquad
\boldsymbol A=
\mathrm{diag}(\boldsymbol A^{(1)},\boldsymbol A^{(2)},\boldsymbol A^{(3)}),
\qquad
\boldsymbol U^\top=((\boldsymbol U^{(1)})^\top,(\boldsymbol U^{(2)})^\top,(\boldsymbol U^{(3)})^\top),
\label{AU:whole}
\end{equation}
with
\[
\boldsymbol L^\top
=
((\boldsymbol L^{(1)})^\top,(\boldsymbol L^{(2)})^\top,(\boldsymbol L^{(3)})^\top).
\]

Since the Schr\"odinger equation is complex-valued, the matrices and vectors in \eqref{AU:whole} are generally complex-valued. One may either solve the resulting complex least-squares problem directly or rewrite it as an equivalent real least-squares problem by separating the real and imaginary parts. The numerical experiments below use the latter implementation.

\subsection{Initial and Interface Conditions}
\label{subsec:whole-conditions}

The initial condition $u(x,0)=g(x)$ is enforced at sampled points
\[
\{(y_j,0)\}_{j=1}^{N_b}.
\]
If $y_j\in\overline{\Omega_\alpha}$, we impose
\[
\sum_{k=1}^{m_\alpha}
\hat u_k^{(\alpha)}\phi_k^{(\alpha)}(y_j,0)
=
g(y_j).
\]
Collecting all initial-condition equations gives
\begin{equation}
\boldsymbol B\boldsymbol U=\boldsymbol F.
\label{bd3:sho}
\end{equation}

Next, the subnetworks are coupled at the two near--far interfaces
\[
x=x_{\mathrm{cut}}^{(1)},
\qquad
x=x_{\mathrm{cut}}^{(2)}.
\]
For the second-order equation \eqref{eq:Whole}, we impose continuity of both the solution and its first spatial derivative. At sampled temporal points
\[
\{t_j^{(1)}\}_{j=1}^{N_c^{(1)}}
\quad\text{on }x=x_{\mathrm{cut}}^{(1)},
\qquad
\{t_j^{(2)}\}_{j=1}^{N_c^{(2)}}
\quad\text{on }x=x_{\mathrm{cut}}^{(2)},
\]
we impose
\begin{align}
&\sum_{k=1}^{m_1}
\hat u_k^{(1)}
\phi_k^{(1)}(x_{\mathrm{cut}}^{(1)},t_j^{(1)})
-
\sum_{k=1}^{m_2}
\hat u_k^{(2)}
\phi_k^{(2)}(x_{\mathrm{cut}}^{(1)},t_j^{(1)})
=0,
\nonumber\\
&\sum_{k=1}^{m_1}
\hat u_k^{(1)}
\partial_x\phi_k^{(1)}(x_{\mathrm{cut}}^{(1)},t_j^{(1)})
-
\sum_{k=1}^{m_2}
\hat u_k^{(2)}
\partial_x\phi_k^{(2)}(x_{\mathrm{cut}}^{(1)},t_j^{(1)})
=0,
\qquad j=1,\ldots,N_c^{(1)},
\nonumber\\
&\sum_{k=1}^{m_2}
\hat u_k^{(2)}
\phi_k^{(2)}(x_{\mathrm{cut}}^{(2)},t_j^{(2)})
-
\sum_{k=1}^{m_3}
\hat u_k^{(3)}
\phi_k^{(3)}(x_{\mathrm{cut}}^{(2)},t_j^{(2)})
=0,
\nonumber\\
&\sum_{k=1}^{m_2}
\hat u_k^{(2)}
\partial_x\phi_k^{(2)}(x_{\mathrm{cut}}^{(2)},t_j^{(2)})
-
\sum_{k=1}^{m_3}
\hat u_k^{(3)}
\partial_x\phi_k^{(3)}(x_{\mathrm{cut}}^{(2)},t_j^{(2)})
=0,
\qquad j=1,\ldots,N_c^{(2)}.
\end{align}
Stacking all interface equations gives
\begin{equation}
\boldsymbol C\boldsymbol U=\boldsymbol Z,
\qquad
\boldsymbol C\in
\mathbb R^{(2N_c^{(1)}+2N_c^{(2)})\times(m_1+m_2+m_3)},
\qquad
\boldsymbol Z=\mathbf 0.
\label{nbd4}
\end{equation}

Combining the PDE collocation equations, the initial condition, and the interface conditions, we obtain
\begin{equation}
\begin{bmatrix}
\boldsymbol A\\
\xi\boldsymbol B\\
\eta\boldsymbol C
\end{bmatrix}
\boldsymbol U
=
\begin{bmatrix}
\boldsymbol L\\
\xi\boldsymbol F\\
\eta\boldsymbol Z
\end{bmatrix},
\label{ab4}
\end{equation}
where $\xi$ and $\eta$ are penalty parameters for the initial and interface equations, respectively. The least-squares solution of \eqref{ab4} gives all output-layer coefficients of the three subnetworks.

\begin{remark}
The three-network construction is not restricted to the linear Schr\"odinger equation. It can also be used for other fully unbounded problems with different left and right far-field behavior. For more complicated solution profiles, the number of subnetworks can be increased. One may also replace the sharp domain partition by smooth window functions, although this leads to a different coupling formulation.
\end{remark}

%%%%%%%%%%%%%%%%%%%%%%%%%%%%%%%%

	\section{Numerical Analysis}\label{sec:analysis}

In this section, the theoretical analysis is formulated in a broken Sobolev framework. This is necessary because the domain-decomposed ansatz uses characteristic functions. A piecewise function of the form
\[
	u_N=\chi_1u_{N_1}+\chi_2u_{N_2},\quad \chi_i=\mathbf{1}_{\Omega_i},\,i=1,2,  
\]
need not belong to the global space $H^s(\Real^d)$ for $s\ge 1$ unless exact trace matching conditions are imposed on the interface. Therefore, the approximation theorem below is stated in a broken Sobolev norm. Interface compatibility is treated separately through a broken graph norm and an interface jump functional.

Let
\[
	\Omega_1=\{\bx\in\Real^d:\|\bx\|_2<M\},\qquad
	\Omega_2=\{\bx\in\Real^d:\|\bx\|_2>M\},\qquad
	\Gamma=\{\bx\in\Real^d:\|\bx\|_2=M\}.
\]
Then
\[
	\Real^d=\Omega_1\cup\Gamma\cup\Omega_2.
\]
Here $\Omega_1$ denotes the near-field region and $\Omega_2$ denotes the far-field region. The same analysis applies to more general decompositions with a sufficiently regular interface.

\begin{definition}[Broken Sobolev space]
For $s\ge0$, define
\[
	H^s_{\rm br}(\Omega_1\cup\Omega_2)
	:=\{v=(v_1,v_2): v_i\in H^s(\Omega_i),\ i=1,2\},
\]
with norm
\[
	\|v\|_{H^s_{\rm br}}^2
	:=\|v_1\|_{H^s(\Omega_1)}^2+\|v_2\|_{H^s(\Omega_2)}^2.
\]
When a piecewise function is written as $v=\chi_1v_1+\chi_2v_2$, it is identified with the pair $(v_1,v_2)$ in $H^s_{\rm br}(\Omega_1\cup\Omega_2)$. No trace matching across $\Gamma$ is imposed in the definition of the broken space.
\end{definition}

\begin{remark}
If $v=\chi_1v_1+\chi_2v_2$ and $[v]:=v_1|_\Gamma-v_2|_\Gamma\ne0$, then the distributional derivative contains an interface measure of the form $[v]\bn\delta_\Gamma$. Hence $v$ generally fails to belong to $H^1(\Real^d)$, and consequently cannot be used in a global $H^s(\Real^d)$ approximation theorem for $s\ge1$. The broken norm avoids this difficulty by measuring the Sobolev regularity separately on each subdomain. Interface compatibility is then measured by a separate jump functional.
\end{remark}

We use the following broken randomized neural trial space:
\[
\begin{aligned}
\cN_{\rho,\mathbf N}^{\rm br}
:=\Bigl\{&u=\chi_1u_1+\chi_2u_2:\ 
 u_l(\bx)=\sum_{k=1}^{N_l}c_k^{(l)}\prod_{j=1}^d\rho(w_{j,k}^{(l)}x_j+b_{j,k}^{(l)}),\ l=1,2,\\
&\sum_{k=1}^{N_l}|c_k^{(l)}|\le C_{N_l},\quad
R_{\min}\le |w_{j,k}^{(l)}|\le R_{\max},\quad |b_{j,k}^{(l)}|\le B_l\Bigr\}.
\end{aligned}
\]
Here $\mathbf N=(N_1,N_2)$ denotes the subnetwork sizes together with the chosen coefficient budgets $C_{N_1}$ and $C_{N_2}$. The coefficient budgets are part of the trial-space specification. In approximation statements below, they are chosen large enough to contain the finite-sum approximants obtained from the representing measures. Thus the approximation theorem is a bounded-coefficient result on a sufficiently large coefficient ball, not a claim for an arbitrarily prescribed small coefficient budget. This notation is separated from the decomposition radius $M$ in the definition of $\Omega_1$, $\Omega_2$, and $\Gamma$.
The functions in $\cN_{\rho,\mathbf N}^{\rm br}$ are not assumed to belong to global $H^s(\Real^d)$.

The coefficient bound in the definition of $\cN_{\rho,\mathbf N}^{\rm br}$ is part of the theoretical stability framework. Consequently, the least-squares error estimate below applies to a constrained, regularized, or a posteriori bounded least-squares solution whose output-layer coefficients remain in this coefficient ball. If an unconstrained least-squares system is solved in practice, the estimate should be interpreted after verifying such a coefficient bound a posteriori, or after enlarging the coefficient ball so that it contains the computed solution.

Let $V$ be a broken Hilbert space, and let
\[
	\cG_{\rm br}v=(\cG_1v_1,\cG_2v_2)
\]
be the piecewise differential operator. We define the population broken loss by
\begin{equation}\label{eq:pop-loss-br}
	\cL_{\rm br}(v)
	:=\|\cG_{\rm br}v-f\|_Y^2+\mathcal J_\Gamma(v),
\end{equation}
where $\mathcal J_\Gamma(v)$ is an interface jump functional compatible with the chosen broken space. Its empirical counterpart is denoted by $\widetilde{\cL}_{\rm br}(v)$.

\begin{assumption}[Broken graph norm equivalence]\label{assmp:graph norm}
Assume that the boundary, decay, normalization, or transmission conditions needed to remove the kernel of the operator have been imposed. More precisely, the estimate below is required on the admissible error space associated with the PDE. Assume that there exist constants $C_L,C_U>0$ such that, for all admissible error functions $v$ in the chosen broken graph space $V$,
\begin{equation}\label{eq:broken-graph-norm}
	C_L\|v\|_V^2
	\le
	\|\cG_{\rm br}v\|_Y^2+\mathcal J_\Gamma(v)
	\le
	C_U\|v\|_V^2.
\end{equation}
The space $V$, the residual space $Y$, and the interface functional $\mathcal J_\Gamma$ must be chosen consistently. This assumption is a stability statement for the error equation; it is not meant to assert coercivity on a larger space containing nontrivial kernels.
\end{assumption}

\begin{remark}[Examples of broken graph norms]
For the elliptic operator
\[
	\cG u=-\nabla\cdot(\beta\nabla u)+\gamma u,
\]
assume that $\beta$ is uniformly elliptic and sufficiently regular, and that $\gamma\ge\gamma_0>0$. We fix a unit normal vector $\bn$ on $\Gamma$, pointing from $\Omega_1$ to $\Omega_2$, and use the scalar jump convention
\[
[v]=v_1|_\Gamma-v_2|_\Gamma,
\qquad
[\partial_n v]=\nabla v_1\cdot\bn-\nabla v_2\cdot\bn.
\]
For conormal fluxes, if $\bn_i$ denotes the outward normal of $\Omega_i$, then the transmission residual is
\[
\mathcal F_\Gamma(v)
:=
[(\beta\nabla v)\cdot n]_{\rm flux}
:=
(\beta_1\nabla v_1)\cdot\bn_1
+
(\beta_2\nabla v_2)\cdot\bn_2 .
\]
Equivalently, with the fixed normal $\bn$ from $\Omega_1$ to $\Omega_2$, this residual may be written as
\[
\mathcal F_\Gamma(v)
=
(\beta_1\nabla v_1)\cdot\bn
-
(\beta_2\nabla v_2)\cdot\bn .
\]
The notation \(\mathcal F_\Gamma(v)\) is used below for the conormal-flux transmission residual, independently of which equivalent normal convention is adopted. The same convention is used in the interface functional $\mathcal J_\Gamma$. If $\Gamma$ is an artificial interface inside a medium with globally smooth coefficients, a strong broken $H^2$ estimate may be formulated with
\[
	V=H^2_{\rm br}(\Omega_1\cup\Omega_2),\qquad
	Y=L^2(\Omega_1)\times L^2(\Omega_2),
\]
and
\[
	\mathcal J_\Gamma(v)
	=\|[v]\|_{H^{3/2}(\Gamma)}^2
	+\|[\partial_n v]\|_{H^{1/2}(\Gamma)}^2.
\]
Under piecewise $H^2$ elliptic regularity and a compatible transmission stability estimate, Assumption~\ref{assmp:graph norm} gives a broken $H^2$ error estimate. For a genuine transmission problem with discontinuous coefficients, the normal derivative jump should be replaced by the conormal-flux residual $\mathcal F_\Gamma(v)$.

A weaker energy framework should be understood in a graph space rather than in the whole space $H^1_{\rm br}$. For instance, one may use a space of the form
\[
V_{1,{\rm br}}
=
\{v=(v_1,v_2): v_i\in H^1(\Omega_i),\ \cG_i v_i\in H^{-1}(\Omega_i),\
(\beta\nabla v_i)\cdot \bn_i\in H^{-1/2}(\Gamma)\},
\]
where the conormal trace is understood in the weak sense induced by Green's identity. More precisely, for a trace datum $\psi\in H^{1/2}(\Gamma)$ and any lifting $\Psi\in H^1(\Omega_i)$ with $\Psi|_\Gamma=\psi$, one may define the weak conormal trace by
\[
\langle (\beta\nabla v_i)\cdot\bn_i,\psi\rangle_\Gamma
=
\int_{\Omega_i}(\beta\nabla v_i)\cdot\nabla\Psi\,\mathrm dx
-
\langle \cG_i v_i-\gamma v_i,\Psi\rangle_{\Omega_i},
\]
where \(\bn_i\) is the outward normal of \(\Omega_i\). If the opposite normal convention is used, the sign of the conormal trace must be changed accordingly. 
This definition is independent of the chosen lifting provided
\[
\cG_i v_i-\gamma v_i=-\nabla\cdot(\beta\nabla v_i)
\quad\hbox{in }H^{-1}(\Omega_i):=(H^{1}(\Omega_i))'.
\]
Indeed, if two liftings have the same trace on \(\Gamma\), then their difference
\(\Phi\) belongs to \(H_0^1(\Omega_i)\), and
\[
\int_{\Omega_i}(\beta\nabla v_i)\cdot\nabla\Phi\,\mathrm dx
-
\langle \cG_i v_i-\gamma v_i,\Phi\rangle_{\Omega_i}
=
\int_{\Omega_i}(\beta\nabla v_i)\cdot\nabla\Phi\,\mathrm dx
-
\langle -\nabla\cdot(\beta\nabla v_i),\Phi\rangle_{\Omega_i}
=0 .
\]
 In this case one may take
\[
	Y=H^{-1}(\Omega_1)\times H^{-1}(\Omega_2),
\]
and
\[
	\mathcal J_\Gamma(v)
	=\|[v]\|_{H^{1/2}(\Gamma)}^2
	+\|\mathcal F_\Gamma(v)\|_{H^{-1/2}(\Gamma)}^2.
\]
The conormal-flux jump is needed for a genuine broken energy estimate. The pointwise or $L^2(\Gamma)$ interface penalties used in the numerical implementation should be regarded as computable surrogates of these trace norms, rather than as uniformly equivalent quantities, unless additional finite-dimensional sampling or inverse estimates are imposed.
\end{remark}

We next record the precise conditional error decomposition used in the sequel. Assume that $\cG_{\rm br}$ is linear and that the exact solution $u^*$ satisfies
\[
	\cG_{\rm br}u^*=f,
\]
together with the homogeneous transmission conditions appearing in $\mathcal J_\Gamma$. Thus, for every trial function $u$,
\[
	\cG_{\rm br}(u-u^*)=\cG_{\rm br}u-f,
	\qquad
	\mathcal J_\Gamma(u-u^*)=\mathcal J_\Gamma(u).
\]
We also assume that the errors \(u-u^*\), for
\(u\in\cN_{\rho,\mathbf N}^{\rm br}\), belong to the admissible error space
in Assumption~\ref{assmp:graph norm}.

Let \(u_\rho\in\cN_{\rho,\mathbf N}^{\rm br}\) be a computed empirical
least-squares solution satisfying
\[
\widetilde\cL_{\rm br}(u_\rho)
\le
\inf_{u\in\cN_{\rho,\mathbf N}^{\rm br}}
\widetilde\cL_{\rm br}(u)+\eta_{\rm opt},
\]
where \(\eta_{\rm opt}\ge0\) denotes the optimization suboptimality. Define the empirical-consistency gap on the same coefficient-bounded trial space by
\[
S_{\rm br}:=
\sup_{u\in\cN_{\rho,\mathbf N}^{\rm br}}
|\cL_{\rm br}(u)-\widetilde\cL_{\rm br}(u)|.
\]
The coefficient-bounded trial space is fixed throughout this estimate. The comparison functions, the computed least-squares solution, and the uniform consistency quantity \(S_{\rm br}\) are all taken over this same coefficient ball. If the coefficient budgets are enlarged to improve approximation accuracy, then the constants in the empirical-consistency estimates may also change.

For any \(u_a\in\cN_{\rho,\mathbf N}^{\rm br}\), set
\(e_\rho=u_\rho-u^*\). By Assumption~\ref{assmp:graph norm},
\[
\begin{aligned}
	C_L\|u_\rho-u^*\|_V^2
	&=C_L\|e_\rho\|_V^2 \le \|\cG_{\rm br}e_\rho\|_Y^2+\mathcal J_\Gamma(e_\rho)=\cL_{\rm br}(u_\rho) \\
	&\le \widetilde\cL_{\rm br}(u_\rho)+S_{\rm br} \le \widetilde\cL_{\rm br}(u_a)+\eta_{\rm opt}+S_{\rm br} \\
	&\le \cL_{\rm br}(u_a)+\eta_{\rm opt}+2S_{\rm br} \\
	&\le C_U\|u_a-u^*\|_V^2+\eta_{\rm opt}+2S_{\rm br}.
\end{aligned}
\]
Taking the infimum over \(u_a\in\cN_{\rho,\mathbf N}^{\rm br}\) yields
\begin{equation}\label{eq:broken-error-decomposition}
	\|u_\rho-u^*\|_V^2
	\le
	\frac{C_U}{C_L}
	\inf_{u\in\cN_{\rho,\mathbf N}^{\rm br}}\|u-u^*\|_V^2
	+\frac{2}{C_L}S_{\rm br}
	+\frac{1}{C_L}\eta_{\rm opt}.
\end{equation}

It is important to distinguish the stability norm $V$ from the approximation norm used later. The broken $H^s$ approximation theorem below directly controls the approximation term in \eqref{eq:broken-error-decomposition} when $V=H^s_{\rm br}$, for instance in the broken $H^2$ framework with $s=2$. If $V$ is a stronger graph space containing additional interface traces or fluxes, then approximation in that stronger $V$-norm must be assumed separately.

\subsection{Approximation Error}

We now prove the approximation result in the broken Sobolev norm. The conclusion is weakened from a global $H^s(\Real^d)$ estimate to a broken $H^s$ estimate, which is the natural statement for characteristic-function-based domain decomposition.

\begin{assumption}\label{As:decay}
The target function $u^*\in H^{s+\varepsilon_C}(\Real^d)\cap L^1(\Real^d)$ satisfies
\[
	|u^*(\bx)|\le \frac{C_f}{(1+\|\bx\|_2)^{d+\epsilon_f}},
	\qquad \bx\in\Real^d,
\]
where $C_f>0$, $\epsilon_f>0$, and $\varepsilon_C>0$.
\end{assumption}

\begin{assumption}[Activation admissibility]\label{As:act}
Throughout the approximation analysis, $s\in\mathbb N_0$ is fixed. The activation function satisfies:
\begin{itemize}
\item[(A1)] $\rho\in C^{s+1}(\Real)$.
\item[(A2)] $\rho^{(j)}\in L^1(\Real)$ for $j=0,\ldots,s+1$ and $\widehat{\rho'}(1)\ne0$.
\item[(A3)] There exist $k_\rho>1/2$ and $C_\rho>0$ such that
\[
	|\rho^{(j)}(x)|\le \frac{C_\rho}{(1+|x|)^{k_\rho}},
	\qquad j=0,\ldots,s+1.
\]
\end{itemize}
\end{assumption}
The decay condition is chosen sufficiently strong to ensure that, for bounded parameter sets with $|w_j|\ge R_{\min}>0$, the tensor-product feature map
\[
\theta\mapsto \phi_\theta,
\qquad
\phi_\theta(\bx)=\prod_{j=1}^d\rho(w_jx_j+b_j),
\]
is continuous and uniformly bounded as a map from the parameter set into $H^s(\Omega_l)$, $l=1,2$. The lower bound $|w_j|\ge R_{\min}>0$ is essential for the far-field subdomain; without it, the decay of the rescaled features may degenerate and uniform $H^s(\Omega_2)$ bounds may fail. This property is used below in the Bochner quadrature and randomized-parameter arguments.

In numerical implementations, one may also use tensor-product features of the form
\[
\phi_\theta(\bx)
=
\prod_{j=1}^d \rho_j(w_jx_j+b_j),
\]
where the one-dimensional activation functions \(\rho_j\) are allowed to be
different. In the experiments, we sometimes use non-decaying activations such as \(\tanh\). These choices are empirical variants and are not directly covered by the decay-based approximation estimates above. The theoretical estimates below apply to decay-admissible activations, or to trial spaces for which the far-field tail-control assumptions can be verified.

We use the Fourier convention
\[
	\widehat g(\bxi)=\int_{\Real^d}g(\bx)e^{-i\bxi^\top\bx}\dbx,
	\qquad
	g(\bx)=(2\pi)^{-d}\int_{\Real^d}\widehat g(\bxi)e^{i\bxi^\top\bx}\,\mathrm d\bxi.
\]
For the fixed integer $s\in\mathbb N_0$ used in the approximation analysis, set
\[
	\|g\|_{s,\Real^d}^2
	:=\int_{\Real^d}(1+\|\bxi\|_2^2)^s|\widehat g(\bxi)|^2\,\mathrm d\bxi.
\]
This norm is equivalent to the standard Sobolev norm. Constants depending on the Fourier normalization are absorbed into generic constants.

\begin{lemma}[Band-limited approximation]\label{lem:band}
Let $v\in H^{s+\varepsilon_C}(\Real^d)\cap L^1(\Real^d)$ with $\varepsilon_C>0$. Let $m_{R,a}\in C_c^\infty(\Real^d)$ satisfy $0\le m_{R,a}\le1$,
\[
	m_{R,a}(\bxi)=1
	\quad\hbox{if}\quad
	R_{\min}+a\le |\xi_j|\le R_{\max}-a,
	\quad j=1,\ldots,d,
\]
and
\[
	m_{R,a}(\bxi)=0
	\quad\hbox{if}\quad
	|\xi_j|\le R_{\min}\ \hbox{for some }j
	\quad\hbox{or}\quad
	|\xi_j|\ge R_{\max}\ \hbox{for some }j.
\]
Define $v_R=P_{R,a}v$ by
\[
	\widehat{v_R}(\bxi)=m_{R,a}(\bxi)\widehat v(\bxi).
\]
Assume $R_{\max}\ge1$, $0<a\le R_{\max}/2$, $R_{\min}+a<R_{\max}-a$, and
\[
	R_{\min}+a\le  R_{\max}^{-(2s+d-1+2\varepsilon_R)}.
\]
The lower cutoff $R_{\min}$ is therefore not fixed independently of $R_{\max}$; it must tend to zero as $R_{\max}\to\infty$. Otherwise, neighborhoods of the coordinate hyperplanes in frequency space would be permanently removed and a general function in $H^s(\Real^d)\cap L^1(\Real^d)$ could not be approximated.
Then
\begin{equation}\label{eq:H1-band-error}
	\|v-v_R\|_{s,\Real^d}
	\le
	C R_{\max}^{-\varepsilon_C}\|v\|_{s+\varepsilon_C,\Real^d}
	+C R_{\max}^{-\varepsilon_R}\|v\|_{L^1(\Real^d)}.
\end{equation}
Consequently, by restriction continuity,
\[
	\|v-v_R\|_{H^s_{\rm br}}
	\le C_{\rm br}\|v-v_R\|_{s,\Real^d}.
\]
\end{lemma}
\begin{proof}
Since $0\le m_{R,a}\le1$, the error is supported where the multiplier is not equal to one. We cover this set by the following high-frequency, low-frequency, and coordinate-hyperplane regions. The high-frequency part is
\[
A_{\rm high}=\{\bxi: |\xi_j|>R_{\max}-a\ \hbox{for some }j\},
\]
a low-frequency box near the origin,
\[
A_{\rm low}=\{\bxi: |\xi_j|<R_{\min}+a\ \hbox{for all }j\},
\]
and a mixed region near the coordinate hyperplanes,
\[
A_{\rm mix}=\{\bxi: |\xi_j|<R_{\min}+a\ \hbox{for at least one }j,
\ |\xi_k|\ge R_{\min}+a\ \hbox{for at least one }k,
\ |\xi_m|\le R_{\max}-a\ \hbox{for all }m\}.
\]

By construction,
\[
\{\bxi:\ m_{R,a}(\bxi)\ne1\}
\subset A_{\rm high}\cup A_{\rm low}\cup A_{\rm mix}.
\]
Hence it is sufficient to estimate the error over these three regions.
On $A_{\rm high}$,
\[
(1+\|\bxi\|_2^2)^s
\le
C R_{\max}^{-2\varepsilon_C}(1+\|\bxi\|_2^2)^{s+\varepsilon_C},
\]
where $a\le R_{\max}/2$ is used. This gives the first term in \eqref{eq:H1-band-error}. On $A_{\rm low}\cup A_{\rm mix}$, use $|\widehat v(\bxi)|\le \|v\|_{L^1(\Real^d)}$. Moreover,
\[
|A_{\rm mix}|
\le C_d(R_{\min}+a)(R_{\max}-a)^{d-1}
\le C_d R_{\max}^{-2s-2\varepsilon_R},
\]
and
\[
|A_{\rm low}|\le C_d(R_{\min}+a)^d.
\]
On $A_{\rm mix}$ one has $(1+\|\bxi\|_2^2)^s\le C R_{\max}^{2s}$. Hence
\[
\int_{A_{\rm mix}}(1+\|\bxi\|_2^2)^s|\widehat v(\bxi)|^2\,d\bxi
\le C R_{\max}^{2s}|A_{\rm mix}|\|v\|_{L^1(\Real^d)}^2
\le C R_{\max}^{-2\varepsilon_R}\|v\|_{L^1(\Real^d)}^2.
\]
On $A_{\rm low}$, $(1+\|\bxi\|_2^2)^s\le C$ for $R_{\min}+a$ sufficiently small, and therefore
\[
\int_{A_{\rm low}}(1+\|\bxi\|_2^2)^s|\widehat v(\bxi)|^2\,d\bxi
\le
C |A_{\rm low}|\|v\|_{L^1(\Real^d)}^2
\le
C(R_{\min}+a)^d\|v\|_{L^1(\Real^d)}^2.
\]
Using $R_{\min}+a\le C R_{\max}^{-(2s+d-1+2\varepsilon_R)}$, and taking $R_{\max}\ge1$, we have
\[
(R_{\min}+a)^d
\le
C R_{\max}^{-d(2s+d-1+2\varepsilon_R)}
\le
C R_{\max}^{-2\varepsilon_R},
\]
since $d(2s+d-1+2\varepsilon_R)\ge2\varepsilon_R$. Therefore
\[
\int_{A_{\rm low}}(1+\|\bxi\|_2^2)^s|\widehat v(\bxi)|^2\,d\bxi
\le
C R_{\max}^{-2\varepsilon_R}\|v\|_{L^1(\Real^d)}^2.
\]
After combining the estimates on $A_{\rm high}$, $A_{\rm mix}$, and
$A_{\rm low}$, we take square roots and use
$\sqrt{a+b}\le \sqrt a+\sqrt b$ to obtain \eqref{eq:H1-band-error}. The broken estimate follows from the continuity of the restriction map $H^s(\Real^d)\to H^s(\Omega_i)$ for each subdomain.
\end{proof}

We first verify the local bounded-parameter representation on the bounded subdomain. The argument is based on the Fourier representation of the band-limited function and a truncation of the bias variables. Since the subdomain is bounded, the truncation error is controlled only by the \(L^1\)-tails of the one-dimensional activation derivatives.

\begin{theorem}[Near-field bounded-parameter representation]
	\label{thm:near-field-representation}
	Let \(v_R\) be the band-limited approximation from Lemma~\ref{lem:band}. Assume
	that
	\(
	\operatorname{supp}\widehat{v_R}
	\subset
	\widehat{\Omega}_R
	:=
	\set{\bw\in\Real^d:
		R_{\min}\le |w_j|\le R_{\max},\ j=1,\ldots,d},
	\)
	where \(0<R_{\min}<R_{\max}<\infty\). Let
	\[
	\Omega_1\subset\set{\bx\in\Real^d:\|\bx\|_2<M}
	\]
	be a bounded subdomain. Suppose that
	\[
	\rho\in C^{s+1}(\Real),\qquad
	\rho^{(j)}\in L^1(\Real),\quad j=1,\ldots,s+1,
	\qquad
	\widehat{\rho'}(1)\ne0.
	\]
	Then, for every \(\varepsilon_F>0\), there exist \(B_1>0\) and a finite
	signed Borel measure \(\nu_1\) supported on
	\[
	\Theta_1
	=
	\set{(\bw,\bb):
		R_{\min}\le |w_j|\le R_{\max},\
		|b_j|\le B_1,\ j=1,\ldots,d}
	\]
	such that
	\[
	\left\|
	v_R-
	\int_{\Theta_1}
	\prod_{j=1}^d\rho(w_jx_j+b_j)\,\mathrm{d}\nu_1(\bw,\bb)
	\right\|_{H^s(\Omega_1)}
	\le
	\varepsilon_F/\sqrt2 .
	\]
	Moreover,
	\[
	\|\nu_1\|_{\rm TV}
	\le
	C_{T_1}(R_{\min},R_{\max},B_1,d,\rho)
	\|v_R\|_{0,\Real^d}.
	\]
	Here \(B_1\) may depend on \(v_R\), \(M\), and the requested accuracy
	\(\varepsilon_F\).
\end{theorem}

\begin{proof}
	By Fourier inversion and the support condition on \(\widehat{v_R}\),
	\[
	v_R(\bx)
	=
	\operatorname{Re}\left\{
	\frac1{(2\pi)^d}
	\int_{\widehat{\Omega}_R}
	\widehat{v_R}(\bw)e^{i\bw^\top\bx}\,\mathrm{d}\bw
	\right\}.
	\]
	For each coordinate,
	\[
	\int_{\Real}\rho'(w_jx_j+b_j)e^{-ib_j}\,\mathrm{d}b_j
	=
	e^{iw_jx_j}\widehat{\rho'}(1).
	\]
	Since \(\widehat{\rho'}(1)\ne0\), we obtain
	\[
	v_R(\bx)
	=
	\operatorname{Re}\left\{
	\frac1{(2\pi\widehat{\rho'}(1))^d}
	\int_{\widehat{\Omega}_R}
	\widehat{v_R}(\bw)
	\prod_{j=1}^d
	\int_{\Real}\rho'(w_jx_j+b_j)e^{-ib_j}\,\mathrm{d}b_j
	\,\mathrm{d}\bw
	\right\}.
	\]
	The interchange of integrals is justified by \(\rho'\in L^1(\Real)\) and
	the boundedness of \(\widehat{\Omega}_R\).
	
	For \(B_1>R_{\max}M\), let \(v_{R,B_1}\) be the same expression with each
	\(b_j\)-integral restricted to \(|b_j|\le B_1\). For
	\(|\boldsymbol{\alpha}|\le s\), using
	\[
	\Real^d\setminus[-B_1,B_1]^d
	\subset
	\bigcup_{k=1}^d\{\bb:|b_k|>B_1\},\ \text{and}\ |w_jx_j+b_j|\ge B_1-R_{\max}M
	\]
	we have, for \(\bx\in\Omega_1\),
	\[
	\begin{aligned}
		\left|
		\partial_{\bx}^{\boldsymbol{\alpha}}v_R(\bx)
		-
		\partial_{\bx}^{\boldsymbol{\alpha}}v_{R,B_1}(\bx)
		\right|                                                   
		&\le
		C_{\rho,d}
		\int_{\widehat{\Omega}_R}
		|\widehat{v_R}(\bw)|
		\left|
		\int_{\Real^d\setminus[-B_1,B_1]^d}
		\prod_{j=1}^d
		w_j^{\alpha_j}
		\rho^{(1+\alpha_j)}(w_jx_j+b_j)e^{-ib_j}
		\,\mathrm{d}\bb
		\right|
		\,\mathrm{d}\bw                                           \\
		&\le
		C_{\alpha,R,\rho,d}
		\int_{\widehat{\Omega}_R}
		|\widehat{v_R}(\bw)|
		\sum_{k=1}^d
		\int_{|b_k|>B_1}
		\left|
		\rho^{(1+\alpha_k)}(w_kx_k+b_k)
		\right|
		\,\mathrm{d}b_k
		\,\mathrm{d}\bw                                           \\
		&\le
		C_{\alpha,R,\rho,d}
		\|v_R\|_{0,\Real^d}
		\sum_{k=1}^d
		\eta_{1+\alpha_k}(B_1-R_{\max}M),
	\end{aligned}
	\]
	where
	\[
	\eta_m(L):=\int_{|t|>L}|\rho^{(m)}(t)|\,\mathrm{d}t.
	\]
	Since
	\(\rho^{(m)}\in L^1(\Real)\), \(m=1,\ldots,s+1\), we have
	\(\eta_m(L)\to0\) as \(L\to\infty\). Therefore, after integrating over
	\(\Omega_1\), summing over all \(|\boldsymbol{\alpha}|\le s\), and choosing
	\(B_1\) sufficiently large,
	\[
	\|v_R-v_{R,B_1}\|_{H^s(\Omega_1)}
	\le
	\frac{\varepsilon_F}{\sqrt2}.
	\]
	
	It remains to rewrite \(v_{R,B_1}\) by the original feature \(\rho\). By
	integration by parts,
	\[
	\int_{-B_1}^{B_1}
	\rho'(w_jx_j+b_j)e^{-ib_j}\,\mathrm{d}b_j
	=
	\rho(w_jx_j+B_1)e^{-iB_1}
	-\rho(w_jx_j-B_1)e^{iB_1}      \quad
	+i\int_{-B_1}^{B_1}
	\rho(w_jx_j+b_j)e^{-ib_j}\,\mathrm{d}b_j .
	\]
	Thus there exists a finite complex measure \(\mu_{B_1}\) supported on
	\([-B_1,B_1]\), with \(\|\mu_{B_1}\|_{\rm TV}\le 2+2B_1\), such that
	\[
	\int_{-B_1}^{B_1}
	\rho'(w_jx_j+b_j)e^{-ib_j}\,\mathrm{d}b_j
	=
	\int_{[-B_1,B_1]}
	\rho(w_jx_j+b_j)\,\mathrm{d}\mu_{B_1}(b_j).
	\]
	Substituting this identity into \(v_{R,B_1}\) and taking the real part of the
	resulting complex measure gives a finite signed Borel measure \(\nu_1\)
	supported on \(\Theta_1\) such that
	\[
	v_{R,B_1}(\bx)
	=
	\int_{\Theta_1}\phi_\theta(\bx)\,\mathrm{d}\nu_1(\theta).
	\]
	Moreover,
	\[
	\|\nu_1\|_{\rm TV}
	\le
	\frac{(2+2B_1)^d}{(2\pi|\widehat{\rho'}(1)|)^d}
	\int_{\widehat{\Omega}_R}|\widehat{v_R}(\bw)|\,\mathrm{d}\bw
	\le
	C_{T_1}(R_{\min},R_{\max},B_1,d,\rho)
	\|v_R\|_{0,\Real^d}.
	\]
	Combining this representation with the preceding \(H^s(\Omega_1)\)-estimate
	proves the result.
\end{proof}

We next verify the far-field part. The proof requires control of the Sobolev tail of \(v_R\) on the exterior domain. We use the following derivative-decay lemma for this purpose.

\begin{lemma}[Decay preservation for the band-limited approximation]
	\label{lem:bandlimited-derivative-decay}
	Let \(v_R=P_{R,a}v\) be the band-limited approximation defined in
	Lemma~\ref{lem:band}. Assume that
	\[
	|v(\bx)|
	\le
	\frac{C}{(1+\|\bx\|_2)^k},
	\qquad \bx\in\Real^d,
	\]
	for some \(C>0\) and \(k>d\). Then, for every multi-index
	\(\boldsymbol{\alpha}\) with \(|\boldsymbol{\alpha}|\le s\), there exists
	\(C_{R,\boldsymbol{\alpha}}>0\) such that
	\[
	|\partial_{\bx}^{\boldsymbol{\alpha}}v_R(\bx)|
	\le
	\frac{C_{R,\boldsymbol{\alpha}}}{(1+\|\bx\|_2)^k},
	\qquad \bx\in\Real^d .
	\]
\end{lemma}

\begin{proof}
	Since
	\[
	\widehat{v_R}(\bxi)=m_{R,a}(\bxi)\widehat v(\bxi),
	\]
	we have
	\[
	\partial_{\bx}^{\boldsymbol{\alpha}}v_R
	=
	K_{\boldsymbol{\alpha}}*v,
	\qquad
	K_{\boldsymbol{\alpha}}
	=
	\mathcal F^{-1}
	\left[
	(i\bxi)^{\boldsymbol{\alpha}}m_{R,a}(\bxi)
	\right].
	\]
	Because \(m_{R,a}\in C_c^\infty(\Real^d)\), the kernel
	\(K_{\boldsymbol{\alpha}}\) is a Schwartz function. Hence, for any \(N>0\),
	\[
	|K_{\boldsymbol{\alpha}}(\by)|
	\le
	C_{\boldsymbol{\alpha},N}(1+\|\by\|_2)^{-N}.
	\]
	Using the decay assumption on \(v\), we obtain
	\[
	\begin{aligned}
		|\partial_{\bx}^{\boldsymbol{\alpha}}v_R(\bx)|
		&\le
		C\int_{\Real^d}
		|K_{\boldsymbol{\alpha}}(\by)|
		(1+\|\bx-\by\|_2)^{-k}
		\,\mathrm d\by                                      \\
		&\le
		C(1+\|\bx\|_2)^{-k}
		\int_{\Real^d}
		(1+\|\by\|_2)^{-N+k}
		\,\mathrm d\by .
	\end{aligned}
	\]
	Here we used
	\[
	(1+\|\bx-\by\|_2)^{-k}
	\le
	C(1+\|\bx\|_2)^{-k}(1+\|\by\|_2)^k .
	\]
	Choosing \(N>d+k\) makes the last integral finite. Therefore,
	\[
	|\partial_{\bx}^{\boldsymbol{\alpha}}v_R(\bx)|
	\le
	\frac{C_{R,\boldsymbol{\alpha}}}
	{(1+\|\bx\|_2)^k},
	\qquad \bx\in\Real^d .
	\]
	This proves the lemma.
\end{proof}

\begin{theorem}[Far-field bounded-parameter representation]
	\label{thm:far-field-representation}
	Let \(v_R\) be the band-limited approximation from Lemma~\ref{lem:band}. Assume that
	\(
	\mathrm{supp}\,\widehat{v_R}
	\subset
	\widehat{\Omega}_R
	:=
	\set{\bxi\in\Real^d:
		R_{\min}\le |\xi_j|\le R_{\max},\ j=1,\ldots,d},
	\)
	where \(0<R_{\min}<R_{\max}<\infty\). Assume also that
	\[
	|v(\bx)|
	\le
	\frac{C}{(1+\|\bx\|_2)^{d+\epsilon_f}},
	\qquad \bx\in\Real^d,
	\]
	for some \(C>0\) and \(\epsilon_f>0\). Let Assumption~\ref{As:act} hold.
	Then, for every fixed \(B_2>0\) and every \(\varepsilon_F>0\), there exists
	\(M>0\) such that, with
	\[
	\Omega_2=\set{\bx\in\Real^d:\|\bx\|_2>M},
	\]
	there is a finite signed Borel measure \(\nu_2\) supported on
	\[
	\Theta_2
	=
	\set{(\bw,\bb):
		R_{\min}\le |w_j|\le R_{\max},\ |b_j|\le B_2,\ j=1,\ldots,d}
	\]
	satisfying
	\[
	\left\|
	v_R-
	\int_{\Theta_2}
	\prod_{j=1}^d\rho(w_jx_j+b_j)\,\mathrm{d}\nu_2(\bw,\bb)
	\right\|_{H^s(\Omega_2)}
	\le
	\varepsilon_F/\sqrt2 .
	\]
	Moreover,
	\[
	\|\nu_2\|_{\rm TV}
	\le
	C_{T_2}(R_{\min},R_{\max},B_2,d,\rho)
	\|v_R\|_{0,\Real^d}.
	\]
\end{theorem}

\begin{proof}
	By Lemma~\ref{lem:bandlimited-derivative-decay}, applied with
	\(k=d+\epsilon_f\), and polar-coordinate integration, the far-field Sobolev tail of \(v_R\) satisfies \(\|v_R\|_{H^s(\Omega_2)}^2\le C M^{-d-2\epsilon_f}\). Hence, after increasing \(M\) if necessary,
	\begin{equation}\label{eq:vR-tail-small-short}
		\|v_R\|_{H^s(\Omega_2)}
		\le
		\frac{\varepsilon_F}{2\sqrt2}.
	\end{equation}
	
	By Fourier inversion,
	\[
	v_R(\bx)
	=
	\mathrm{Re}
	\left\{
	\frac{1}{(2\pi)^d}
	\int_{\widehat{\Omega}_R}
	\widehat{v_R}(\bw)e^{i\bw^\top\bx}\,\mathrm{d}\bw
	\right\}.
	\]
	Since \(\widehat{\rho'}(1)\ne0\), for each coordinate,
	\[
	\int_{\Real}\rho'(w_jx_j+b_j)e^{-ib_j}\,\mathrm{d}b_j
	=
	e^{iw_jx_j}\widehat{\rho'}(1),
	\]
	and therefore
	\[
	e^{i\bw^\top\bx}
	=
	\frac{1}{(\widehat{\rho'}(1))^d}
	\prod_{j=1}^d
	\int_{\Real}
	\rho'(w_jx_j+b_j)e^{-ib_j}\,\mathrm{d}b_j .
	\]
	Since \(\rho'\in L^1(\Real)\) and \(\widehat{v_R}\) is supported in the
	bounded set \(\widehat{\Omega}_R\), Fubini's theorem applies.
	
	For the fixed bias bound \(B_2>0\), define
	\[
	v_{R,B_2}(\bx)
	:=
	\mathrm{Re}
	\left\{
	\frac{1}{(2\pi\widehat{\rho'}(1))^d}
	\int_{\widehat{\Omega}_R}
	\widehat{v_R}(\bw)
	\prod_{j=1}^d
	\left(
	\int_{|b_j|\le B_2}
	\rho'(w_jx_j+b_j)e^{-ib_j}
	\,\mathrm{d}b_j
	\right)
	\mathrm{d}\bw
	\right\}.
	\]
	
	We estimate \(v_{R,B_2}\) on \(\Omega_2\). For \(|\balpha|\le s\),
	differentiating under the integral is justified by Assumption~\ref{As:act}
	and the boundedness of \(\widehat{\Omega}_R\times[-B_2,B_2]^d\). The factors
	\(w_j^{\alpha_j}\) are uniformly bounded by \(R_{\max}^{\alpha_j}\), and
	\(\rho'\) is replaced by \(\rho^{(1+\alpha_j)}\). Hence it is enough to
	estimate
	\[
	G_j(x_j,w_j)
	:=
	\int_{-B_2}^{B_2}
	|\rho^{(1+\alpha_j)}(w_jx_j+b_j)|^2
	\,\mathrm{d}b_j .
	\]
	Using Cauchy--Schwarz in the \(\bw\)-integral, we obtain
	\[
	|\partial_{\bx}^{\balpha}v_{R,B_2}(\bx)|^2
	\le
	C\|v_R\|_{0,\Real^d}^2
	\int_{\widehat{\Omega}_R}
	\prod_{j=1}^d
	G_j(x_j,w_j)
	\,\mathrm{d}\bw ,
	\]
	where \(C\) is independent of \(M\).
	
	Set
	\[
	E_m:=\set{\bx\in\Real^d:|x_m|>M/\sqrt d},
	\qquad m=1,\ldots,d.
	\]
	Since \(\Omega_2\subset\bigcup_{m=1}^d E_m\), it suffices to estimate the
	integral over each \(E_m\). Since \(E_m\) only restricts the coordinate
	\(x_m\), Fubini's theorem gives
	\[
	\int_{E_m}\int_{\widehat{\Omega}_R}\prod_{j=1}^dG_j(x_j,w_j)
	\,\mathrm{d}\bw\,\mathrm{d}\bx
	=
	I_m^{\rm tail}\prod_{j\ne m}I_j^{\rm full},
	\]
	where
	\[
	I_m^{\rm tail}
	:=
	\int_{|x_m|>M/\sqrt d}
	\int_{R_{\min}\le |w_m|\le R_{\max}}
	G_m(x_m,w_m)\,\mathrm{d}w_m\,\mathrm{d}x_m
	\]
	and
	\[
	I_j^{\rm full}
	:=
	\int_{\Real}
	\int_{R_{\min}\le |w_j|\le R_{\max}}
	G_j(x_j,w_j)\,\mathrm{d}w_j\,\mathrm{d}x_j .
	\]
	For \(j\ne m\), the change of variable \(u=w_jx_j+b_j\) gives a factor
	\(1/|w_j|\), and since \(R_{\min}\le |w_j|\le R_{\max}\), we have
	\[
	I_j^{\rm full}\le C<\infty .
	\]
	
	For the distinguished coordinate \(m\), assume
	\(M\ge B_2\sqrt d/R_{\min}\). Then, whenever
	\(|x_m|>M/\sqrt d\), \(R_{\min}\le |w_m|\le R_{\max}\), and
	\(|b_m|\le B_2\), we have
	\[
	|w_mx_m+b_m|
	\ge
	R_{\min}|x_m|-B_2	\ge R_{\min}M/\sqrt d-B_2 .
	\]
	Using the decay condition in Assumption~\ref{As:act} and the one-dimensional
	tail integral with \(2k_\rho>1\), we obtain
	\[
	I_m^{\rm tail}
	\le
	C
	\left(1+R_{\min}M/\sqrt d-B_2\right)^{1-2k_\rho}.
	\]
	Combining the preceding estimates and summing over \(m=1,\ldots,d\), we get
	\[
	\int_{\Omega_2}
	|\partial_{\bx}^{\balpha}v_{R,B_2}(\bx)|^2
	\,\mathrm{d}\bx
	\le
	C_{B_2,\balpha}
	\left(1+R_{\min}M/\sqrt d-B_2\right)^{1-2k_\rho},
	\]
	where
	\[
	C_{B_2,\balpha}
	=
	C_{B_2,\balpha}
	\bigl(
	d,\|v_R\|_{0,\Real^d},R_{\min},R_{\max},
	\widehat{\rho'}(1),C_\rho,k_\rho,B_2,s
	\bigr).
	\]
	Since \(2k_\rho>1\), the right-hand side tends to zero as \(M\to\infty\).
	Therefore, after increasing \(M\) if necessary,
	\begin{equation}\label{eq:vRB2-tail-small-short}
		\|v_{R,B_2}\|_{H^s(\Omega_2)}
		\le
		\frac{\varepsilon_F}{2\sqrt2}.
	\end{equation}
	Combining \eqref{eq:vR-tail-small-short} and
	\eqref{eq:vRB2-tail-small-short}, we obtain
	\[
	\|v_R-v_{R,B_2}\|_{H^s(\Omega_2)}
	\le
	\varepsilon_F/\sqrt2 .
	\]
	
	It remains to rewrite \(v_{R,B_2}\) using \(\rho\) instead of \(\rho'\).
	For each coordinate, integration by parts gives
	\[
	\int_{-B_2}^{B_2}
	\rho'(w_jx_j+b_j)e^{-ib_j}\,\mathrm{d}b_j
	=
	\rho(w_jx_j+B_2)e^{-iB_2}
	-
	\rho(w_jx_j-B_2)e^{iB_2}
	+
	i\int_{-B_2}^{B_2}
	\rho(w_jx_j+b_j)e^{-ib_j}\,\mathrm{d}b_j .
	\]
	Thus \(v_{R,B_2}\) can be written as
	\[
	v_{R,B_2}(\bx)
	=
	\int_{\Theta_2}
	\prod_{j=1}^d\rho(w_jx_j+b_j)
	\,\mathrm{d}\nu_2(\bw,\bb),
	\]
	for a finite signed Borel measure \(\nu_2\) supported on \(\Theta_2\).
	Moreover, since \(\widehat{\Omega}_R\) is bounded,
	\[
	\|\nu_2\|_{\rm TV}
	\le
	C(B_2,d,\rho)
	\int_{\widehat{\Omega}_R}
	|\widehat{v_R}(\bw)|\,\mathrm{d}\bw \le
	C(B_2,d,\rho)|\widehat{\Omega}_R|^{1/2}
	\|v_R\|_{0,\Real^d}.
	\]
	Therefore
	\[
	\|\nu_2\|_{\rm TV}
	\le
	C_{T_2}(R_{\min},R_{\max},B_2,d,\rho)
	\|v_R\|_{0,\Real^d}.
	\]
	The proof is complete.
\end{proof}

The near-field and far-field results above provide local bounded-parameter integral representations on \(\Omega_1\) and \(\Omega_2\), respectively. In the far-field result, the exterior radius \(M\) is chosen sufficiently large, depending on the prescribed tolerance and the fixed bias bound \(B_2\), so that the far-field tail error is controlled. Once this radius is fixed, the representations have finite total variation on bounded parameter sets. The following lemma then discretizes these local integral representations into finite randomized-feature sums in the broken Sobolev norm.

\begin{lemma}[Bochner finite-sum approximation]\label{lem:quad}
	Under Assumptions~\ref{As:decay} and~\ref{As:act}, let
	\(v_R=P_{R,a}u^*\) be the band-limited approximation constructed in
	Lemma~\ref{lem:band}. Then, for every \(\varepsilon_D>0\), there exist
	bounded parameter sets \(\Theta_l\), nodes
	\(\theta_k^{(l)}\in\Theta_l\), and real coefficients \(c_k^{(l)}\),
	\(l=1,2\), such that
	\[
	\left\|
	v_R
	-
	\chi_1\sum_{k=1}^{N_1}c_k^{(1)}\phi_{\theta_k^{(1)}}
	-
	\chi_2\sum_{k=1}^{N_2}c_k^{(2)}\phi_{\theta_k^{(2)}}
	\right\|_{H^s_{\rm br}}
	\le \varepsilon_D .
	\]
	Moreover,
	\[
	\sum_{k=1}^{N_l}|c_k^{(l)}|
	\le \|\nu_l\|_{\rm TV},
	\qquad l=1,2,
	\]
	where \(\nu_l\) are the finite signed Borel measures obtained from the
	local representation results on \(\Omega_l\).
\end{lemma}
\begin{proof}
	By Theorems~\ref{thm:near-field-representation} and~\ref{thm:far-field-representation},
	there exist bounded parameter sets \(\Theta_l\) and finite signed
	Borel measures \(\nu_l\), \(l=1,2\), such that, after choosing the
	local representation accuracies sufficiently small,
	\[
	\left\|
	v_R-\int_{\Theta_l}\phi_\theta\,\mathrm{d}\nu_l(\theta)
	\right\|_{H^s(\Omega_l)}
	\le \frac{\varepsilon_D}{2\sqrt 2},
	\qquad l=1,2.
	\]
	
By Assumption~\ref{As:act}, the map \(\theta\mapsto\phi_\theta\) is continuous and uniformly bounded from \(\Theta_l\) into \(H^s(\Omega_l)\). Hence the Bochner integral
\[
	\int_{\Theta_l}\phi_\theta\,\mathrm{d}\nu_l(\theta)
\]
is well-defined in \(H^s(\Omega_l)\). Choose a finite Borel partition \(\{A_k^{(l)}\}_{k=1}^{N_l}\) of \(\Theta_l\) with sufficiently small diameter and select \(\theta_k^{(l)}\in A_k^{(l)}\). Set
\[
	c_k^{(l)}=\nu_l(A_k^{(l)}).
\]
By the uniform continuity of \(\theta\mapsto\phi_\theta\) in \(H^s(\Omega_l)\), the Riemann-type sum
\[
	\sum_{k=1}^{N_l}c_k^{(l)}\phi_{\theta_k^{(l)}}
\]
approximates the Bochner integral in \(H^s(\Omega_l)\). We choose the partition sufficiently fine so that
\[
\left\|
\int_{\Theta_l}\phi_\theta\,\mathrm{d}\nu_l(\theta)
-
\sum_{k=1}^{N_l}c_k^{(l)}\phi_{\theta_k^{(l)}}
\right\|_{H^s(\Omega_l)}
\le \frac{\varepsilon_D}{2\sqrt2},
\qquad l=1,2.
\]
Therefore, for each \(l=1,2\),
\[
\left\|
 v_R-
\sum_{k=1}^{N_l}c_k^{(l)}\phi_{\theta_k^{(l)}}
\right\|_{H^s(\Omega_l)}
\le
\frac{\varepsilon_D}{\sqrt2}.
\]
The broken norm identity gives
\[
\begin{aligned}
&\left\|
 v_R-
 \sum_{k=1}^{N_1}\chi_1c_k^{(1)}\phi_{\theta_k^{(1)}}
 -
 \sum_{k=1}^{N_2}\chi_2c_k^{(2)}\phi_{\theta_k^{(2)}}
\right\|_{H^s_{\rm br}}^2 \\
&=
\left\|v_R-
\sum_{k=1}^{N_1}c_k^{(1)}\phi_{\theta_k^{(1)}}
\right\|_{H^s(\Omega_1)}^2
+
\left\|v_R-
\sum_{k=1}^{N_2}c_k^{(2)}\phi_{\theta_k^{(2)}}
\right\|_{H^s(\Omega_2)}^2
\le \varepsilon_D^2.
\end{aligned}
\]
Moreover,
\[
\sum_{k=1}^{N_l}|c_k^{(l)}|
=
\sum_{k=1}^{N_l}|\nu_l(A_k^{(l)})|
\le |\nu_l|(\Theta_l)
=
\|\nu_l\|_{\rm TV},
\qquad l=1,2.
\]
This proves the result.
\end{proof}

\begin{theorem}[Bounded-parameter broken \(H^s\) approximation]
	\label{Thm:H1bound-UAT}
	Under Assumptions~\ref{As:decay} and~\ref{As:act}, the following finite-sum
	approximation result holds. For any \(\varepsilon_A>0\), there exist
	\(R_{\min}>0\), \(R_{\max}>0\), a sufficiently large exterior radius \(M>0\),
	bias bounds \(B_1>0\), \(B_2>0\), coefficient budgets \(C_{N_1}\) and
	\(C_{N_2}\), and two single-hidden-layer networks \(u_{N_1}\) and
	\(u_{N_2}\) with bounded inner parameters such that
	\[
	\|u^*-\chi_1u_{N_1}-\chi_2u_{N_2}\|_{H^s_{\rm br}}
	\le\varepsilon_A .
	\]
	The inner parameters satisfy
	\[
	R_{\min}\le |w_{j,k}^{(l)}|\le R_{\max},
	\qquad
	|b_{j,k}^{(l)}|\le B_l,
	\qquad l=1,2,
	\]
	and
	\[
	\sum_{k=1}^{N_l}|c_k^{(l)}|\le C_{N_l},
	\qquad l=1,2.
	\]
	Here the coefficient budgets may depend on the requested accuracy through
	the total-variation bounds of the representing measures.
\end{theorem}
\begin{proof}
	Choose \(v_R=P_{R,a}u^*\) by Lemma~\ref{lem:band} so that
	\[
	\|u^*-v_R\|_{H^s_{\rm br}}
	\le \varepsilon_A/2 .
	\]
	This is possible by the full-space band-limited estimate and restriction
	continuity.
	
	By Assumption~\ref{As:decay} and
	Lemma~\ref{lem:bandlimited-derivative-decay}, the band-limited approximation
	\(v_R\) and its derivatives up to order \(s\) have the required far-field
	decay. Hence the hypotheses of the near-field and far-field representation
	results are satisfied after choosing a sufficiently large exterior radius
	\(M\). Applying Lemma~\ref{lem:quad} with
	\(\varepsilon_D=\varepsilon_A/2\), we obtain finite sums \(u_{N_1}\) and
	\(u_{N_2}\) such that
	\[
	\|v_R-\chi_1u_{N_1}-\chi_2u_{N_2}\|_{H^s_{\rm br}}
	\le \varepsilon_A/2 .
	\]
	The triangle inequality in \(H^s_{\rm br}\) gives
	\[
	\|u^*-\chi_1u_{N_1}-\chi_2u_{N_2}\|_{H^s_{\rm br}}
	\le \varepsilon_A .
	\]
	The parameter bounds follow from the bounded parameter sets in the near-field
	and far-field representations. The coefficient bounds are obtained by choosing
	\(C_{N_l}\ge \|\nu_l\|_{\rm TV}\), \(l=1,2\), after the representing measures
	have been selected.
\end{proof}

The result should be interpreted in the broken Sobolev setting associated with the chosen domain decomposition. In particular, the exterior radius
\(M\) is selected to control the far-field tail, and the coefficient budgets may depend on the requested accuracy.

\begin{assumption}\label{assmp:param-lip}
For each $l=1,2$, there exists $L_l>0$ such that
\[
	\|\phi_\theta-\phi_{\theta'}\|_{H^s(\Omega_l)}
	\le L_l|\theta-\theta'|,
	\qquad \theta,\theta'\in\Theta_l.
\]
This property follows, for compact parameter sets, from the smoothness and decay assumptions on $\rho$, but it is stated explicitly to make the randomized-parameter argument transparent.
\end{assumption}

The randomized-parameter result below is an existence and consistency statement, not a sharp sampling-complexity estimate. The required sampling radius depends on the coefficient variation $\sum_k|c_k^{(l)}|$, which is controlled by the total variation of the representing measure and may grow as the requested representation accuracy increases. Consequently, the number of sampled hidden parameters may grow rapidly as $\varepsilon_A\to0$.

\begin{theorem}[Random-parameter broken $H^s$ approximation]\label{thm:random-broken-UAT}
Let $u_N=\chi_1u_{N_1}+\chi_2u_{N_2}$ be a deterministic approximant from Theorem~\ref{Thm:H1bound-UAT} satisfying
\[
	\|u^*-u_N\|_{H^s_{\rm br}}\le \varepsilon_A/2.
\]
For $l=1,2$, let $\theta_1^{(l)},\ldots,\theta_{N_l}^{(l)}$ be the target parameters. Assume that the hidden parameters in the $l$-th randomized subnetwork are sampled independently from the uniform distribution on $\Theta_l$. Define
\[
	p_{l,\min}:=
	\min_{1\le k\le N_l}
	\frac{|\Theta_l\cap B(\theta_k^{(l)},\delta_r)|}{|\Theta_l|}.
\]
Assume that $\delta_r>0$ is chosen so that
\[
	\sum_{l=1}^2
	\left(L_l\delta_r\sum_{k=1}^{N_l}|c_k^{(l)}|\right)^2
	\le (\varepsilon_A/2)^2.
\]
Here $p_{l,\min}$ depends on this accuracy-driven radius $\delta_r$. Since the target parameters belong to the bounded sampling set $\Theta_l$, one has $p_{l,\min}>0$, although it may be small when $\delta_r$ is small or when some target parameters are close to the boundary of $\Theta_l$.
If the number of randomly sampled hidden parameters in the $l$-th subnetwork satisfies
\[
	M_l\ge p_{l,\min}^{-1}\ln\frac{2N_l}{\delta_p},
	\qquad l=1,2,
\]
then, with probability at least $1-\delta_p$, every deterministic target parameter lies within distance \(\delta_r\) of at least one sampled hidden parameter. By selecting such nearby sampled parameters, possibly with repetition, and by using the same output coefficients, one obtains randomized subnetworks such that
\[
	\|u^*-\chi_1u_{a_1}-\chi_2u_{a_2}\|_{H^s_{\rm br}}
	\le\varepsilon_A.
\]
\end{theorem}
\begin{proof}
For each subnetwork, the probability that a fixed target parameter ball is missed by $M_l$ independent samples is at most $(1-p_{l,\min})^{M_l}\le e^{-M_lp_{l,\min}}$. Hence the probability that some target ball in the $l$-th subnetwork is missed is bounded by $N_l e^{-M_lp_{l,\min}}\le\delta_p/2$. A union bound over $l=1,2$ gives probability at least $1-\delta_p$.

On this event, for each target parameter $\theta_k^{(l)}$ there exists a sampled parameter $\widetilde\theta_k^{(l)}$ such that $|\theta_k^{(l)}-\widetilde\theta_k^{(l)}|\le\delta_r$. We allow the same sampled hidden parameter to be used for more than one matched deterministic term; equivalently, repeated hidden parameters may be merged by summing their output coefficients. Thus no distinct matching argument is needed. Keeping the same output coefficients gives
\[
\begin{aligned}
\|u_N-u_a\|_{H^s_{\rm br}}^2
&=
\sum_{l=1}^2
\left\|
\sum_{k=1}^{N_l}c_k^{(l)}
\left(\phi_{\theta_k^{(l)}}-\phi_{\widetilde\theta_k^{(l)}}\right)
\right\|_{H^s(\Omega_l)}^2 \\
&\le
\sum_{l=1}^2
\left(L_l\delta_r\sum_{k=1}^{N_l}|c_k^{(l)}|\right)^2
\le (\varepsilon_A/2)^2.
\end{aligned}
\]
Therefore,
\[
\begin{aligned}
	\|u^*-u_a\|_{H^s_{\rm br}}
	\le
	\|u^*-u_N\|_{H^s_{\rm br}}+
	\|u_N-u_a\|_{H^s_{\rm br}}\le \varepsilon_A.
\end{aligned}
\]
\end{proof}

\subsection{Quadrature Consistency of the Empirical Broken Loss}

The abstract error estimate \eqref{eq:broken-error-decomposition} only requires the uniform consistency quantity
\[
	S_{\rm br}=
	\sup_{u\in\cN_{\rho,\mathbf N}^{\rm br}}
	|\cL_{\rm br}(u)-\widetilde\cL_{\rm br}(u)|.
\]
We next give a deterministic quadrature-consistency estimate for the practical broken loss used in the numerical method. The result is deliberately stated for an $L^2$-type residual and an $L^2(\Gamma)$-type interface surrogate. It should not be interpreted as a full sampling theorem for the negative Sobolev loss or for the strong trace norms $H^{3/2}(\Gamma)$ and $H^{1/2}(\Gamma)$ appearing in the abstract graph norm. The term empirical in this subsection means a finite quadrature or collocation approximation of a population loss. The estimate below is deterministic; it is not a probabilistic concentration bound for independently sampled points. A probabilistic version would require additional covering-number or concentration arguments for the sampled quadrature rules.

For a truncation radius $R>M$, set
\[
	\Omega_{1,R}:=\Omega_1,
	\qquad
	\Omega_{2,R}:=\Omega_2\cap B_R.
\]
For the practical interface residual, set
\[
\mathcal I_\Gamma(u)=
\begin{cases}
[\partial_nu], & \hbox{for a smooth artificial interface with matching coefficients},\\
\mathcal F_\Gamma(u), & \hbox{for a conormal transmission problem}.
\end{cases}
\]
Thus, in a discontinuous-coefficient transmission problem, the normal-derivative jump in the practical loss should be replaced by the conormal-flux residual. In both cases, \(\mathcal I_\Gamma\) is a linear interface residual. Hence, if the exact solution satisfies \([u^*]=0\) and \(\mathcal I_\Gamma(u^*)=0\), then \([u]=[u-u^*]\) and \(\mathcal I_\Gamma(u)=\mathcal I_\Gamma(u-u^*)\) in the error equation.

For clarity, we distinguish the inhomogeneous data-residual loss from the homogeneous error-residual loss. For a trial function $u=\chi_1u_1+\chi_2u_2$, define the truncated practical data loss
\begin{equation}\label{eq:truncated-practical-loss}
\begin{aligned}
	\cL_{{\rm data},R}^{\rm q}(u)
	:=
	\sum_{l=1}^2
	\int_{\Omega_{l,R}}|\cG_lu_l-f_l|^2\,\mathrm dx  +
	\int_\Gamma
	\Bigl(|[u]|^2+|\mathcal I_\Gamma(u)|^2\Bigr)\,\mathrm ds .
\end{aligned}
\end{equation}
The corresponding full-domain practical data loss is
\begin{equation}\label{eq:full-practical-data-loss}
\begin{aligned}
	\cL_{{\rm data}}^{\rm q}(u)
	:=
	\sum_{l=1}^2
	\int_{\Omega_l}|\cG_lu_l-f_l|^2\,\mathrm dx +
	\int_\Gamma
	\Bigl(|[u]|^2+|\mathcal I_\Gamma(u)|^2\Bigr)\,\mathrm ds .
\end{aligned}
\end{equation}
For an admissible error function $v$, define the associated homogeneous error loss
\begin{equation}\label{eq:practical-error-loss}
\begin{aligned}
	\cL_{{\rm err}}^{\rm q}(v)
	:=
	\sum_{l=1}^2
	\int_{\Omega_l}|\cG_lv_l|^2\,\mathrm dx +
	\int_\Gamma
	\Bigl(|[v]|^2+|\mathcal I_\Gamma(v)|^2\Bigr)\,\mathrm ds .
\end{aligned}
\end{equation}
The empirical quadrature version of the truncated practical data loss is
\begin{equation}\label{eq:empirical-practical-loss}
\begin{aligned}
	\widetilde\cL_{{\rm data},R}^{\rm q}(u)
	:=&
	\sum_{l=1}^2
	Q_l\bigl(|\cG_lu_l-f_l|^2\bigr)
	+
	Q_\Gamma\bigl(|[u]|^2+|\mathcal I_\Gamma(u)|^2\bigr)  \\
	=&
	\sum_{l=1}^2\sum_{q=1}^{N_l^{\rm q}}
	\omega_{l,q}|\cG_lu_l(x_{l,q})-f_l(x_{l,q})|^2  +
	\sum_{q=1}^{N_\Gamma}\omega_{\Gamma,q}
	\Bigl(|[u](x_{\Gamma,q})|^2+|\mathcal I_\Gamma(u)(x_{\Gamma,q})|^2\Bigr).
\end{aligned}
\end{equation}
Here $Q_l$ denotes a quadrature rule on $\Omega_{l,R}$ and $Q_\Gamma$ denotes a quadrature rule on the interface $\Gamma$.

\begin{assumption}[Far-field tail control]\label{assmp:tail-control}
There exists a function $E_{\rm tail}(R)\to0$ as $R\to\infty$ such that
\[
	\sup_{u\in\cN_{\rho,\mathbf N}^{\rm br}}
	\int_{\Omega_2\setminus B_R}|\cG_2u_2-f_2|^2\,\mathrm dx
	\le E_{\rm tail}(R).
\]
\end{assumption}

This assumption is needed because a finite quadrature rule only acts on a truncated computational region. It is a uniform tail assumption over the coefficient-bounded trial space. It is generally valid only for trial spaces with sufficient far-field decay, or after imposing an appropriate weighted-loss or truncation strategy; it is not guaranteed for non-decaying activations such as $\tanh$, unless the stated tail-control assumption can be verified. In particular, it also requires the forcing tail
	\[
	\int_{\Omega_2\setminus B_R}|f_2|^2\,\mathrm dx
	\]
	to vanish as $R\to\infty$. Hence it is a condition on both the far-field trial features and the data. If a mapped or weighted quadrature rule is used instead of a hard truncation, then $E_{\rm tail}(R)$ should be replaced by the corresponding mapping or weight-induced tail error.
	
\begin{assumption}[Uniform quadrature accuracy]\label{assmp:quad-accuracy}
For some integer $m_q\ge1$, the quadrature rules satisfy
\[
\left|
\int_{\Omega_{l,R}} g\,\mathrm dx-Q_l(g)
\right|
\le
C_{Q,l}(R)h_l^{m_q}\|g\|_{W^{m_q,\infty}(\Omega_{l,R})},
\qquad l=1,2,
\]
for all $g\in W^{m_q,\infty}(\Omega_{l,R})$, and
\[
\left|
\int_\Gamma q\,\mathrm ds-Q_\Gamma(q)
\right|
\le
C_{Q,\Gamma}h_\Gamma^{m_q}\|q\|_{W^{m_q,\infty}(\Gamma)}
\]
for all $q\in W^{m_q,\infty}(\Gamma)$. Here \(h_l\) and \(h_\Gamma\) denote the quadrature resolution parameters on \(\Omega_{l,R}\) and \(\Gamma\), respectively, such as mesh sizes or fill distances. The constants associated with the exterior truncated region, such as $C_{Q,2}(R)$, may depend on the truncation radius $R$.
\end{assumption}

 This is an assumption on the chosen deterministic quadrature or cubature rule. For randomly sampled collocation points, such an estimate must be replaced by a probabilistic uniform law of large numbers or a concentration bound.

\begin{assumption}[Uniform integrand regularity]
\label{assmp:loss-integrand-regularity}
There exist constants $C_{g,1}$, $C_{g,2}(R)$, and $C_{\Gamma}$ such that
\begin{align*}
\sup_{u\in\cN_{\rho,\mathbf N}^{\rm br}}
\left\||\cG_1u_1-f_1|^2\right\|_{W^{m_q,\infty}(\Omega_{1})}
&\le C_{g,1},\\
\sup_{u\in\cN_{\rho,\mathbf N}^{\rm br}}
\left\||\cG_2u_2-f_2|^2\right\|_{W^{m_q,\infty}(\Omega_{2,R})}
&\le C_{g,2}(R),
\end{align*}
and
\begin{equation*}
\sup_{u\in\cN_{\rho,\mathbf N}^{\rm br}}
\left\||[u]|^2+|\mathcal I_\Gamma(u)|^2\right\|_{W^{m_q,\infty}(\Gamma)}
\le C_{\Gamma}.
\end{equation*}
\end{assumption}

\begin{lemma}[A sufficient condition for uniform integrand regularity]\label{lem:uniform-integrand-regularity}
Assume that each second-order operator has the form
\[
	\cG_l v=-\nabla\cdot(\beta_l\nabla v)+\gamma_l v,
\]
with coefficients sufficiently smooth so that
\[
	\beta_l\in W^{m_q+1,\infty}(\Omega_{l,R}),\qquad\gamma_l\in W^{m_q,\infty}(\Omega_{l,R}),
	\qquad
	f_l\in W^{m_q,\infty}(\Omega_{l,R}).
\]
Assume also that the bounded parameter sets and the activation function satisfy the uniform feature bounds
\[
	\sup_{\theta\in\Theta_l}\|\phi_\theta\|_{W^{m_q+2,\infty}(\Omega_{l,R})}
	\le C_{\phi,l}(R),
\]
and, on the interface,
\[
	\sup_{\theta\in\Theta_l}\|\phi_\theta\|_{W^{m_q+1,\infty}(\Gamma)}
	\le C_{\phi,\Gamma,l}.
\]
In the conormal case $\mathcal I_\Gamma(u)=\mathcal F_\Gamma(u)$, assume also that the coefficient traces entering the conormal flux are bounded in $W^{m_q,\infty}(\Gamma)$; for smooth interfaces this follows from the stated interior regularity of $\beta_l$.

Then Assumption~\ref{assmp:loss-integrand-regularity} holds, with constants depending on the coefficient bounds $C_{N_l}$, the parameter bounds, the activation bounds, the PDE coefficients, $f_l$, and the geometry, but not on the particular trial function $u$.
\end{lemma}
\begin{proof}
For $u_l=\sum_k c_k^{(l)}\phi_{\theta_k^{(l)}}$ and
$\sum_k|c_k^{(l)}|\le C_{N_l}$, the feature bound gives
\[
	\|u_l\|_{W^{m_q+2,\infty}(\Omega_{l,R})}
	\le
	C_{N_l}C_{\phi,l}(R).
\]
The assumed $W^{m_q+1,\infty}$-smoothness of $\beta_l$ and $W^{m_q,\infty}$-smoothness of $\gamma_l$ imply, by standard product estimates in $W^{m_q,\infty}$,
\[
	\|\cG_lu_l\|_{W^{m_q,\infty}(\Omega_{l,R})}
	\le C\,C_{N_l}C_{\phi,l}(R).
\]
Hence
\[
	\|\cG_lu_l-f_l\|_{W^{m_q,\infty}(\Omega_{l,R})}
	\le C\bigl(C_{N_l}C_{\phi,l}(R)+\|f_l\|_{W^{m_q,\infty}(\Omega_{l,R})}\bigr).
\]
Since $W^{m_q,\infty}$ is an algebra for integer $m_q\ge1$, we obtain the subdomain bounds in the form used in Assumption~\ref{assmp:loss-integrand-regularity}. For the interior subdomain, the constants are independent of the exterior truncation radius and
\[
	\left\||\cG_1u_1-f_1|^2\right\|_{W^{m_q,\infty}(\Omega_1)}
	\le C_{g,1}.
\]
For the truncated exterior subdomain, the constants may depend on $R$ and
\[
	\left\||\cG_2u_2-f_2|^2\right\|_{W^{m_q,\infty}(\Omega_{2,R})}
	\le C_{g,2}(R).
\]
The interface estimate follows similarly from the trace feature bounds on $\Gamma$ and the coefficient bounds:
\[
	\|[u]\|_{W^{m_q,\infty}(\Gamma)}+
	\|\mathcal I_\Gamma(u)\|_{W^{m_q,\infty}(\Gamma)}
	\le C,
\]
and another use of the algebra property gives the asserted bound for the interface integrand.
\end{proof}

\begin{theorem}[Quadrature consistency of the empirical broken loss]\label{thm:quadrature-consistency}
Under Assumptions~\ref{assmp:tail-control}, \ref{assmp:quad-accuracy}, and \ref{assmp:loss-integrand-regularity}, the practical empirical broken loss satisfies
\begin{equation}\label{eq:quadrature-consistency-bound}
\begin{aligned}
&\sup_{u\in\cN_{\rho,\mathbf N}^{\rm br}}
\left|
\cL_{{\rm data}}^{\rm q}(u)-\widetilde\cL_{{\rm data},R}^{\rm q}(u)
\right|  \\
&\qquad\le
C_{Q,1}C_{g,1}h_1^{m_q}
+
C_{Q,2}(R)C_{g,2}(R)h_2^{m_q}
+
C_{Q,\Gamma}C_\Gamma h_\Gamma^{m_q}
+E_{\rm tail}(R),
\end{aligned}
\end{equation}
where
\[
	\cL_{{\rm data}}^{\rm q}(u)
	:=
	\sum_{l=1}^2
	\int_{\Omega_l}|\cG_lu_l-f_l|^2\,\mathrm dx
	+
	\int_\Gamma
	\Bigl(|[u]|^2+|\mathcal I_\Gamma(u)|^2\Bigr)\,\mathrm ds.
\]
The constants associated with the truncated exterior region may depend on $R$. Therefore convergence of the right-hand side requires balancing the quadrature resolution $h_2$ on $\Omega_{2,R}$ with the far-field tail $E_{\rm tail}(R)$.
\end{theorem}
\begin{proof}
For each $u\in\cN_{\rho,\mathbf N}^{\rm br}$, write
\[
	g_l(u)=|\cG_lu_l-f_l|^2,
	\qquad
	q_\Gamma(u)=|[u]|^2+|\mathcal I_\Gamma(u)|^2.
\]
Then
\[
\begin{aligned}
&\left|
\cL_{{\rm data}}^{\rm q}(u)-\widetilde\cL_{{\rm data},R}^{\rm q}(u)
\right| \\
&\le
\sum_{l=1}^2
\left|
\int_{\Omega_{l,R}}g_l(u)\,\mathrm dx-Q_l(g_l(u))
\right|
+
\left|
\int_\Gamma q_\Gamma(u)\,\mathrm ds-Q_\Gamma(q_\Gamma(u))
\right| \\
&\quad+
\int_{\Omega_2\setminus B_R}|\cG_2u_2-f_2|^2\,\mathrm dx .
\end{aligned}
\]
Assumptions~\ref{assmp:quad-accuracy} and \ref{assmp:loss-integrand-regularity} bound the two subdomain quadrature terms and the interface quadrature term by
\[
C_{Q,1}C_{g,1}h_1^{m_q}
+
C_{Q,2}(R)C_{g,2}(R)h_2^{m_q}
+
C_{Q,\Gamma}C_\Gamma h_\Gamma^{m_q},
\]
while Assumption~\ref{assmp:tail-control} bounds the far-field tail by $E_{\rm tail}(R)$ uniformly over the coefficient-bounded trial space. Taking the supremum over $u$ gives \eqref{eq:quadrature-consistency-bound}.
\end{proof}

Since the exact solution satisfies the practical strong-form and transmission
conditions, the practical data loss and the homogeneous error loss are related
by
\[
\cL_{{\rm data}}^{\rm q}(u)
=
\cL_{{\rm err}}^{\rm q}(u-u^*)
\qquad
\text{for all }u\in\cN_{\rho,\mathbf N}^{\rm br}.
\]
Applying the same least-squares argument as in
\eqref{eq:broken-error-decomposition}, but now to the practical loss
\(\cL_{\rm data}^{\rm q}\), and then using the quadrature-consistency bound
\eqref{eq:quadrature-consistency-bound}, yields the following consequence.

	\begin{corollary}
	\label{cor:ls-quadrature-error}
	Assume that the operators \(\cG_l\) and the interface operator
	\(\mathcal I_\Gamma\) are linear, and that the exact solution satisfies
	\[
	\cG_lu_l^*=f_l\quad \hbox{in }\Omega_l,\qquad l=1,2,
	\qquad [u^*]=0,
	\qquad \mathcal I_\Gamma(u^*)=0 .
	\]
	Assume also that, for every \(u\in\cN_{\rho,\mathbf N}^{\rm br}\), the error
	\(u-u^*\) belongs to the admissible error space, and that the practical error
	loss
	\[
	\cL_{{\rm err}}^{\rm q}(v)
	:=
	\sum_{l=1}^2\|\cG_lv_l\|_{L^2(\Omega_l)}^2
	+
	\int_\Gamma
	\Bigl(|[v]|^2+|\mathcal I_\Gamma(v)|^2\Bigr)\,\mathrm ds
	\]
	satisfies the stability estimate
	\[
	C_L\|v\|_V^2
	\le
	\cL_{{\rm err}}^{\rm q}(v)
	\le
	C_U\|v\|_V^2
	\qquad
	\text{for all admissible }v .
	\]
	Let \(u_\rho\in\cN_{\rho,\mathbf N}^{\rm br}\) be a constrained, regularized,
	or a posteriori coefficient-bounded least-squares solution satisfying
	\[
	\widetilde\cL_{{\rm data},R}^{\rm q}(u_\rho)
	\le
	\inf_{u\in\cN_{\rho,\mathbf N}^{\rm br}}
	\widetilde\cL_{{\rm data},R}^{\rm q}(u)
	+
	\eta_{\rm opt},
	\]
	where the infimum and the consistency supremum in
	Theorem~\ref{thm:quadrature-consistency} are taken over the same fixed
	coefficient-bounded trial space. Then, under
	Assumptions~\ref{assmp:tail-control}, \ref{assmp:quad-accuracy}, and
	\ref{assmp:loss-integrand-regularity},
	\[
	\begin{aligned}
		\|u_\rho-u^*\|_V^2
		\le&
		\frac{C_U}{C_L}
		\inf_{u_a\in\cN_{\rho,\mathbf N}^{\rm br}}
		\|u_a-u^*\|_V^2 \\
		&+
		\frac{2}{C_L}
		\Bigl(
		C_{Q,1}C_{g,1}h_1^{m_q}
		+
		C_{Q,2}(R)C_{g,2}(R)h_2^{m_q}
		+
		C_{Q,\Gamma}C_\Gamma h_\Gamma^{m_q}
		+
		E_{\rm tail}(R)
		\Bigr)
		+
		\frac{1}{C_L}\eta_{\rm opt}.
	\end{aligned}
	\]
\end{corollary}

The stability assumption in Corollary~\ref{cor:ls-quadrature-error} is a
stability assumption for the practical loss itself. It should not be read as an
automatic consequence of a stability theory formulated with stronger trace
norms or negative Sobolev residual norms. For example, pointwise or
\(L^2(\Gamma)\)-type interface penalties do not, by themselves, provide
uniform control of strong trace quantities such as
\(H^{3/2}(\Gamma)\), \(H^{1/2}(\Gamma)\), or
\(H^{-1/2}(\Gamma)\) flux norms. Similarly, a collocation residual should not
be interpreted as a direct discretization of
\(H^{-1}(\Omega_1)\times H^{-1}(\Omega_2)\) unless a corresponding
duality-based, spectral, or mesh-dependent realization is specified.

Thus the preceding theorem and corollary establish a deterministic consistency
and error estimate for the computable \(L^2\)-type practical broken loss used
in the coefficient-bounded least-squares analysis. If one wishes to work with
an abstract broken \(H^2\) or broken \(H^1\) graph norm involving stronger
trace, flux, or negative Sobolev terms, an additional surrogate-stability,
finite-dimensional inverse, or sampling estimate is required to connect that
abstract stability theory to the practical empirical loss. The conclusion is
therefore conditional but precise: once the practical loss is stable on the
admissible error space and the empirical consistency gap is controlled on the
same coefficient-bounded trial space, the least-squares DD--RaNN solution satisfies
the broken-space error estimate. The approximation term is controlled by
Theorem~\ref{Thm:H1bound-UAT} when the stability norm \(V\) coincides with the
corresponding broken \(H^s\) norm; for stronger graph norms involving
additional trace or flux quantities, approximation in that stronger \(V\)-norm
must be imposed separately.

		%%%%%%%%%%%%%%%%%%%%%%%%%%%%%
	\section{Numerical Examples}
	\label{sec:num}

In this section, we test the proposed domain-decomposed RaNN methods on several representative unbounded-domain problems. Unless otherwise specified, the degrees of freedom are defined as the total number of randomized features over all subnetworks. Some numerical tests use activations such as \(\tanh\),  which have finite one-dimensional limits at infinity but do not decay to zero. These experiments are intended to assess the practical performance and robustness of the DD-RaNN construction, and should not be interpreted as direct instances of the decaying-activation approximation theorem proved in Section~\ref{sec:analysis}.

When considering the \(L^2\) and \(H^1\) errors, the error calculation is performed using the Laguerre-Gaussian quadrature formula and the Hermite-Gaussian quadrature formula for the half-unbounded and unbounded domains. When \(H^1\) errors are reported for domain-decomposed solutions, they should be understood as broken \(H^1\) errors, computed by summing the subdomain \(H^1\) contributions. Interface mismatches are controlled through the least-squares interface equations and are not included in the broken \(H^1\) norm unless stated otherwise.

	In some cases, we use the mapped error which is calculated as follows:  
	For the bounded domain \([-1,1]\), the Legendre-Gauss quadrature formula is given by:
	\[
	\int_{-1}^{1} f(x) \, \mathrm{d}x \approx \sum_{k=0}^{n} f(x_k)w_k.
	\]
	Set $x=\dfrac{s\,y}{\sqrt{1-y^2}}$ so that $\dfrac{\mathrm{d}x}{\mathrm{d}y}=\dfrac{s}{(1-y^2)^{3/2}}$, where $ s $ is a scaling parameter, often set to $1$. Then
	\[
	\int_{-\infty}^{+\infty} f(x)\,\mathrm{d}x
	=\int_{-1}^{1} f\!\left(\frac{s\,y}{\sqrt{1-y^2}}\right)\frac{s}{(1-y^2)^{3/2}}\,\mathrm{d}y.
	\]
	In 2D,
	\begin{equation}
			\iint_{\mathbb{R}^2} f(x,y)\,\mathrm{d}x\,\mathrm{d}y
		\approx\sum_{j,k} f\!\left(\frac{s\,p_k}{\sqrt{1-p_k^2}},\frac{s\,q_j}{\sqrt{1-q_j^2}}\right)
		\frac{s^2}{(1-p_k^2)^{3/2}(1-q_j^2)^{3/2}}\,w_jw_k,
		\label{mapped-error}
	\end{equation}
	where $p_k,\ q_j$ are Legendre-Gauss points, and $ w_k,\ w_j $ are the corresponding Gaussian weights.

	The use of mapped error is due to the fact that, in certain cases, with the same number of Gaussian points, the point distribution resulting from the mapping technique can cover a larger area, which can demonstrate the accuracy of the solution of RaNNs. A similar integral computation scheme of $L^2$ error and $H^1$ error can be given for higher dimensions. Thus, taking \( \mathbb{R}^2 \) as an example, we have two ways to compute the \( L^2 \) and \(H^1\) errors, respectively:
	\paragraph{Mapped Legendre on $\mathbb{R}^2$.}
	Let $e= u_h-u$. With the mapping in \eqref{mapped-error},
	\[
	\|e\|_{L^2(\mathbb{R}^2)}^2 \approx 
	\sum_{j,k} e\!\left(\frac{s\,p_k}{\sqrt{1-p_k^2}},\frac{s\,q_j}{\sqrt{1-q_j^2}}\right)^{\!2}
	\frac{s^2}{(1-p_k^2)^{3/2}(1-q_j^2)^{3/2}}\,w_kw_j,
	\]
	and analogously for $\|e\|_{H^1}$ with $|e|^2+ |\nabla e|^2$.
	\paragraph{Hermite--Gauss on $\mathbb{R}^2$.}
	Let $e= u_h-u$. Then
	\[
	\|e\|_{L^2(\mathbb{R}^2)}^2 \approx \sum_{j,k} |e(p_k,q_j)|^2\,w_kw_j,\qquad
	\|e\|_{H^1(\mathbb{R}^2)}^2 \approx \sum_{j,k} \bigl(|e|^2+|\nabla e|^2\bigr)(p_k,q_j)\,w_kw_j,
	\]
	where $p_k,\ q_j$ are Hermite-Gauss points, and $ w_k,\ w_j $ are the corresponding Gaussian weights, where the compensation for the Hermite weight function $e^{-x^2}$ has already been incorporated.
	
	The computation method depends on the size of the coverage interval of the current sampling points. When there is a significant difference between the two computed results, we adopt the method with a larger sampling range to compute the error, which is usually the method that yields a larger error. This allows for a better assessment of the computational results in the unbounded domain.
	
	When computing errors on a semi-unbounded domain, similarly, we use the Laguerre-Gaussian quadrature formula:
		\paragraph{Laguerre--Gauss on $\mathbb{R}^2_{+}$.}
		Let $e= u_h-u$. Then
	\begin{equation}
		\|e\|_{L^2(\mathbb{R}^2_{+})}^2 \approx \sum_{j,k} e(p_k,q_j)^2\,w_kw_j,\qquad
		\|e\|_{H^1(\mathbb{R}^2_{+})}^2 \approx \sum_{j,k} \bigl(e^2+|\nabla e|^2\bigr)(p_k,q_j)\,w_kw_j,
	\end{equation}
	where $p_k,\ q_j$ are Laguerre-Gaussian points, and $ w_k,\ w_j $ are the corresponding Gaussian weights, where the compensation for the Laguerre weight function $e^{-x}$ has already been incorporated.

	\begin{example}
		In this example, we consider the equation $(\ref{eq:half})$ with $\Omega=[0,+\infty)$, $\beta=1$ and $\gamma=1$. The exact solution is 
		\begin{equation}
			u(x)=\dfrac{1}{x+1},
		\end{equation}
		and the source term $f$ can be set by calculation. The Dirichlet boundary condition is given as
		\begin{equation}
			u(0)=1.
		\end{equation}
		\label{eg:1d-half}
	\end{example}

	This example is used to test whether the proposed near--far DD-RaNN-PG decomposition can approximate algebraically decaying solutions on a semi-unbounded interval. It also provides a comparison with a classical Laguerre spectral Galerkin discretization using a comparable number of degrees of freedom.

	In Example~\ref{eg:1d-half}, we use the DD-RaNN-PG method with two single-hidden-layer RaNNs. 
	For a single-hidden-layer RaNN, the weight matrix $\bW$ and the bias matrix $\bb$ are sampled from a uniform distribution $U(-r,r)$. The activation function is \(\rho(x)=\tanh(x)\), and the near--far interface $x_{\mathrm{cut}}$ is 45 and scaling parameter $S$ is 10. The activation \(\tanh\) does not satisfy the decaying activation assumption used in the conditional approximation theorem of Section~\ref{sec:analysis}. Hence the following computation should be interpreted as a numerical demonstration of the DD-RaNN strategy, rather than as a direct instance covered by that theorem.
	
	For the near-field subnetwork and the far-field subnetwork, the parameter \(r_1 = 0.7\) and \(r_2 = 0.2\). The number of neurons \(m_1\) and \(m_2\) vary from 10 to 200. Unless otherwise noted, we define the degrees of freedom (DoF) as $m_1+m_2$. The number of points used to construct the test function is the same as the number of neurons. The errors shown in Table~\ref{tab:1d-half-errors} and Figure~\ref{fig:1d-half} are directly computed using the Laguerre-Gaussian quadrature formula with 180 points, and we use 5 quadrature points per element.
	\begin{figure}[H]
		\centering
		\begin{minipage}{0.45\linewidth}
			\centering
			\includegraphics[width=\textwidth]{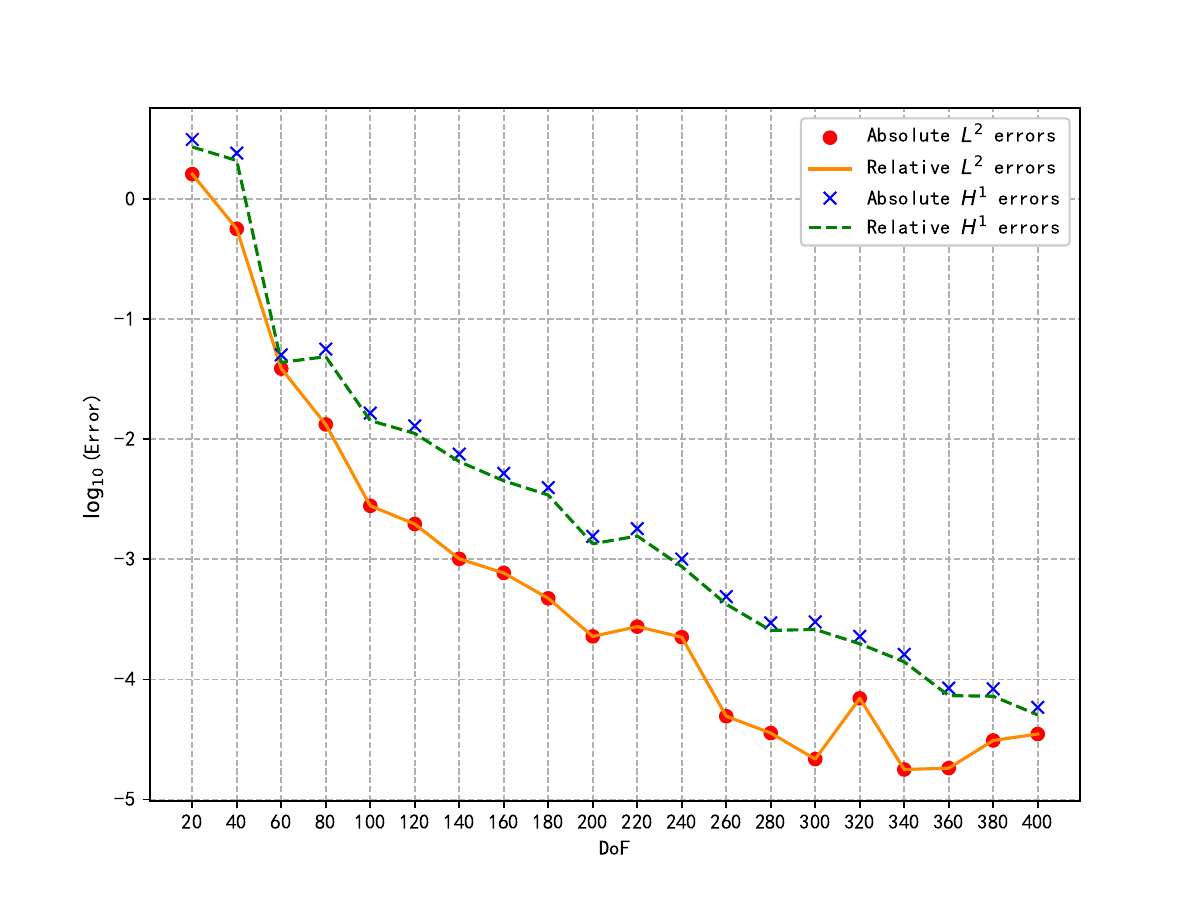}
		\end{minipage}
		\hspace{0.05\linewidth} 
		\begin{minipage}{0.45\linewidth}
			\centering
			\includegraphics[width=\textwidth]{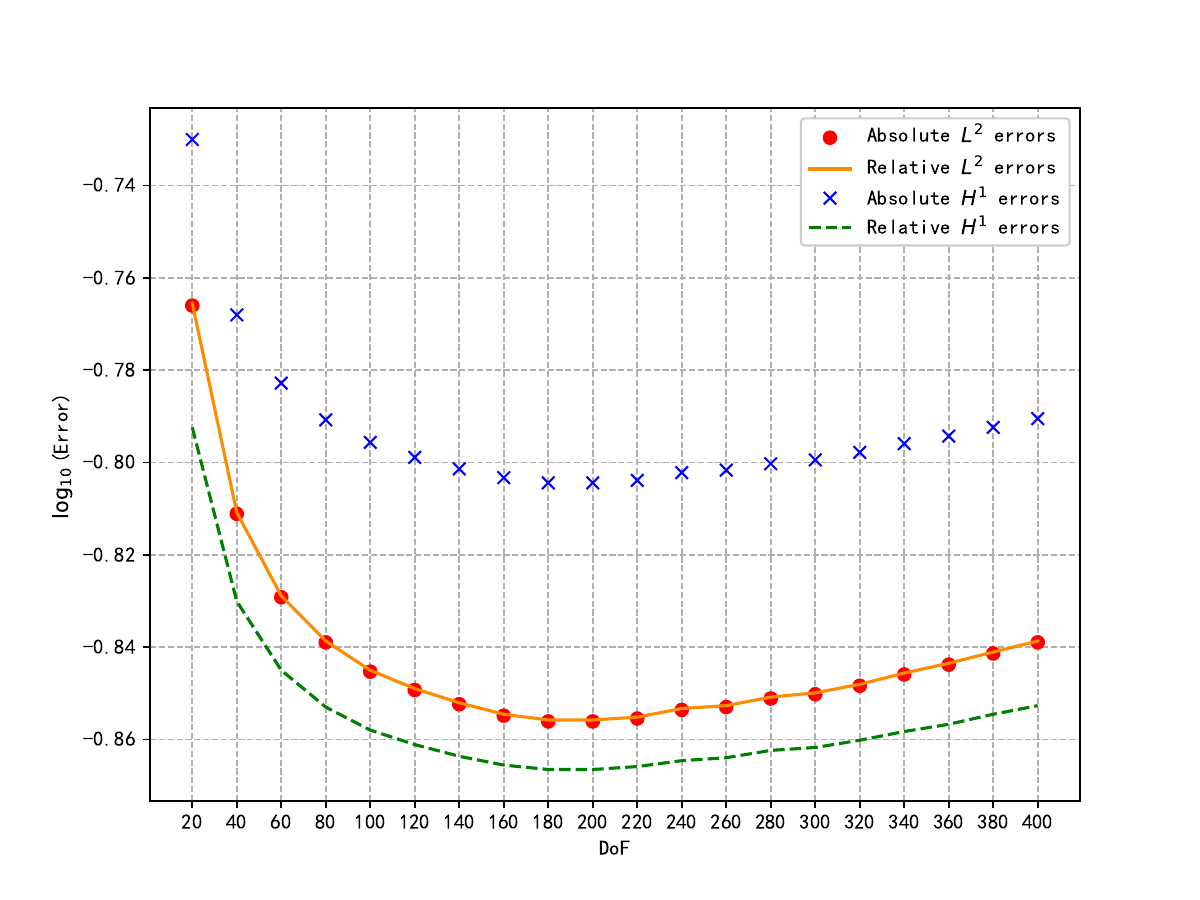}
		\end{minipage}
		\caption{Logarithmic errors of the DD-RaNN–PG (left) and the Laguerre-Spectral-Galerkin method (right) for Example~\ref{eg:1d-half}.}
		\label{fig:1d-half}
	\end{figure}

	As shown in Figure~\ref{fig:1d-half}, when $m_1$ and $m_2$ increase simultaneously, the DD-RaNN-PG error decreases significantly. Eventually, all four types of errors stabilize at the order of $10^{-5}$.
	
	On the other hand, the Laguerre-Spectral-Galerkin method, which is naturally constructed to satisfy the boundary conditions, is formulated as 
	\[
	u = \hat{L}_1 + \sum_{k=1}^{\text{DoF}} \tilde{u}_k (\hat{L}_{k+1} - \hat{L}_k),
	\]
	where $\hat{L}_k$ represents the $k$-th order Laguerre function. By matching the number of basis functions to that of DD-RaNN-PG and setting the number of Gauss points for calculating the global error to 180, the results are shown in the right panel of Figure \ref{fig:1d-half}. It can be observed that as the DoF increases, the convergence rate of the spectral method slows down gradually, and eventually, all four types of errors stabilize at the order of $10^{-1}$.

	\begin{table}[H]
		\centering
		\caption{Relative $L^2$ and $H^1$ errors for Example~\ref{eg:1d-half} using DD-RaNN-PG and Laguerre-Spectral-Galerkin methods}
		\label{tab:1d-half-errors}
		\begin{tabular}{ccccc}
			\hline
			& \multicolumn{2}{c}{DD-RaNN-PG (2 networks)} & \multicolumn{2}{c}{Laguerre-Spectral-Galerkin} \\ \cline{2-5} 
			$DoF$ & Relative $L^2$ Error      & Relative $H^1$ Error     & Relative $L^2$ Error         & Relative $H^1$ Error        \\ \hline
			20  & $1.6157$ & $2.7195$ & $1.7160\times10^{-1}$ & $1.6130\times10^{-1}$ \\
			80  & $1.3286\times10^{-2}$ & $4.8716\times10^{-2}$ & $1.4500\times10^{-1}$ & $1.4030\times10^{-1}$ \\
			140 & $1.0085\times10^{-3}$ & $6.5099\times10^{-3}$ & $1.4060\times10^{-1}$ & $1.3690\times10^{-1}$ \\
			200 & $2.2786\times10^{-4}$ & $1.3465\times10^{-3}$ & $1.3940\times10^{-1}$ & $1.3600\times10^{-1}$ \\
			260 & $4.9253\times10^{-5}$ & $4.2275\times10^{-4}$ & $1.4040\times10^{-1}$ & $1.3680\times10^{-1}$ \\
			320 & $6.9469\times10^{-5}$ & $1.9778\times10^{-4}$ & $1.4190\times10^{-1}$ & $1.3800\times10^{-1}$ \\
			380 & $3.0954\times10^{-5}$ & $7.2287\times10^{-5}$ & $1.4420\times10^{-1}$ & $1.3980\times10^{-1}$ \\
			\hline
		\end{tabular}
	\end{table}

	From Table~\ref{tab:1d-half-errors}, it can be further observed that for Example~\ref{eg:1d-half}, increasing the number of basis functions leads to a slow decrease in the spectral method's error, and its accuracy remains relatively poor. In contrast, the DD-RaNN-PG achieves better accuracy as the number of neurons increases.

	\begin{example}
		In this example, we consider the equation $(\ref{eq:half})$ with $\beta=1$ and $\gamma=1$. The domain $\Omega$ is defined as:
		\begin{equation}
			\begin{aligned}
				&\Omega=\Omega_{D1}\setminus (\Omega_{D2}\cup\Omega_N),\\
				&\Omega_{D1}=\{(x,y);\ x\ge0,\ y\ge0\},\ \Gamma_{D1}=\partial \Omega_{D1},\\
				&\Omega_{D2}=\left\{(x,y);\ \frac{(x+y-9)^2}{9}+(y-x+1)^2\le 2 \right\},\ \Gamma_{D2}=\partial \Omega_{D2},\\
				&\Omega_N=\{(x,y);\ (x-1)^2+(y-1)^2\le 1\},\ \Gamma_N=\partial \Omega_N.
			\end{aligned}
		\end{equation}
		The exact solution is 
		\begin{equation}
			u(x,y)=\exp\left(-\frac{x}{2}-\frac{y}{2}\right)\sin(x+y),
		\end{equation}
		and the source term $f$ can be set by calculation. The boundary conditions are given as
		\begin{equation}
			\begin{cases}
				u(\bx)=g_1(\bx),\qquad \bx\in \Gamma_D=\Gamma_{D1}\cup\Gamma_{D2},\\
				\nabla u(\bx)\cdot \bn=g_2(\bx),\qquad \bx\in \Gamma_N,\\
			\end{cases}
		\end{equation}
		where $g_1$ and $g_2$ can be set by calculation.
		\label{eg:2d-half}
	\end{example}

	This example is designed to examine the geometric flexibility of the proposed DD-RaNN-CM on a semi-unbounded domain with both interior holes and mixed boundary conditions. The purpose is to test whether the domain-decomposed randomized trial space can handle a near-field region with complicated geometry together with an exterior unbounded region.

	In Example~\ref{eg:2d-half}, we use the DD-RaNN-CM with two single-hidden-layer RaNNs. For a single-hidden-layer RaNN, the weight matrix $\bW$ is sampled from the uniform distribution $U(-r,r)$ and the bias matrix $\bb$ is constructed by the second strategy:
	\[
	\bb =(- \bW \odot \mathbf{R})\mathbf{1}, \quad \mathbf{R} \sim U(-1, 1),
	\]
	where \(\mathbf{R}\) is a sequence of numbers sampled from the uniform distribution \(U(-1, 1)\) entry-wise, and \(\odot\) represents the element-wise multiplication, and $\mathbf{1}$ is the all-ones matrix. The activation function is \(\rho(x,y)=e^{-x^2-y^2}\).
	
	In this example, the near--far interface is a quarter-circle boundary with a radius of 18 centered at the origin. For the near-field subnetwork (NFs), the parameter \(r_1 = 0.5\), and the number of neurons \(m_1\) varies from 50 to 450. For the far-field subnetwork, the parameter \(r_2 = 1\), and the number of neurons \(m_2 = 50\). Unless otherwise noted, ``DoF'' is defined as $m_1+m_2$. Each boundary is sampled with 500 points, and the interior of both networks is sampled with 4000 points. When assembling the linear equations, the penalty parameters $\xi  $ and $ \eta $ are set to 10.

	\begin{figure}[H]
		\centering
		\begin{minipage}{0.45\linewidth}
			\centering
			\includegraphics[width=\textwidth]{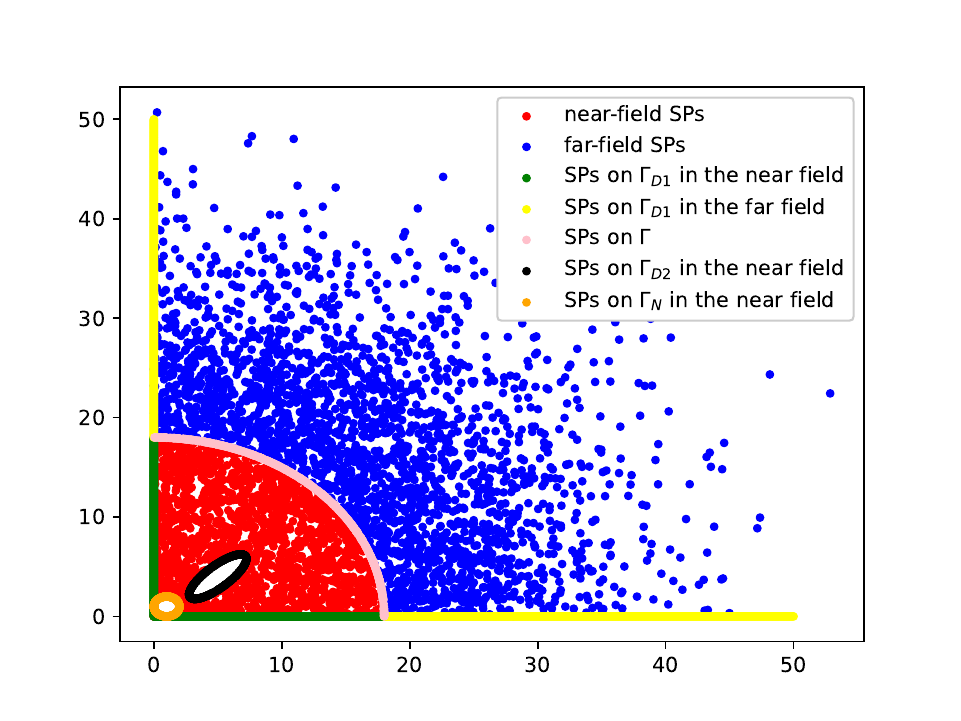}
			\label{fig:2d-half-sample}
		\end{minipage}
		\hspace{0.05\linewidth} 
		\begin{minipage}{0.45\linewidth}
			\centering
			\includegraphics[width=\textwidth]{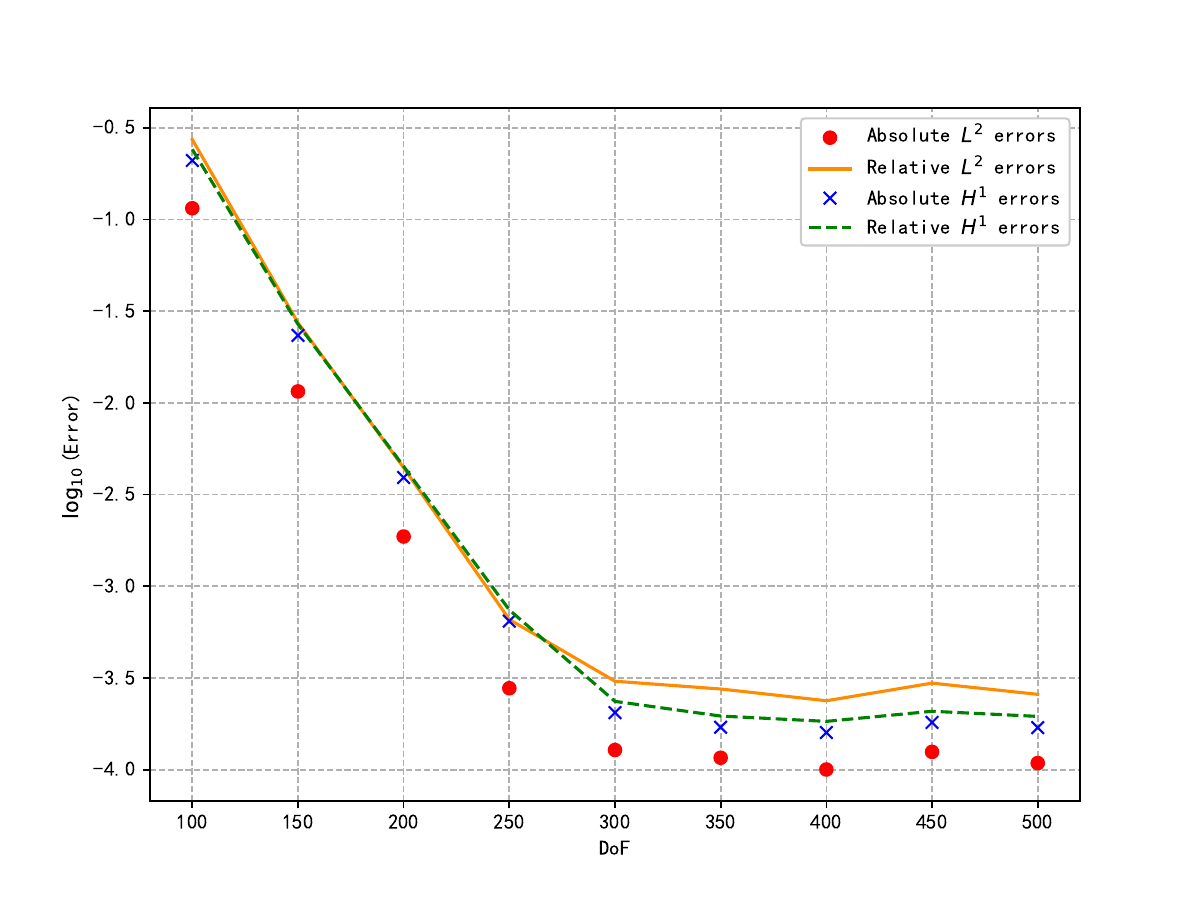}
			\label{fig:2d-half-error}
		\end{minipage}
		\caption{Scatter plot of sampling points (SPs) in DD-RaNN-CM (left) and logarithmic errors for DD-RaNN-CM (right) for Example~\ref{eg:2d-half}.}
		\label{fig:2d-half}
	\end{figure}	
	
	The left figure in Figure~\ref{fig:2d-half} illustrates the point distribution from one of the experiments. The interior points are sampled in polar coordinates. The radial distance is sampled from the normal distribution \(\mathcal{N}(20, 10)\), and the angle is sampled from the uniform distribution \(U(0, \pi/2)\). If the radius is negative, its absolute value is taken. Each sampled point is checked to determine whether it falls within an interior hole; if so, the point is resampled.

	\begin{figure}[H]
		\centering
		\includegraphics[width=1\textwidth]{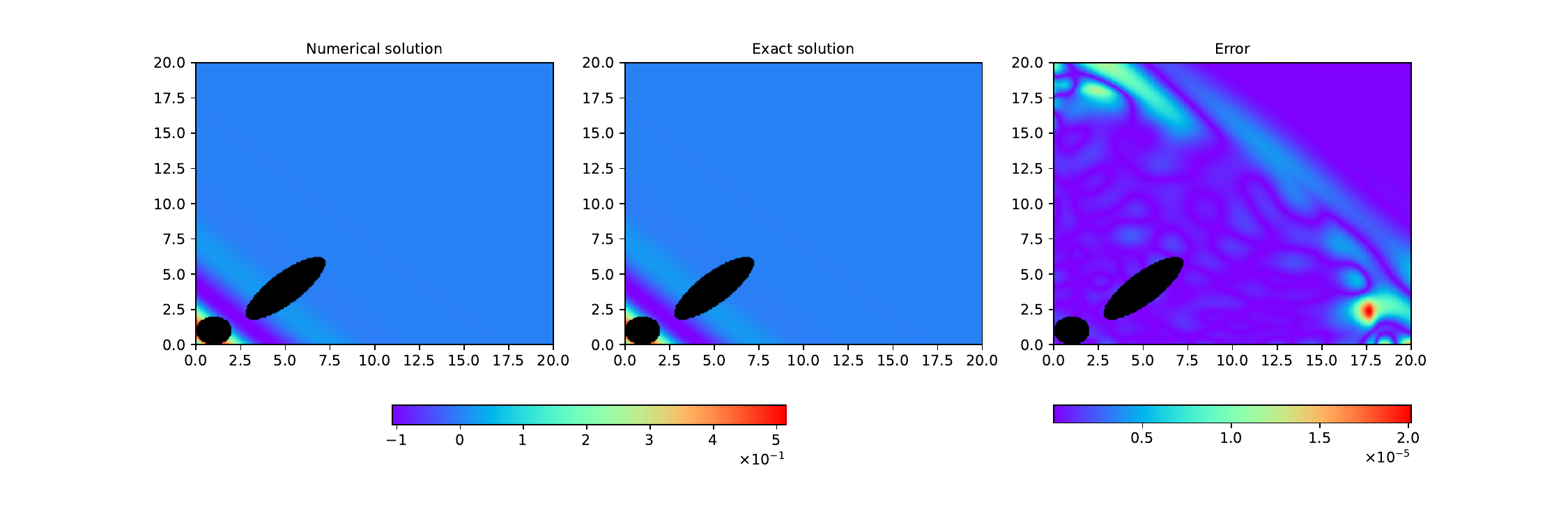}
		\caption{Computational solution obtained using the DD-RaNN-CM (left), exact solution (middle), and absolute error plot (right) for $m_1 = 450$ in Example~\ref{eg:2d-half}.}
		\label{fig:2d-half-sol}
	\end{figure}

	The elliptical boundary and circular boundary are also sampled in polar coordinates, with angles sampled from the uniform distribution \(U(0, 2\pi)\). For the quarter-circle boundary, angles are sampled from the uniform distribution \(U(0, \pi/2)\). On the boundaries along the coordinate axes, points are uniformly distributed within \([0, 50]\).
	
	The errors presented in Table~\ref{tab:2d-half-errors} and Figure~\ref{fig:2d-half} are the averages obtained from 30 experiments for each parameter, while Figure~\ref{fig:2d-half-sol} shows the results of a single experiment. The error for each experiment is directly computed using the Laguerre-Gaussian quadrature formula with 50 points. We can observe that as the number of neurons in RaNN1 increases, the error gradually decreases and eventually stabilizes at the order of \(10^{-4}\).

	\begin{table}[H]
		\centering
	\caption{$L^2$ error, relative $L^2$ error, $H^1$ error, and relative $H^1$ error for different numbers of neurons in Example~\ref{eg:2d-half}.}
		\label{tab:2d-half-errors}
		\begin{tabular}{c|cccc}
			\toprule
			\textbf{$ DoF $} & $L^2$ Error & Relative $L^2$ Error & $H^1$ Error & Relative $H^1$ Error \\
			\midrule
			100 & $1.1508\times10^{-1}$ & $2.7280\times10^{-1}$ & $2.0963\times10^{-1}$ & $2.4116\times10^{-1}$ \\
			150 & $1.1544\times10^{-2}$ & $2.7365\times10^{-2}$ & $2.3338\times10^{-2}$ & $2.6847\times10^{-2}$ \\
			200 & $1.8685\times10^{-3}$ & $4.4292\times10^{-3}$ & $3.9191\times10^{-3}$ & $4.5084\times10^{-3}$ \\
			250 & $2.7824\times10^{-4}$ & $6.5956\times10^{-4}$ & $6.4622\times10^{-4}$ & $7.4340\times10^{-4}$ \\
			300 & $1.2818\times10^{-4}$ & $3.0384\times10^{-4}$ & $2.0511\times10^{-4}$ & $2.3596\times10^{-4}$ \\
			350 & $1.1627\times10^{-4}$ & $2.7563\times10^{-4}$ & $1.7061\times10^{-4}$ & $1.9627\times10^{-4}$ \\
			400 & $1.0034\times10^{-4}$ & $2.3786\times10^{-4}$ & $1.5957\times10^{-4}$ & $1.8356\times10^{-4}$ \\
			450 & $1.2516\times10^{-4}$ & $2.9670\times10^{-4}$ & $1.8114\times10^{-4}$ & $2.0838\times10^{-4}$ \\
			500 & $1.0873\times10^{-4}$ & $2.5775\times10^{-4}$ & $1.6988\times10^{-4}$ & $1.9542\times10^{-4}$ \\
			\bottomrule
		\end{tabular}
	\end{table}

	\begin{example}
		In this example, we consider the equation $(\ref{eq:GF})$ with $\beta=1$ and $\gamma=1$. The domain $\Omega$ is set as:
		\begin{equation}
			\begin{aligned}
				&\Omega=\mathbb{R}^2\setminus(\Omega_{D}\cup\Omega_N),\\
				&\Omega_{D}=\left\{(x,y);\frac{(x+y-9)^2}{9}+(y-x+1)^2\le 2 \right\},\ &\Gamma_{D}=\partial \Omega_{D},\\
				&\Omega_N=\left\{(x,y);(x+2)^2+(y+2)^2\le 4\right\},\ \Gamma_N=\partial \Omega_N.
			\end{aligned}
		\end{equation}
		The exact solution is 
		\begin{equation}
			u(x,y)=\exp\left(-\dfrac{x^2}{10}-\dfrac{y^2}{10}\right)\sin(x+y)\sin(x-y),
		\end{equation}
		and the source term $f$ can be set by calculation. The boundary conditions are given as
		\begin{equation}
			\begin{cases}
				u(\bx)=g_1(\bx),\qquad \bx\in \Gamma_D,\\
				\nabla u(\bx)\cdot \bn=g_2(\bx),\qquad \bx\in \Gamma_N,\\
			\end{cases}
		\end{equation}
		where $g_1$ and $g_2$ can be set by calculation.
		\label{eg:2d-whole}
	\end{example}

	This example tests the fully unbounded two-dimensional setting. In contrast to the semi-unbounded case, there is no physical outer boundary. The exterior network is therefore used to represent the far-field behavior beyond the near-field region, and the interface equations couple the two randomized representations.

	In Example~\ref{eg:2d-whole}, we use the DD-RaNN-CM with two single-hidden-layer RaNNs. For a single-hidden-layer RaNN, the weight matrix $\bW$ is sampled from a uniform distribution $U(-r,r)$ and the bias matrix $\bb$ is constructed by the second strategy. The activation function is \(\rho(x,y)=e^{-x^2-y^2}\).
	
	In this example, the near--far interface is a circle boundary with a radius of 9 centered at the origin. For the near-field subnetwork, the parameter \(r_1 = 0.5\), and the number of neurons \(m_1\) varies from 50 to 450. For the far-field subnetwork, the parameter \(r_2 = 10\), and the number of neurons \(m_2 = 50\). Unless otherwise noted, ``DoF'' is defined as $m_1+m_2$. Each boundary is sampled with 500 points, and the interior of both networks is sampled with 4000 points. When assembling the linear equations, the penalty parameter $\xi  $ and $ \eta $ are set to 10.
	
		\begin{figure}[H]
		\centering
		\begin{minipage}{0.45\linewidth}
			\centering
			\includegraphics[width=\textwidth]{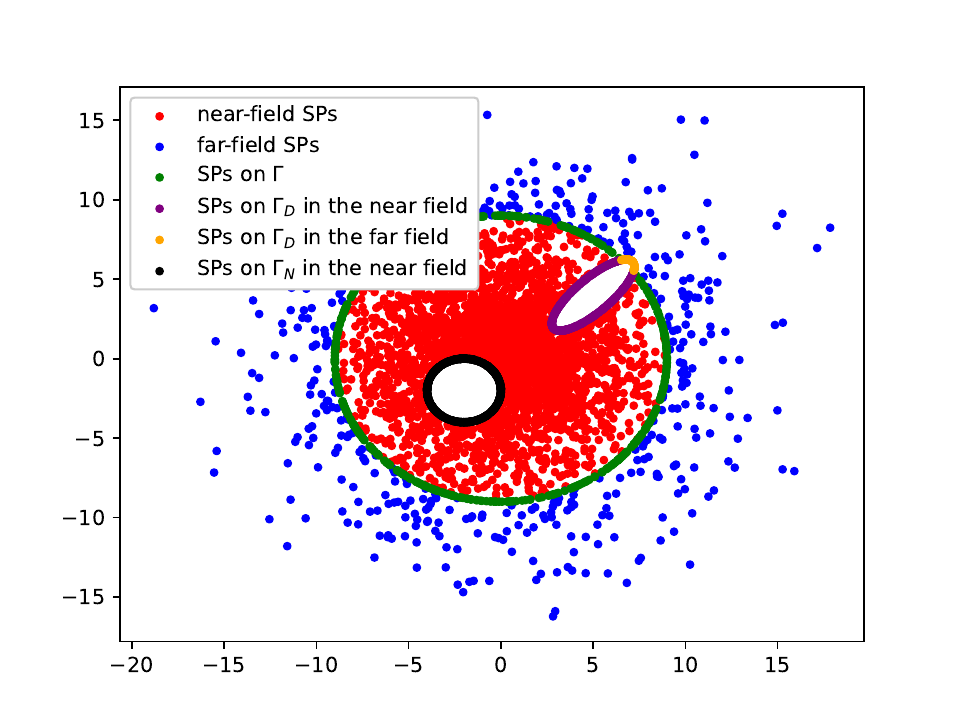}
			\label{fig:2d-whole-sample}
		\end{minipage}
		\hspace{0.05\linewidth} 
		\begin{minipage}{0.45\linewidth}
			\centering
			\includegraphics[width=\textwidth]{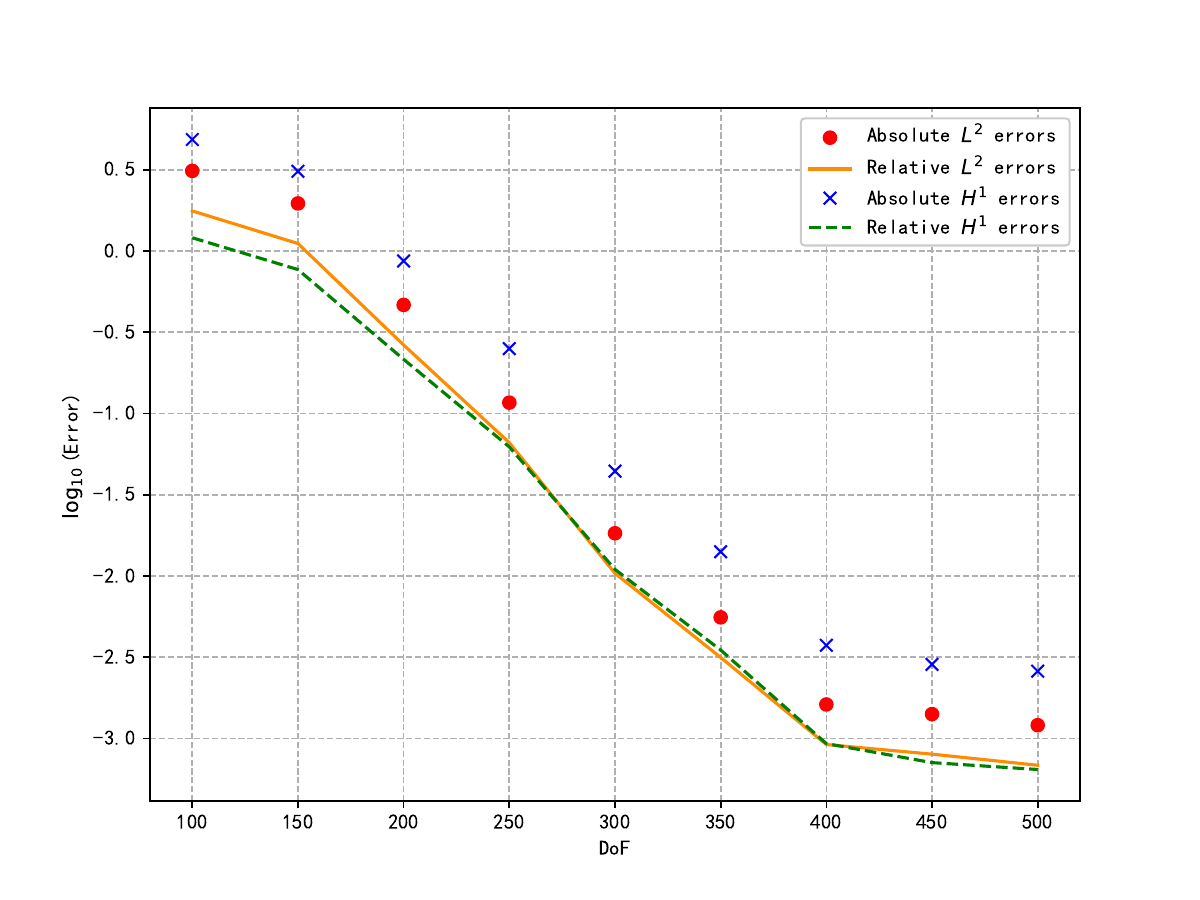}
			\label{fig:2d-whole-error}
		\end{minipage}
		\caption{Scatter plot of sampling points (SPs) for DD-RaNN-CM (left) and the logarithmic errors for DD-RaNN-CM in Example~\ref{eg:2d-whole}.}
		\label{fig:2d-whole}
	\end{figure}

	The left panel in Figure~\ref{fig:2d-whole} illustrates the point distribution from one of the experiments. The interior points are sampled in polar coordinates. The radial distance is sampled from the normal distribution \(\mathcal{N}(2, 5)\), and the angle is sampled from the uniform distribution \(U(0, 2\pi)\). If the radius is negative, its absolute value is taken. Each sampled point is checked to determine whether it falls within a interior
hole; if so, the point is resampled.
	
	The elliptical boundary, circular boundary and near--far interface are also sampled in polar coordinates, with angles sampled from the uniform distribution \(U(0, 2\pi)\).

	\begin{figure}[H]
		\centering
		\includegraphics[width=\textwidth]{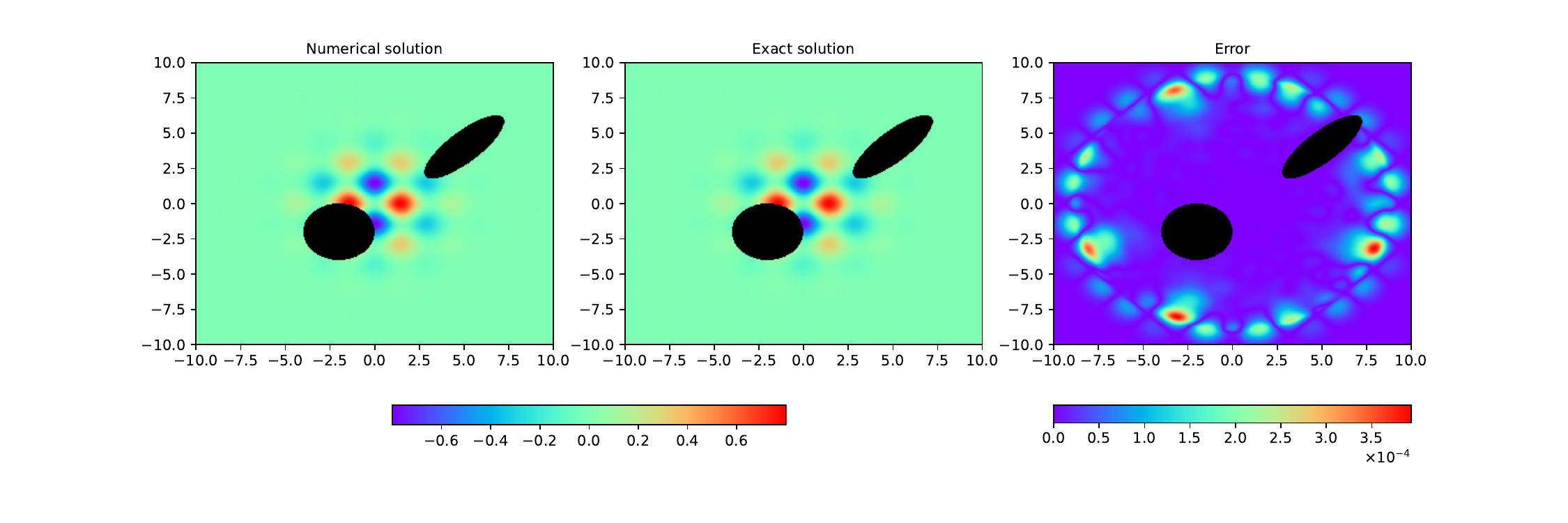}
		% \label{fig:1+1_sample}
		\caption{Computational solution obtained using the DD-RaNN-CM (left), exact solution (middle), and absolute error plot (right) for $m_1 = 450$ in Example~\ref{eg:2d-whole}.}
		\label{fig:2d-whole-sol}
	\end{figure}

	The errors shown in Table~\ref{tab:2d-whole-errors} and Figure~\ref{fig:2d-whole} are the averages obtained from 5 experiments for each parameter, while Figure~\ref{fig:2d-whole-sol} shows the results of a single experiment. The error for each experiment is computed using (\ref{mapped-error}) with 50 points. We can observe that as the number of neurons in the near-field subnetwork increases, the error gradually decreases and eventually stabilizes around the order of \(10^{-4}\).
	
	\begin{table}[H]
		\centering
		\caption{$L^2$ error, relative $L^2$ error, $H^1$ error, and relative $H^1$ error for different numbers of neurons in Example~\ref{eg:2d-whole}.}
		\label{tab:2d-whole-errors}
		\begin{tabular}{c|c c c c}
			\toprule
			\textbf{$ DoF $} & $L^2$ Error & Relative $L^2$ Error & $H^1$ Error & Relative $H^1$ Error \\
			\midrule
			150  & 1.9628 & 1.1126 & 3.0959 & $7.6921 \times 10^{-1}$ \\
			200  & $4.6616 \times 10^{-1}$ & $2.6425 \times 10^{-1}$ & $8.6752 \times 10^{-1}$ & $2.1555 \times 10^{-1}$ \\
			250  & $1.1675 \times 10^{-1}$ & $6.6178 \times 10^{-2}$ & $2.5078 \times 10^{-1}$ & $6.2309 \times 10^{-2}$ \\
			300  & $1.8335 \times 10^{-2}$ & $1.0393 \times 10^{-2}$ & $4.4201 \times 10^{-2}$ & $1.0982 \times 10^{-2}$ \\
			350  & $5.5675 \times 10^{-3}$ & $3.1560 \times 10^{-3}$ & $1.4104 \times 10^{-2}$ & $3.5042 \times 10^{-3}$ \\
			400  & $1.6209 \times 10^{-3}$ & $9.1883 \times 10^{-4}$ & $3.7460 \times 10^{-3}$ & $9.3074 \times 10^{-4}$ \\
			450  & $1.4142 \times 10^{-3}$ & $8.0165 \times 10^{-4}$ & $2.8613 \times 10^{-3}$ & $7.1091 \times 10^{-4}$ \\
			500  & $1.2081 \times 10^{-3}$ & $6.8481 \times 10^{-4}$ & $2.5960 \times 10^{-3}$ & $6.4499 \times 10^{-4}$ \\
			\bottomrule
		\end{tabular}
	\end{table}

	\begin{example}
		
		In this example, we consider the equation $(\ref{eq:GF})$ with $\Omega=\R^3$, $\beta=1$ and $\gamma=1$. The exact solution is 
		\begin{equation}
			u(\bx)=e^{-\|\bx\|_2^2}\sin(x_1+x_2+x_3),\quad \bx=(x_1,x_2,x_3),
		\end{equation}
		and the source term $f$ can be set by calculation. 
		\label{eg:3d-whole}
	\end{example}

	This example is used to evaluate the scalability of the proposed collocation framework in a three-dimensional whole-space problem. It tests whether the near--far RaNN decomposition remains effective when the number of spatial dimensions and sampling points increases.

	In Example~\ref{eg:3d-whole}, we use the DD-RaNN-CM with two single-hidden-layer RaNNs.
	For a single-hidden-layer RaNN, the weight matrix $\bW$ is sampled from a uniform distribution $U(-r,r)$ and the bias matrix $\bb$ is constructed by the second strategy. The activation function is \(\rho(\bx)=e^{-x_1^2-x_2^2-x_3^2}\).

	\begin{figure}[H]
		\centering
		\begin{minipage}{0.50\linewidth}
			\centering
			\includegraphics[width=1\textwidth]{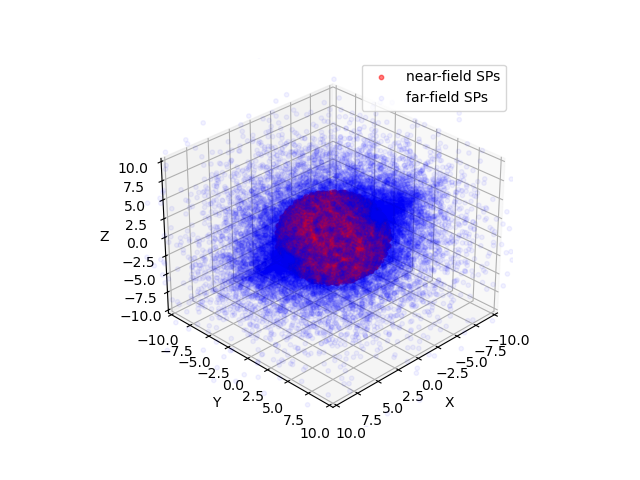}
			\label{fig:3d-whole-SPs}
		\end{minipage}
		\begin{minipage}{0.45\linewidth}
			\centering
			\includegraphics[width=\textwidth]{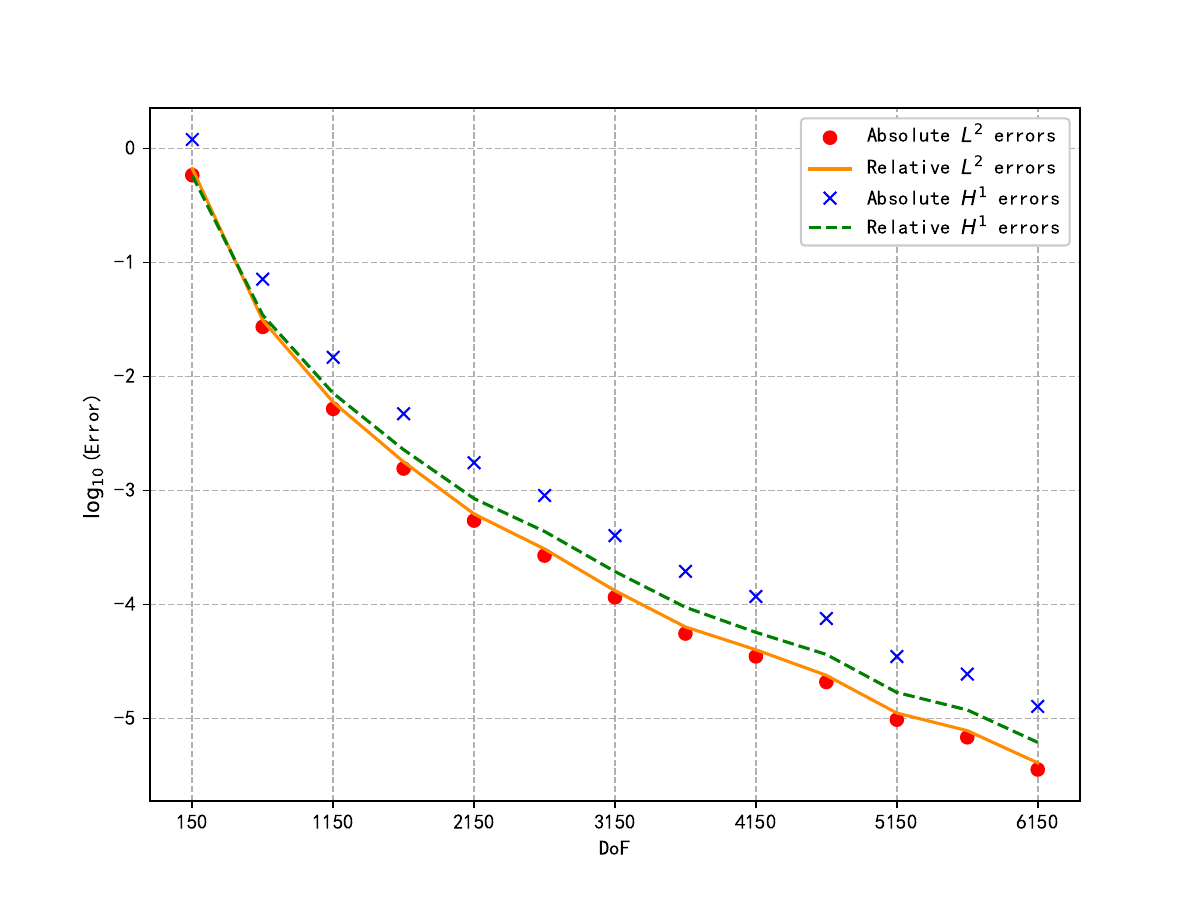}
			\label{fig:3d-whole-RANN}
		\end{minipage}
		\caption{Scatter plot of sampling points (SPs) for RaNN1 and RaNN2 within a cube ranging from -10 to 10 (left) and the logarithmic errors for DD-RaNN-CM in Example~\ref{eg:3d-whole}.}
		\label{fig:3d-whole}
	\end{figure}
	
	In this example, the near--far interface is a spherical surface with a radius of 5, centered at the origin. For the near-field subnetwork, the parameter \(r_1 = 2\), and the number of neurons \(m_1\) varies from 100 to 6100. For the far-field subnetwork, the parameter \(r_2 = 10\), and the number of neurons \(m_2 = 50\). Unless otherwise noted, ``DoF'' is defined as $m_1+m_2$. Each boundary is sampled with 5000 points, and the interior of both networks is sampled with 50000 points. When assembling the linear equations, the boundary conditions are enforced with a penalty parameter of 10.
	
	The left figure in Figure~\ref{fig:3d-whole} illustrates the point distribution from one of the experiments. The interior points are sampled in spherical coordinates. The radial distance is sampled from the normal distribution \(\mathcal{N}(2, 5)\), and the angles are sampled from the uniform distribution \(U(0, 2\pi)\). If the radius is negative, its absolute value is taken. The points on the near--far interface are also sampled in polar coordinates, with angles sampled from the uniform distribution \(U(0, 2\pi)\).

	\begin{table}[H]
		\centering
		\caption{$L^2$ error, relative $L^2$ error, $H^1$ error, and relative $H^1$ error for different numbers of neurons in Example~\ref{eg:3d-whole}.}
		\label{tab:3d-whole-errors}
		\begin{tabular}{c|c c c c}
			\toprule
			\textbf{$ DoF $} & $L^2$ Error & Relative $L^2$ Error & $H^1$ Error & Relative $H^1$ Error \\
			\midrule
					150  & $5.8198\times10^{-1}$ & $6.6552\times10^{-1}$ & $1.1962$ & $5.7939\times10^{-1}$ \\
					1150 & $5.2012\times10^{-3}$ & $5.9478\times10^{-3}$ & $1.4717\times10^{-2}$ & $7.1281\times10^{-3}$ \\
					2150 & $5.4472\times10^{-4}$ & $6.2291\times10^{-4}$ & $1.7545\times10^{-3}$ & $8.4978\times10^{-4}$ \\
					3150 & $1.1572\times10^{-4}$ & $1.3233\times10^{-4}$ & $4.0190\times10^{-4}$ & $1.9465\times10^{-4}$ \\
					4150 & $3.5105\times10^{-5}$ & $4.0144\times10^{-5}$ & $1.1777\times10^{-4}$ & $5.7042\times10^{-5}$ \\
					5150 & $9.7874\times10^{-6}$ & $1.1192\times10^{-5}$ & $3.5030\times10^{-5}$ & $1.6966\times10^{-5}$ \\
					6150 & $3.5790\times10^{-6}$ & $4.0928\times10^{-6}$ & $1.2777\times10^{-5}$ & $6.1884\times10^{-6}$ \\
			\bottomrule
		\end{tabular}
	\end{table}
	
	The errors shown in Table~\ref{tab:3d-whole-errors} and Figure~\ref{fig:3d-whole} are the averages obtained after performing 5 experiments for each parameter. The error for each experiment is calculated using the three-dimensional form of equation (\ref{mapped-error}), with 50 Gaussian points in each dimension, totaling 125,000 quadrature points. We can observe that as the number of neurons in RaNN1 increases, the error gradually decreases to the order of \(10^{-6}\).

	\begin{example}
		In this example, we consider a linear Schr\"{o}dinger equation:
		\begin{equation}
			\begin{cases}
				i u_t+u_{xx}=f,\quad x\in \mathbb{R},\quad t\in [0,2],\\
				u(x,0)=g(x),\quad x\in \mathbb{R},\\
				u(x,t)\rightarrow 0,\quad \|x\|_2\rightarrow \infty,
			\end{cases}
		\end{equation}
		The exact solution is 
		\begin{equation}
			u(x,t)=e^{-x^2}\cos(t) +i e^{-x^2}\sin(t).
		\end{equation}
		The source term $f$ and the initial condition $g(x)$ can be set by calculation. 
		\label{eg:sho_1+1}
	\end{example}

	This example demonstrates that the proposed DD-RaNN-CM framework can also be applied to time-dependent complex-valued equations. This complex-valued time-dependent example is not directly covered by the elliptic broken graph-norm theory in Section~\ref{sec:analysis}; it is included to illustrate the flexibility of the domain-decomposed RaNN construction. The three-network decomposition separates the left far-field, central region, and right far-field, while the spatio-temporal random features allow the output coefficients to be determined from one global least-squares system.

	In Example~\ref{eg:sho_1+1}, we use the DD-RaNN-CM with three single-hidden-layer RaNNs, and each network is spatiotemporal. For a single-hidden-layer RaNN, the weight matrix $\bW$ and the bias matrix $\bb$ are sampled from the uniform distribution $U(-r,r)$. The activation function is \(\rho(x,t)=\tanh(x)\cos(t)+i\tanh(x)\sin(t)\). As above, the use of \(\tanh\)-based non-decaying features in this numerical example lies outside the direct scope of the decaying-activation approximation theorem.

	\begin{figure}[H]
		\centering
		\begin{minipage}{0.45\linewidth}
			\includegraphics[width=\textwidth]{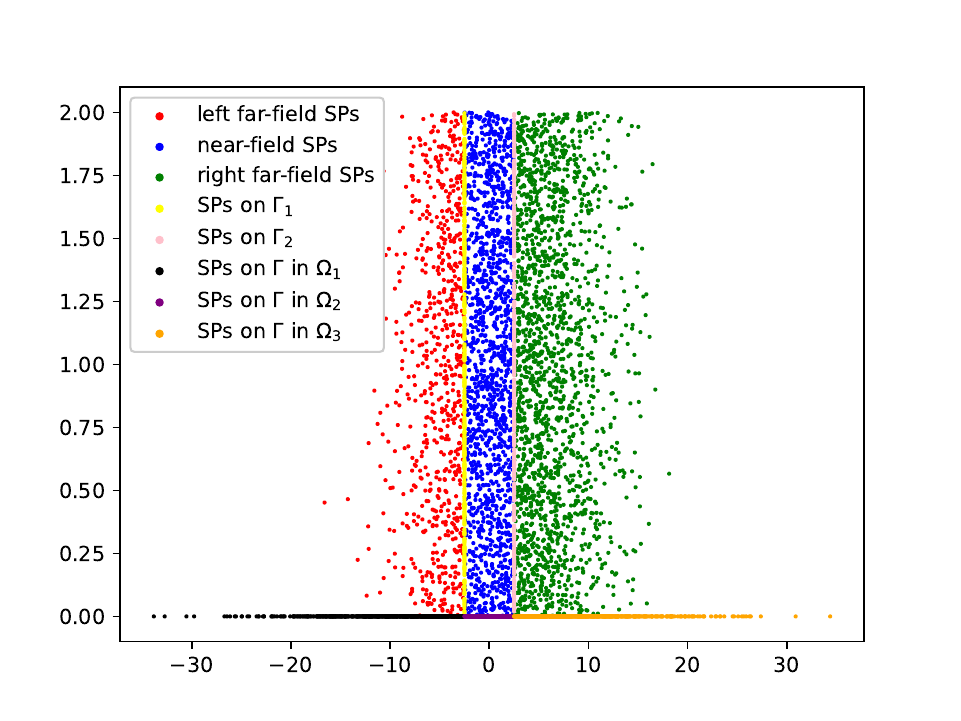}
			\label{fig:1+1_sample}
		\end{minipage}
		\caption{Scatter plot of the sampling points (SPs) distribution for Example~\ref{eg:sho_1+1}.}
		\label{fig:1+1_Sho_sample}
	\end{figure}	
	After applying the activation function, we obtain a trial function space $ \mathcal{U}(\bW,\bb, \rho) $ over the complex domain. Therefore, in this example, we can set the corresponding coefficients in (\ref{solu:Whole}) to be real-valued for convenience. A simple approach is to decompose the equation (\ref{ab4}) into two sub-equations based on its real and imaginary parts.

	In this example, the near--far interfaces are $x=\pm 2.5$. We use three spatio-temporal networks. For the near-field subnetwork we vary the neuron number $m_2\in\{40,60,\cdots,240,260,280\}$ and set the parameter \(r_2 = 0.5\), 
	while the left/right far-field subnetworks are fixed at $m_1=m_3=50$ and \(r_1 = r_3 = 3\). On the near--far interface, 500 points were sampled. For the initial condition, 1,500 points were sampled, and 4,000 points were sampled across the entire plane. Unless otherwise noted, ``DoF'' is defined as $m_1+m_2+m_3$.

	Figure~\ref{fig:1+1_Sho_sample} illustrates the point distribution from one of the experiments. The interior points are sampled in Euclidean coordinates, where \(x\) is sampled from the Gaussian distribution \(\mathcal{N}(2, 5)\), and \(t\) is sampled from the uniform distribution \(U(0, 2)\).
	For the near--far interfaces, the vertical coordinates of the sampled points are selected according to the uniform distribution \(U(0, 2)\). For the initial condition, the samples are selected according to the Gaussian distribution \(\mathcal{N}(0, 10)\).
	\begin{figure}[H]
		\centering
		\begin{minipage}{0.45\linewidth}
			\centering
			\includegraphics[width=\textwidth]{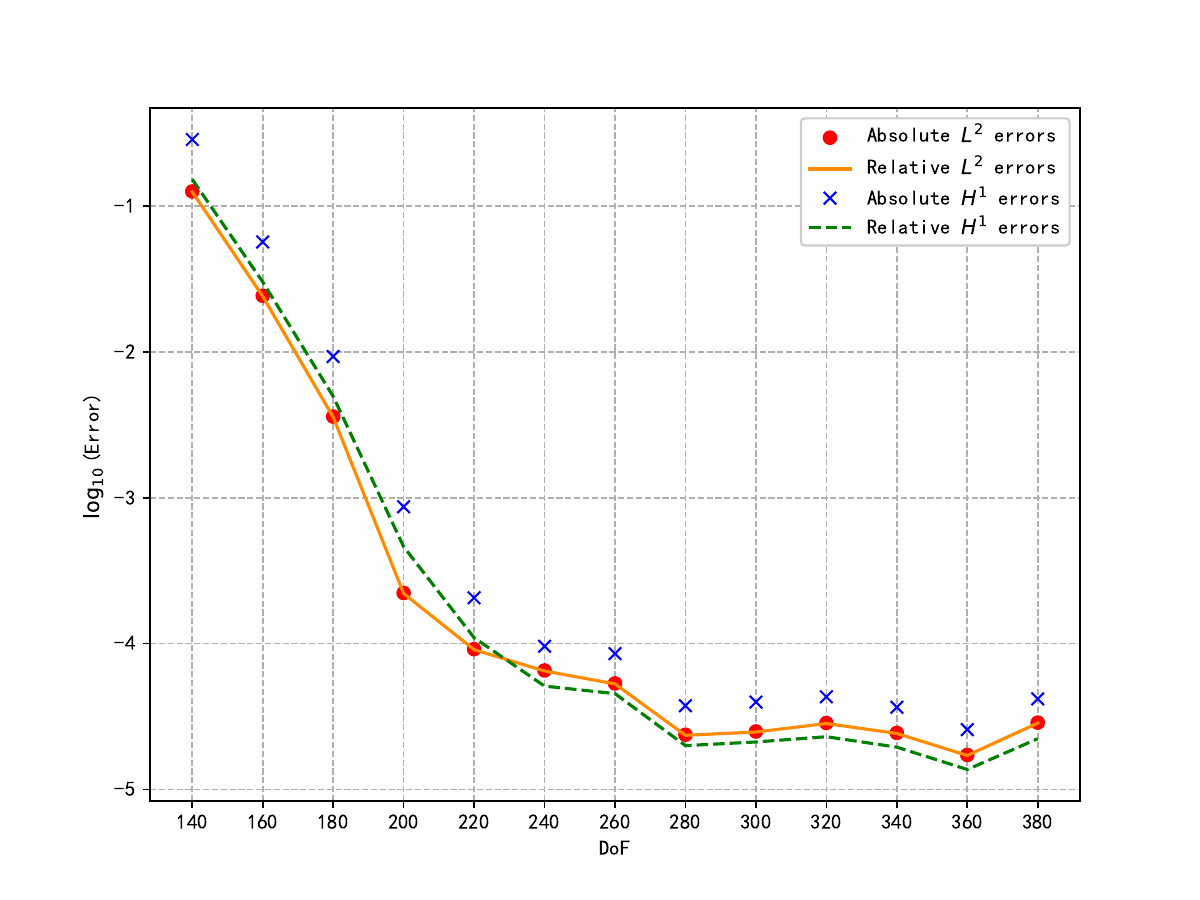}
			\label{fig:1+1 real}
		\end{minipage}
		\hspace{0.05\linewidth} 
		\begin{minipage}{0.45\linewidth}
			\centering
			\includegraphics[width=\textwidth]{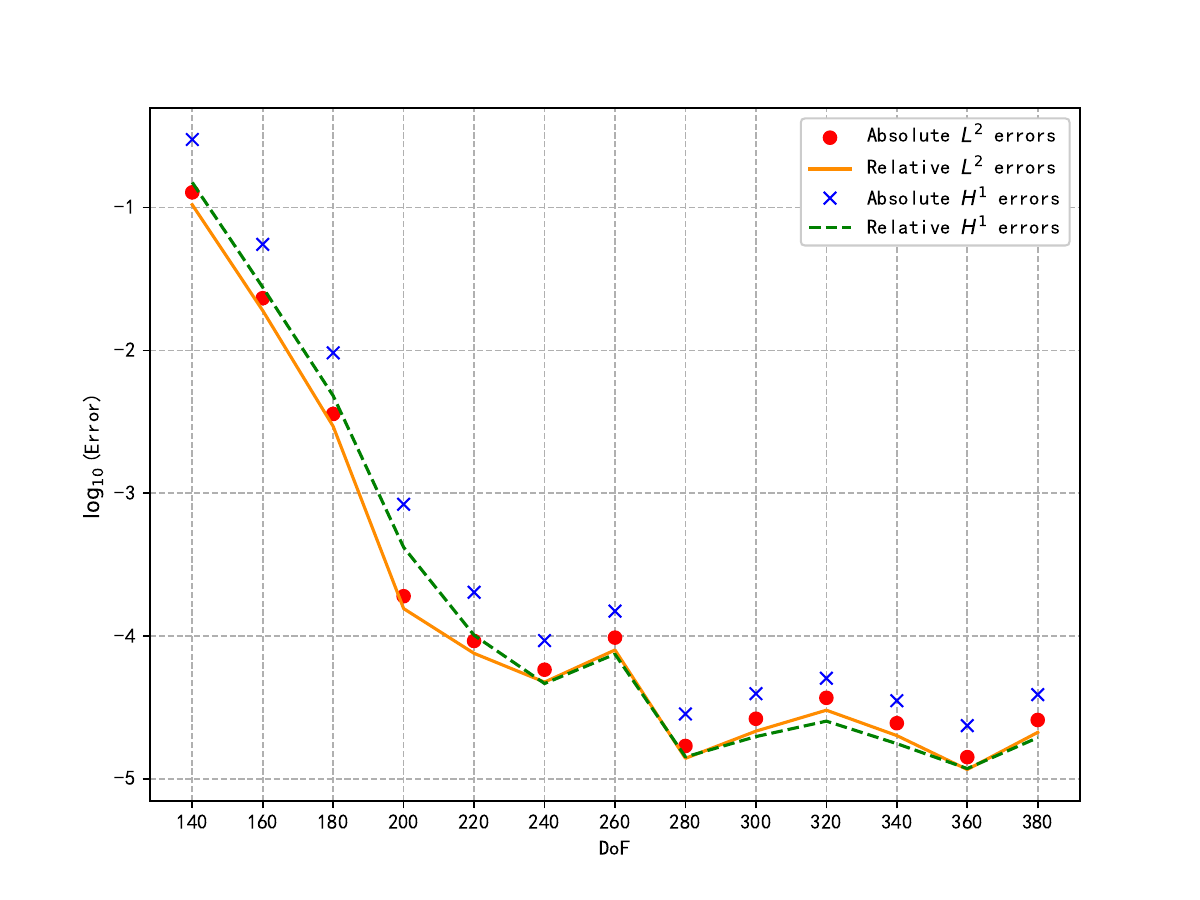}
			\label{fig:1+1 img}
		\end{minipage}
		\caption{Logarithmic errors for the real part (left) and imaginary part (right) of Example~\ref{eg:sho_1+1}.}
		\label{fig:1+1_Sho}
	\end{figure}

	The errors shown in Table~\ref{tab:errors_real}, Table~\ref{tab:errors_imaginary} and Figure~\ref{fig:1+1_Sho} are the averages obtained after performing 20 experiments for each parameter. The error for each experiment is computed using (\ref{mapped-error}) with 50 points. We observe that, as the number of neurons in the near-field subnetwork increases, the error gradually decreases and eventually stabilizes at around \(10^{-5}\).

\begin{table}[H]
	\centering
	\caption{Real part: $L^2$ error, relative $L^2$ error, $H^1$ error, and relative $H^1$ error for different numbers of neurons in Example~\ref{eg:sho_1+1}.}
	\label{tab:errors_real}
	\begin{tabular}{c|c c c c}
		\toprule
		{ $DoF$} & $L^2$ Error & Relative $L^2$ Error & $H^1$ Error & Relative $H^1$ Error \\
		\midrule
		140  & $1.2630 \times 10^{-1}$ & $1.2529 \times 10^{-1}$ & $2.8668 \times 10^{-1}$ & $1.5274 \times 10^{-1}$ \\ 
		160  & $2.4298 \times 10^{-2}$ & $2.4104 \times 10^{-2}$ & $5.6828 \times 10^{-2}$ & $3.0277 \times 10^{-2}$ \\ 
		180  & $3.6085 \times 10^{-3}$ & $3.5797 \times 10^{-3}$ & $9.3133 \times 10^{-3}$ & $4.9620 \times 10^{-3}$ \\ 
		200  & $2.2255 \times 10^{-4}$ & $2.2077 \times 10^{-4}$ & $8.6775 \times 10^{-4}$ & $4.6232 \times 10^{-4}$ \\ 
		260  & $5.3400 \times 10^{-5}$ & $5.2973 \times 10^{-5}$ & $8.5319 \times 10^{-5}$ & $4.5457 \times 10^{-5}$ \\ 
		320  & $2.8557 \times 10^{-5}$ & $2.8329 \times 10^{-5}$ & $4.3179 \times 10^{-5}$ & $2.3005 \times 10^{-5}$ \\ 
		380  & $2.8741 \times 10^{-5}$ & $2.8511 \times 10^{-5}$ & $4.1830 \times 10^{-5}$ & $2.2287 \times 10^{-5}$ \\
		\bottomrule
	\end{tabular}
\end{table}

\begin{table}[H]
	\centering
	\caption{Imaginary part: $L^2$ error, relative $L^2$ error, $H^1$ error, relative $H^1$ error for different numbers of neurons in Example~\ref{eg:sho_1+1}.}
	\label{tab:errors_imaginary}
	\begin{tabular}{c|c c c c}
		\toprule
		{$DoF$}& $L^2$ Error & Relative $L^2$ Error & $H^1$ Error & Relative $H^1$ Error \\
		\midrule
		140  & $1.2771 \times 10^{-1}$ & $1.0461 \times 10^{-1}$ & $2.9943 \times 10^{-1}$ & $1.4977 \times 10^{-1}$ \\ 
		160  & $2.3186 \times 10^{-2}$ & $1.8992 \times 10^{-2}$ & $5.5233 \times 10^{-2}$ & $2.7626 \times 10^{-2}$ \\ 
		180  & $3.5856 \times 10^{-3}$ & $2.9370 \times 10^{-3}$ & $9.6100 \times 10^{-3}$ & $4.8067 \times 10^{-3}$ \\ 
		200  & $1.8994 \times 10^{-4}$ & $1.5558 \times 10^{-4}$ & $8.3564 \times 10^{-4}$ & $4.1797 \times 10^{-4}$ \\ 
		260  & $9.7435 \times 10^{-5}$ & $7.9810 \times 10^{-5}$ & $1.4920 \times 10^{-4}$ & $7.4625 \times 10^{-5}$ \\ 
		320  & $3.6912 \times 10^{-5}$ & $3.0235 \times 10^{-5}$ & $5.0672 \times 10^{-5}$ & $2.5345 \times 10^{-5}$ \\ 
		380  & $2.5823 \times 10^{-5}$ & $2.1152 \times 10^{-5}$ & $3.8834 \times 10^{-5}$ & $1.9424 \times 10^{-5}$ \\
		\bottomrule
	\end{tabular}
\end{table}

	\section{Summary}
	\label{sec:summary}

In this paper, we proposed a domain-decomposed randomized neural network framework for solving partial differential equations on unbounded domains. The main idea is to represent different spatial regimes by different randomized subnetworks: a near-field subnetwork captures local structures, boundary effects, and geometric complexity, while a far-field subnetwork represents the exterior decay. The subnetworks are coupled through boundary and interface equations, and all output-layer coefficients are determined from linear least-squares systems.

Two complementary discretizations were developed. The DD-RaNN Petrov--Galerkin method treats semi-unbounded elliptic problems by combining randomized trial spaces with piecewise polynomial test functions and Gauss-type quadrature. The DD-RaNN collocation method handles fully unbounded domains, geometrically complicated exterior domains, and a time-dependent Schr\"odinger equation. In both cases, the method avoids nonlinear optimization over all neural-network parameters.

The theoretical analysis is formulated in a broken Sobolev framework, which is the natural setting for characteristic-function-based domain decomposition. Under suitable decay, activation, local representation, and stability assumptions, we proved a conditional bounded-parameter broken Sobolev approximation result and derived least-squares error estimates. The total error is decomposed into approximation, empirical-consistency/quadrature, and algebraic optimization components. For a practical \(L^2\)-type broken empirical loss, we also established a deterministic quadrature-consistency estimate. More general losses, such as negative Sobolev losses and strong trace-norm losses, require additional consistency arguments, and the decaying-activation theorem does not directly cover numerical tests using non-decaying activations such as \(\tanh\).

The numerical experiments demonstrate the effectiveness of the proposed methods for one-, two-, and three-dimensional unbounded-domain Poisson problems and for a one-dimensional time-dependent Schr\"odinger equation. The results indicate that the combination of near--far randomized trial spaces, interface coupling, and linear least-squares solvers can provide accurate approximations for the tested unbounded-domain PDEs.

Several issues remain for future work. The interface location, random-parameter ranges, and penalty parameters should be further optimized, possibly by adaptive strategies. It is also important to improve far-field representations for oscillatory or slowly decaying solutions and to extend the analysis to more general nonlinear, time-dependent, and multi-domain settings.

	\bibliographystyle{plain} 
	\bibliography{bibliography.bib} 

@article{Dissanayake1994NNPDE,
  author  = {Dissanayake, M. W. M. G. and Phan-Thien, N.},
  title   = {Neural-network-based approximations for solving partial differential equations},
  journal = {Communications in Numerical Methods in Engineering},
  volume  = {10},
  number  = {3},
  pages   = {195--201},
  year    = {1994},
  doi     = {10.1002/cnm.1640100303}
}

@article{Lagaris1998ANN,
  author  = {Lagaris, I. E. and Likas, A. and Fotiadis, D. I.},
  title   = {Artificial neural networks for solving ordinary and partial differential equations},
  journal = {IEEE Transactions on Neural Networks},
  volume  = {9},
  number  = {5},
  pages   = {987--1000},
  year    = {1998},
  doi     = {10.1109/72.712178}
}

@article{Raissi2019PINN,
  author  = {Raissi, Maziar and Perdikaris, Paris and Karniadakis, George Em},
  title   = {Physics-informed neural networks: A deep learning framework for solving forward and inverse problems involving nonlinear partial differential equations},
  journal = {Journal of Computational Physics},
  volume  = {378},
  pages   = {686--707},
  year    = {2019},
  doi     = {10.1016/j.jcp.2018.10.045}
}

@article{Raissi2017PINNPartI,
  author        = {Raissi, Maziar and Perdikaris, Paris and Karniadakis, George Em},
  title         = {Physics Informed Deep Learning (Part I): Data-driven Solutions of Nonlinear Partial Differential Equations},
  journal       = {arXiv preprint arXiv:1711.10561},
  year          = {2017},
  eprint        = {1711.10561},
  archivePrefix = {arXiv}
}

@article{Raissi2017PINNPartII,
  author        = {Raissi, Maziar and Perdikaris, Paris and Karniadakis, George Em},
  title         = {Physics Informed Deep Learning (Part II): Data-driven Discovery of Nonlinear Partial Differential Equations},
  journal       = {arXiv preprint arXiv:1711.10566},
  year          = {2017},
  eprint        = {1711.10566},
  archivePrefix = {arXiv}
}

@article{Sirignano2018DGM,
  author  = {Sirignano, Justin and Spiliopoulos, Konstantinos},
  title   = {DGM: A deep learning algorithm for solving partial differential equations},
  journal = {Journal of Computational Physics},
  volume  = {375},
  pages   = {1339--1364},
  year    = {2018},
  doi     = {10.1016/j.jcp.2018.08.029}
}

@article{Shen2009spctral,
  title={Some recent advances on spectral methods for unbounded domains},
  author={Shen, Jie and Wang, Lilian},
  journal={Communications in Computational Physics},
  volume={5},
  number={2-4},
  pages={195--241},
  year={2009}
}

@article{Jens2018unified,
   title={A unified deep artificial neural network approach to partial differential equations in complex geometries},
  author={Berg, Jens and Nystr{\"o}m, Kaj},
  journal={Neurocomputing},
  volume={317},
  pages={28--41},
  year={2018},
  publisher={Elsevier}
}

@article{Tang2023Physics,
   title={Physics-informed neural networks combined with polynomial interpolation to solve nonlinear partial differential equations},
  author={Tang, Siping and Feng, Xinlong and Wu, Wei and Xu, Hui},
  journal={Computers \& Mathematics with Applications},
  volume={132},
  pages={48--62},
  year={2023},
  publisher={Elsevier}
}

@article{Ding2023Finite,
   title={Finite difference method for time-fractional {K}lein--{G}ordon equation on an unbounded domain using artificial boundary conditions},
  author={Ding, Peng and Yan, Yubin and Liang, Zongqi and Yan, Yuyuan},
  journal={Mathematics and Computers in Simulation},
  volume={205},
  pages={902--925},
  year={2023},
  publisher={Elsevier}
}

@article{Tai2022Numerical,
   title={Numerical solution of coupled nonlinear {Klein-Gordon} equations on unbounded domains},
  author={Tai, Yinong and Li, Hongwei and Zhou, Zhaojie and Jiang, Ziwen},
  journal={Physical Review E},
  volume={106},
  number={2},
  pages={025317},
  year={2022},
  publisher={APS}
}

@article{JonathanW2023Greedy,
   title={Greedy training algorithms for neural networks and applications to {PDEs}},
  author={Siegel, Jonathan W and Hong, Qingguo and Jin, Xianlin and Hao, Wenrui and Xu, Jinchao},
  journal={Journal of Computational Physics},
  volume={484},
  pages={112084},
  year={2023},
  publisher={Elsevier}
}

@article{Ling2023Analysis,
    title={Analysis and {H}ermite spectral approximation of diffusive-viscous wave equations in unbounded domains arising in geophysics},
  author={Ling, Dan and Mao, Zhiping},
  journal={Journal of Scientific Computing},
  volume={95},
  pages={51},
  year={2023},
  publisher={Springer}
}

@article{zhang2024transferable,
  author  = {Zhang, Zezhong and Bao, Feng and Ju, Lili and Zhang, Guannan},
  title   = {Transferable Neural Networks for Partial Differential Equations},
  journal = {Journal of Scientific Computing},
  volume  = {99},
  number  = {1},
  pages   = {2},
  year    = {2024},
  doi     = {10.1007/s10915-024-02463-y}
}

@article{shang2023randomized,
  author  = {Shang, Yong and Wang, Fei and Sun, Jingbo},
  title   = {Randomized Neural Network with {P}etrov--{G}alerkin Methods for Solving Linear and Nonlinear Partial Differential Equations},
  journal = {Communications in Nonlinear Science and Numerical Simulation},
  volume  = {127},
  pages   = {107518},
  year    = {2023},
  doi     = {10.1016/j.cnsns.2023.107518},
  eprint  = {2201.12995},
  archivePrefix = {arXiv}
}

@article{liu2025integral,
  author  = {Liu, Xinliang and Mao, Tong and Xu, Jinchao},
  title   = {Integral Representations of {S}obolev Spaces via {R}eLU$^k$ Activation Function and Optimal Error Estimates for Linearized Networks},
  journal = {arXiv preprint arXiv:2505.00351},
  year    = {2025},
  doi     = {10.48550/arXiv.2505.00351},
  eprint  = {2505.00351},
  archivePrefix = {arXiv}
}

@article{Chen2023Analysis,
title={Analysis of absorbing boundary conditions for the anomalous diffusion in comb model on unbounded domain by finite volume method},
  author={Chen, Siyu and Liu, Lin and Li, Jiajia and Yang, Jingyu and Feng, Libo and Zhang, Jiangshan},
  journal={Applied Mathematics Letters},
  volume={144},
  pages={108712},
  year={2023},
  publisher={Elsevier}
}

@article{Chen2022TwoGrid,
  title={Two-grid finite volume element method for the time-dependent {S}chr{\"o}dinger equation},
  author={Chen, Chuanjun and Lou, Yuzhi and Hu, Hanzhang},
  journal={Computers \& Mathematics with Applications},
  volume={108},
  pages={185--195},
  year={2022},
  publisher={Elsevier}
}

@article{Xie2023fastBDF2,
   title={A fast {BDF2} {G}alerkin finite element method for the one-dimensional time-dependent {S}chr{\"o}dinger equation with artificial boundary conditions},
  author={Xie, Jiangming and Li, Maojun},
  journal={Applied Numerical Mathematics},
  volume={187},
  pages={89--106},
  year={2023},
  publisher={Elsevier}
}

@article{Guo2021Finite,
  title={Finite difference/generalized {H}ermite spectral method for the distributed-order time-fractional reaction-diffusion equation on multi-dimensional unbounded domains},
  author={Guo, Shimin and Chen, Yaping and Mei, Liquan and Song, Yining},
  journal={Computers \& Mathematics with Applications},
  volume={93},
  pages={1--19},
  year={2021},
  publisher={Elsevier}
}

@article{Guo2022Dissipation,
   title={ Dissipation-preserving rational spectral-{G}alerkin method for strongly damped nonlinear wave system involving mixed fractional {L}aplacians in unbounded domains},
  author={Guo, Shimin and Yan, Wenjing and Li, Can and Mei, Liquan},
  journal={Journal of Scientific Computing},
  volume={93},
  number={2},
  pages={53},
  year={2022},
  publisher={Springer}
}

@article{Dina2023Tanh,
   title={Tanh {J}acobi spectral collocation method for the numerical simulation of nonlinear {S}chr{\"o}dinger equations on unbounded domain},
  author={Mostafa, Dina and Zaky, Mahmoud A and Hafez, Ramy M and Hendy, Ahmed S and Abdelkawy, Mohamed A and Aldraiweesh, Ahmed A},
  journal={Mathematical Methods in the Applied Sciences},
  volume={46},
  number={1},
  pages={656--674},
  year={2023},
  publisher={Wiley Online Library}
}

@article{Carolin2022Magnetostatic,
   title={Magnetostatic simulations with consideration of exterior domains using the scaled boundary finite element method},
  author={Birk, Carolin and Reichel, Maximilian and Schr{\"o}der, J{\"o}rg},
  journal={Computer Methods in Applied Mechanics and Engineering},
  volume={399},
  pages={115362},
  year={2022},
  publisher={Elsevier}
}

@article{Silvia2022TwoFEMBEM,
  title={Two {FEM-BEM} methods for the numerical solution of {2D} transient elastodynamics problems in unbounded domains},
  author={Falletta, Silvia and Monegato, Giovanni and Scuderi, Letizia},
  journal={Computers \& Mathematics with Applications},
  volume={114},
  pages={132--150},
  year={2022},
  publisher={Elsevier}
}

@article{Singh2023Rate,
  title={Rate of convergence of two moments consistent finite volume scheme for non-classical divergence coagulation equation},
  author={Singh, Mehakpreet},
  journal={Applied Numerical Mathematics},
  volume={187},
  pages={120--137},
  year={2023},
  publisher={Elsevier}
}

@article{Zhang2021Spectral,
  title={Spectral method for multi-dimensional problems of high order on unbounded domains using generalized {L}aguerre functions},
  author={Zhang, Chao and Tao, Dongya and Ding, Pan},
  journal={International Journal of Computer Mathematics},
  volume={98},
  number={10},
  pages={2040--2059},
  year={2021},
  publisher={Taylor \& Francis}
}

@article{E2018deepRitz,
   title={The deep {R}itz method: a deep learning-based numerical algorithm for solving variational problems},
  author={E, Weinan and Yu, Bing},
  journal={Communications in Mathematics and Statistics},
  volume={6},
  number={1},
  pages={1--12},
  year={2018},
  publisher={Springer}
}

@article{sun2022local,
	title={Local randomized neural networks with discontinuous {G}alerkin methods for partial differential equations},
	author={Sun, Jingbo and Dong, Suchuan and Wang, Fei},
	journal={Journal of Computational and Applied Mathematics},
	volume={445},
	pages={115830},
	year={2024},
	publisher={Elsevier}
}

@article{dong2021LocalELM,
	title={Local extreme learning machines and domain decomposition for solving linear and nonlinear partial differential equations},
	author={Dong, Suchuan and Li, Zongwei},
	journal={Computer Methods in Applied Mechanics and Engineering},
	volume={387},
	pages={114129},
	year={2021},
	publisher={Elsevier}
}

@article{lin2025finite,
  title={A finite element method for elliptic optimal control problem in the unbounded domain},
  author={Lin, Jitong and Chen, Yanping and Huang, Yunqing},
  journal={Journal of Applied Mathematics and Computing},
  volume={71},
  number={3},
  pages={4375--4396},
  year={2025},
  publisher={Springer}
}

@article{tissaoui2025efficient,
  title={Efficient spectral element method for the {E}uler equations on unbounded domains},
  author={Tissaoui, Yassine and Kelly, James F and Marras, Simone},
  journal={Applied Mathematics and Computation},
  volume={487},
  pages={129080},
  year={2025},
  publisher={Elsevier}
}

@article{kuhn2024explicit,
  title={Explicit time-domain analysis of wave propagation in unbounded domains using the scaled boundary finite element method},
  author={Kuhn, T and Gravenkamp, H and Birk, C},
  journal={Engineering Analysis with Boundary Elements},
  volume={168},
  pages={105891},
  year={2024},
  publisher={Elsevier}
}

@article{zhu2024highly,
  title={A highly efficient numerical method for the time-fractional diffusion equation on unbounded domains},
  author={Zhu, Hongyi and Xu, Chuanju},
  journal={Journal of Scientific Computing},
  volume={99},
  number={2},
  pages={47},
  year={2024},
  publisher={Springer}
}

@article{guo2023new,
  title={A new absorbing layer approach for solving the nonlinear {S}chr{\"o}dinger equation},
  author={Guo, Feng and Dai, Weizhong},
  journal={Applied Numerical Mathematics},
  volume={189},
  pages={88--106},
  year={2023},
  publisher={Elsevier}
}

@article{li2019numerical,
  title={Numerical solution of the regularized logarithmic {S}chr{\"o}dinger equation on unbounded domains},
  author={Li, Hongwei and Zhao, Xin and Hu, Yunxia},
  journal={Applied Numerical Mathematics},
  volume={140},
  pages={91--103},
  year={2019},
  publisher={Elsevier}
}

@article{yang2026adaptive,
  title={Adaptive-Distribution Randomized Neural Networks for {PDE}s: A Low-Dimensional Distribution-Learning Framework},
  author={Yang, You and Wang, Fei},
  journal={arXiv preprint arXiv:2604.23999},
  year={2026}
}

@article{dang2024adaptive,
  title={Adaptive-Growth Randomized Neural Networks for {PDE}s: Algorithms and Numerical Analysis},
  author={Dang, Haoning and Wang, Fei and Jiang, Song},
  journal={arXiv preprint arXiv:2408.17225},
  year={2024}
}

@article{chen2022bridging,
  title={Bridging traditional and machine learning-based algorithms for solving {PDE}s: the random feature method},
  author={Chen, Jingrun and Chi, Xurong and E, Weinan and Yang, Zhouwang},
  journal={Journal of Machine Learning},
  volume={1},
  number={3},
  pages={268--298},
  year={2022}
}

@article{sun2026randomized,
  title={Randomized Neural Networks for Partial Differential Equation on Static and Evolving Surfaces},
  author={Sun, Jingbo and Wang, Fei},
  journal={arXiv preprint arXiv:2603.01689},
  year={2026}
}

@article{shang2025overlapping,
  title={Overlapping {S}chwarz preconditioners for randomized neural networks with domain decomposition},
  author={Shang, Yong and Heinlein, Alexander and Mishra, Siddhartha and Wang, Fei},
  journal={Computer Methods in Applied Mechanics and Engineering},
  volume={442},
  pages={118011},
  year={2025},
  publisher={Elsevier}
}

@article{li2025local,
  title={Local randomized neural networks with finite difference methods for interface problems},
  author={Li, Yunlong and Wang, Fei},
  journal={Journal of Computational Physics},
  volume={529},
  pages={113847},
  year={2025},
  publisher={Elsevier}
}

@article{shang2024randomized,
  title={Randomized neural networks with {P}etrov--{G}alerkin methods for solving linear elasticity and {N}avier--{S}tokes equations},
  author={Shang, Yong and Wang, Fei},
  journal={Journal of Engineering Mechanics},
  volume={150},
  number={4},
  pages={04024010},
  year={2024},
  publisher={American Society of Civil Engineers}
}

@article{dang2024local,
  title={Local randomized neural networks with hybridized discontinuous {P}etrov--{G}alerkin methods for {S}tokes--{D}arcy flows},
  author={Dang, Haoning and Wang, Fei},
  journal={Physics of Fluids},
  volume={36},
  number={8},
  pages={087138},
  year={2024},
  publisher={AIP Publishing}
}

\end{document}